\pdfoutput=1
\documentclass[11pt]{article}
\usepackage[utf8]{inputenc}
\usepackage{authblk}
\usepackage{color}
\usepackage{helvet}         % selects\textbf{\textbf{•}} Helvetica as sans-serif font
\usepackage{courier}        % selects Courier as typewriter font
\usepackage{makeidx}         % allows index generation
\usepackage{graphicx}        % standard LaTeX graphics tool
\usepackage{multicol}        % used for the two-column index
\usepackage{amsmath}
\usepackage{amssymb}
\usepackage{bbm}
\usepackage{amsthm}
\usepackage{hyperref}
\usepackage{enumitem}

\newtheorem{theorem}{Theorem}
\newtheorem{corollary}[theorem]{Corollary}

\newtheorem{lemma}[theorem]{Lemma}

\newtheorem{proposition}[theorem]{Proposition}

\newtheorem{definition}[theorem]{Definition}
\newtheorem{propA}{Proposition}

\newtheorem{corA}{Corollary}

\usepackage[top=2cm, bottom=2cm, left=2cm, right=2cm]{geometry}
\numberwithin{theorem}{section}
\numberwithin{figure}{section}
\numberwithin{equation}{section}
\DeclareMathOperator{\CR}{CR}

\DeclareMathOperator{\dist}{dist}
\DeclareMathOperator{\SLE}{SLE}

\DeclareMathOperator{\CLE}{CLE}

\DeclareMathOperator{\hcap}{hcap}
\DeclareMathOperator{\diam}{diam}

\DeclareMathOperator{\Rep}{Re}
\DeclareMathOperator{\Atom}{Atom}
\DeclareMathOperator{\Range}{Range}
\DeclareMathOperator{\Card}{Card}
\DeclareMathOperator{\Bij}{Bij}
\DeclareMathOperator{\osc}{osc}

\renewcommand{\Im}{\operatorname{Im}} %better than the default

\begin{document}

\title{A level line of the Gaussian free field with \\measure-valued boundary conditions}

\author{Titus Lupu\thanks{Email: titus.lupu@upmc.fr. 
Funded by the French National Research Agency (ANR) grant
within the project MALIN (ANR-16-CE93-0003).}}
\affil{CNRS and Sorbonne Université, LPSM, Paris, France}
\author{Hao Wu\thanks{Email: hao.wu.proba@gmail.com. Funded by Beijing Natural Science Foundation (JQ20001, Z180003). }}
\affil{Tsinghua University, Beijing, China}

\date{}
\maketitle
\vspace{-1cm}
\begin{center}
\begin{minipage}{0.8\textwidth}
\abstract{In this article, we construct samples of SLE-like curves out of samples of CLE and Poisson point process of Brownian excursions. We show that the law of these curves depends continuously on the intensity measure of the Brownian excursions. Using such construction of curves, we extend the notion of level lines of GFF to the case when the boundary condition is measure-valued.
\smallbreak
\noindent\textbf{Keywords}: Schramm Loewner evolution (SLE), conformal loop ensemble (CLE), Brownian excursion, Gaussian free field (GFF). }
\end{minipage}
\end{center}

\tableofcontents

\newcommand{\eps}{\epsilon}
\newcommand{\ov}{\overline}
\newcommand{\U}{\mathbb{U}}
\newcommand{\T}{\mathbb{T}}
\newcommand{\HH}{\mathbb{H}}
\newcommand{\TK}{\mathtt{K}}
\newcommand{\LA}{\mathcal{A}}
\newcommand{\LB}{\mathcal{B}}
\newcommand{\LC}{\mathcal{C}}
\newcommand{\LD}{\mathcal{D}}
\newcommand{\LF}{\mathcal{F}}
\newcommand{\LK}{\mathcal{K}}
\newcommand{\LE}{\mathcal{E}}
\newcommand{\LG}{\mathcal{G}}
\newcommand{\LL}{\mathcal{L}}
\newcommand{\LM}{\mathcal{M}}
\newcommand{\LQ}{\mathcal{Q}}
\newcommand{\LO}{\mathcal{O}}
\newcommand{\CP}{\mathcal{P}}
\newcommand{\LR}{\mathcal{R}}
\newcommand{\LT}{\mathcal{T}}
\newcommand{\LS}{\mathcal{S}}
\newcommand{\LU}{\mathcal{U}}
\newcommand{\LV}{\mathcal{V}}
\newcommand{\PartF}{\mathcal{Z}}
\newcommand{\LH}{\mathcal{H}}
\newcommand{\FC}{\mathfrak{C}}
\newcommand{\R}{\mathbb{R}}
\newcommand{\C}{\mathbb{C}}
\newcommand{\N}{\mathbb{N}}
\newcommand{\Z}{\mathbb{Z}}
\newcommand{\E}{\mathbb{E}}
\newcommand{\D}{\mathbb{D}}
\newcommand{\PP}{\mathbb{P}}
\newcommand{\QQ}{\mathbb{Q}}
\newcommand{\A}{\mathbb{A}}
\newcommand{\SZ}{\mathsf{Z}}
\newcommand{\SU}{\mathsf{U}}
\newcommand{\one}{\mathbbm{1}}
\newcommand{\bn}{\mathbf{n}}
\newcommand{\MR}{MR}
\newcommand{\cond}{\,|\,}
\newcommand{\la}{\langle}
\newcommand{\ra}{\rangle}
\newcommand{\tree}{\Upsilon}
\global\long\def\chamber{\mathfrak{X}}

\section{Introduction}
\label{Sec Intro}
The goal of the present paper is to construct samples of variants of Schramm Loewner Evolution (SLE) curves out of samples of Conformal Loop Ensemble (CLE) and Poisson point process of Brownian excursions.
The intensity of Brownian excursions is parametrized
by a non-negative Radon measure $\nu$ supported on a boundary arc.
Our construction generalizes that of \cite{WernerWuCLEtoSLE},
where instead of the measure $\nu$ one had a constant.
We further show the continuity in law of the curve with respect to the
boundary measure $\nu$.
For the parameter $\kappa=4$ of the CLE,
we show that the curve we construct is actually distributed as a level line of a Gaussian free field (GFF) with boundary condition $\nu$.
This is related to random walk/Brownian motion representations of the GFF, known as isomorphism theorems,
and in particular relies on a construction in
\cite{AruLupuSepulvedaFPSGFFCVGISO}.
In \cite{PowellWuLevellinesGFF}, the authors constructed the level lines of a GFF with regulated boundary conditions and satisfying an additional condition related to the continuation threshold.
The construction of our paper allows both to go beyond the regulated
setting and to drop the continuation threshold condition.

In order to state our main conclusions properly, let us first introduce the CLE and the Brownian excursions. 
We denote by $\D$ the unit disc and by $\HH$ the upper half-plane: 
\begin{equation*}\label{eqn::def_unitdisc_halfplane}
\D:=\{z\in\C: |z|<1\}, \qquad \HH:=\{z\in\C: \Im(z)>0\}.
\end{equation*}

We first introduce the CLE. 
Consider probability measures on collections of countably many continuous simple loops in simply connected domains such that these loops are two-by-two disjoint and non-surrounding. 
In~\cite{SheffieldWernerCLE}, the authors prove that there exists a one-parameter family of such probability measures satisfying conformal invariance, a certain domain Markov property and an extra regularity assumption ``local finiteness". 
This family is denoted by 
$\CLE_{\kappa}$ with $\kappa\in (8/3,4]$. 
In a $\CLE_{\kappa}$ sample, the loops all locally look like 
$\SLE_{\kappa}$ curves. 

We next introduce the Brownian excursion measure. In this paper, ``Brownian" will refer to the standard Brownian motion in $\C$. 
Given $x, y\in\partial\D$, denote by $\mu_{x,y}^{\D}$ the non-normalized measure on Brownian excursions from $x$ to $y$ in $\D$ (see Section~\ref{Subsec Excursions}). Let $\sigma_{\partial \D}$ denote the arc-length measure on
$\partial\D$, with
$\sigma_{\partial \D}(\partial \D)=2\pi$. 
Let $A_{\rm L}$ and $A_{\rm R}$ denote the left and right half-circles 
in $\partial \D$:
\begin{equation}\label{eqn::def_halfcircle}
A_{\rm L}=\{z\in\partial\D : \Rep (z)<0\},
\qquad
A_{\rm R}=\{z\in\partial\D : \Rep (z)>0\}.
\end{equation}
Let $\nu$ be a 
finite non-negative Radon measure on 
$\overline{A_{\rm L}}$. 
The Brownian excursion measure $\mu^{\D}_{\nu}$ is defined as follows:
\begin{equation}\label{eqn::Brownianexcursionmeasurenu}
\mu^{\D}_{\nu}(\cdot)=
\dfrac{1}{2}
\iint_{\overline{A_{\rm L}}\times \overline{A_{\rm L}}}
d(\nu\otimes\nu)(x,y)\mu_{x,y}^{\D}(\cdot). 
\end{equation}

Now, we are ready to state our construction. 
Fix $\kappa\in (8/3 , 4]$ and 
let $\FC_{\kappa}$ be a $\CLE_{\kappa}$ loop ensemble in $\D$.
Fix $\nu$ a finite non-negative Radon measure on 
$\overline{A_{\rm L}}$ and let $\Xi_{\nu}$ be a Poisson point process of intensity $\mu^{\D}_{\nu}$, independent of $\FC_{\kappa}$.
For $\gamma\in \Xi_{\nu}$, let $\LS_{\kappa}(\gamma)$ be the union of $\gamma$ and all loops in $\FC_{\kappa}$ intersecting $\gamma$. Let $\LS_{\kappa,\nu}$ be the union of all $\LS_{\kappa}(\gamma)$ with $\gamma\in \Xi_{\nu}$. 
In the limit case $\kappa=8/3$, we just set $\LS_{8/3,\nu}$ to be the union of all $\gamma\in \Xi_{\nu}$. 
By construction, 
$\LS_{\kappa,\nu}\cap A_{\rm R} = \emptyset$.
Let $\LD_{\rm R,\kappa,\nu}$ be the connected component of
$\overline{\D}\setminus
(\LS_{\kappa,\nu}\cup\overline{A_{\rm L}})$ 
that contains $A_{\rm R}$.
The set $\LD_{\rm R,\kappa,\nu}$ is of form
$\LO_{\kappa,\nu}\cup A_{\rm R}$,
where $\LO_{\kappa,\nu}$ is an open simply connected domain. Let $\psi_{\kappa,\nu}$ be the conformal transformation from $\D$ to $\LO_{\kappa, \nu}$, uniquely defined by the normalization 
\begin{equation}\label{eqn::psikappanu_def}
\psi_{\kappa, \nu}(-i)=-i,\quad \psi_{\kappa, \nu}(1)=1, \quad \psi_{\kappa, \nu}(i)=i\quad\text{and }\psi_{\kappa, \nu}(A_{\rm R})=A_{\rm R}.
\end{equation}
Set
\begin{equation}
\label{Eq eta kappa nu}
\eta_{\kappa,\nu}:=
(\partial \LD_{\rm R,\kappa,\nu})\setminus A_{\rm R}.
\end{equation}
Informally, $\eta_{\kappa,\nu}$ is constructed as the
envelop from the right of the set 
$\LS_{\kappa,\nu}\cup \overline{A_{\rm L}}$. See Figure~\ref{fig::envelopconstruction}. 

\begin{figure}[ht!]
\begin{center}
\includegraphics[width=0.9\textwidth]{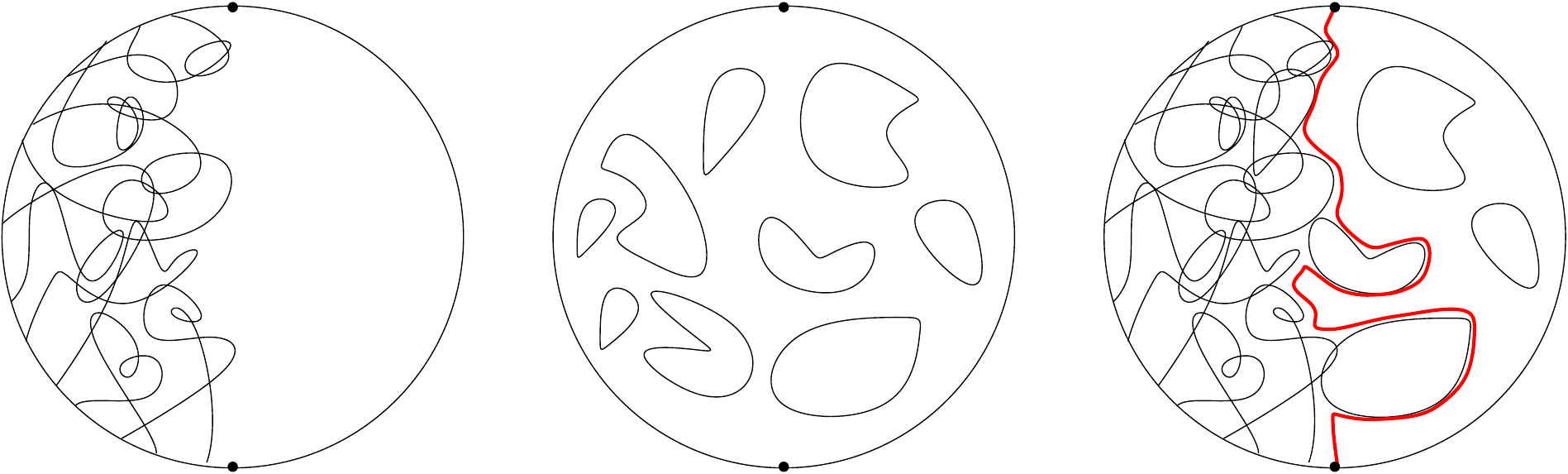}
\end{center}
\caption{\label{fig::envelopconstruction} In the left panel, the curves indicate a Poisson point process of Brownian excursions with intensity $\mu^{\D}_{\nu}$. In the middle panel, the loops indicate a $\CLE_{\kappa}$ loop ensemble. In the right panel, the curve in red indicates $\eta_{\kappa,\nu}$. }
\end{figure}

\subsection{Continuity of the envelop}
\label{subsec::intro_envelop}
The first goal of this paper is to derive continuity properties of $\eta_{\kappa, \nu}$. 

\begin{proposition}\label{prop::envelop_parameterization}
Fix $\kappa\in [8/3,4]$ and $\nu$ a finite non-negative Radon measure on $\overline{A_{\rm{L}}}$. 
\begin{enumerate}[label=(\arabic*)]
\item The law of $\eta_{\kappa,\nu}$ is conformal covariant in the following sense: given $\psi$ a Möbius transformation of $\D$,
with $\psi(-i)=-i$ and $\psi(i)=i$, the set 
$\psi(\eta_{\kappa,\nu})$ is distributed as
$\eta_{\kappa,\nu_{\psi}}$,
where
\[d\nu_{\psi}(x) = \vert\psi'\circ\psi^{-1}(x)\vert d(\psi_{\ast}\nu)(x).\]
\item \label{item::envelop_continuity} The set 
$\C\setminus \LO_{\kappa,\nu}$ is locally connected. The conformal map $\psi_{\kappa, \nu}$ extends continuously to $\overline{\D}$ and $(\psi_{\kappa,\nu}(x))_{x\in\overline{A_{\rm L}}}$ parameterizes $\eta_{\kappa,\nu}$ as a continuous curve in $\overline{\D}$ from $-i$ to $i$. 
\end{enumerate}
\end{proposition}

\begin{theorem}
\label{thm::cvg_envelop}
Fix $\kappa\in [8/3,4]$ and $\nu$ a finite non-negative Radon measure
on $\overline{A_{\rm L}}$.
Let $(\nu_{n})_{n\geq 0}$ be a sequence of finite non-negative Radon measures on $\overline{A_{\rm L}}$, 
converging weakly to $\nu$.
Then the sequence of curves
$\left((\psi_{\kappa,\nu_{n}}(x))_{x\in \overline{A_{\rm L}}}\right)_{n \geq 0}$
converges in law to
$(\psi_{\kappa,\nu}(x))_{x\in \overline{A_{\rm L}}}$
for the uniform topology.
\end{theorem}

Next, we will consider $\eta_{\kappa, \nu}$ as a Loewner chain. To this end, it is more convenient to work in $\HH$. Let $\psi_{0}$ be the Möbius transformation from
$\D$ to $\HH$ with $\psi_{0}(0)=i$,
$\psi_{0}(-i)=0$, and hence
$\psi_{0}(i)= \infty$:
\begin{equation}
\label{eqn::psi0}
\psi_{0}(z)=-i\dfrac{z + i}{z-i}.
\end{equation} 
Define 
\[\tilde{\eta}_{\kappa,\nu}:=\psi_0(\eta_{\kappa,\nu}).\]

\begin{proposition}\label{prop::envelop_drivingfunction}
Fix $\kappa\in [8/3,4]$ and $\nu$ a finite non-negative Radon measure
on $\overline{A_{\rm L}}$. Suppose the support of $\nu$ equals $\overline{A_{\rm L}}$. Then we can parameterize $\tilde{\eta}_{\kappa,\nu}$ (up to its first hitting time of $\infty$) by the half-plane capacity $(\tilde{\eta}_{\kappa,\nu}(t))_{0\leq t<T_{\rm max}}$,
with $T_{\rm max}\in (0,+\infty]$,
such that
\begin{equation}\label{eqn::envelop_hcap}
\tilde{\eta}_{\kappa,\nu}(0)=0,
\qquad
\lim_{t\to T_{\rm max}}\tilde{\eta}_{\kappa,\nu}(t)
= \infty; 
\qquad
\hcap(\tilde{\eta}_{\kappa,\nu}([0,t])) = 2t, \quad \forall t \in [0,T_{\rm max}).
\end{equation}
When parameterized by the half-plane capacity, $\tilde{\eta}_{\kappa,\nu}$ is a continuous curve with continuous driving function. 
\end{proposition}

\begin{proposition}\label{prop::cvg_envelop_drivingfunction}
Fix $\kappa\in [8/3,4]$ and $\nu$ a finite non-negative Radon measure
on $\overline{A_{\rm L}}$.
Let $(\nu_{n})_{n\geq 0}$ be a sequence of finite non-negative Radon measures on $\overline{A_{\rm L}}$, 
converging weakly to $\nu$.
Suppose the supports of $\nu$ and of $\nu_n$ equal $\overline{A_{\rm L}}$. 
When parameterized by the half-plane capacity, $\tilde{\eta}_{\kappa, \nu_n}$ converges in law to $\tilde{\eta}_{\kappa,\nu}$ and the driving function of $\tilde{\eta}_{\kappa, \nu_n}$ converges in law to the driving function of $\tilde{\eta}_{\kappa,\nu}$ for the local uniform topology.  
\end{proposition}

We will complete the proof of Propositions~\ref{prop::envelop_parameterization} and~\ref{prop::envelop_drivingfunction} in Section~\ref{Sec Construction}. 
We will complete the proof of Theorem~\ref{thm::cvg_envelop} and Proposition~\ref{prop::cvg_envelop_drivingfunction} in Section~\ref{Sec Continuity}. 
For the proof of Theorem~\ref{thm::cvg_envelop}
we rely on a strong coupling between the Poisson point processes
with intensity $\mu^{\D}_{\nu}$,
respectively $\mu^{\D}_{\nu_{n}}$
(see Proposition~\ref{Prop Coupling Exc})
and on the notion of uniform local connectedness
(see Definition \ref{Def Loc Connec}).

In~\cite{WernerWuCLEtoSLE}, the authors construct the same process $\eta_{\kappa,\nu}$ as in our construction except that they focus on the case when $\nu$ is a constant $a>0$ times the arc-length measure $\sigma_{\partial\D}$ restricted to $\overline{A_{\rm L}}$. In such case, they prove that $\eta_{\kappa,\nu}$ has the same law as $\SLE_{\kappa}(\rho)$ process where $\rho>-2$ is uniquely determined by $\kappa$ and $a>0$, see Theorem~\ref{Thm Werner Wu}. 
Readers may wonder whether the law of $\eta_{\kappa, \nu}$ for general $\nu$ is the same as $\SLE_{\kappa}(\rho)$ process with multiple force points. We will show that $\eta_{\kappa, \nu}$ is absolutely continuous with respect to $\SLE_{\kappa}$ process away from the boundary, see Proposition~\ref{Prop abs cont}. However, $\eta_{\kappa, \nu}$ is not in the family of $\SLE_{\kappa}(\rho)$ processes with multiple force points in general. 
Here is a preciser answer. 
\begin{itemize}
\item The answer is negative for $\kappa\in [8/3,4)$.
By construction, 
$\eta_{\kappa, \nu}$ enjoys ``reversibility": the time-reversal of 
$\eta_{\kappa, \nu}$ has the same law as 
$\eta_{\kappa, \bar{\nu}}$,
where $\bar{\nu}$ is the image of $\nu$ by the reflection
$z\mapsto\bar{z}$.
However, 
$\SLE_{\kappa}(\rho)$ with multiple force points does not have such reversibility in general, see discussion in~\cite{DubedatCommutationSLE, MillerSheffieldIG2}. 
\item The answer is positive for $\kappa=4$ and this is related to the second goal of this paper, see Section~\ref{Subsec equation kappa 4}. 
\end{itemize}

%Let us discuss possible applications of our continuity results. 
%As $\eta_{\kappa, \nu}$ is the same as $\SLE_{\kappa}(\rho)$ process when $\nu$ is a constant times the arc-length measure, our continuity results imply that $\SLE_{\kappa}(\rho)$ process is continuous in $(\kappa, \rho)$: Suppose $(\kappa_n, \rho_n)\in [8/3,4]\times (-2,\infty)$ converges to $(\kappa, \rho)\in [8/3,4]\times (-2,\infty)$, then $\SLE_{\kappa_n}(\rho_n)$ process converges in law to $\SLE_{\kappa}(\rho)$ for the local uniform topology. \footnote{Although our conclusion does not involve the continuity result on $\kappa$, the proof can be easily generalize to show the continuity in $\kappa$. }
%Such continuity result was also proved in~\cite[Theorem~1.10]{KemppainenSmirnovRandomCurves}.
%{\color{blue} [Doesn't this already follow from the continuity in law of the driving function? Hao: see~\cite[Theorem~1.10]{KemppainenSmirnovRandomCurves}, this is not immediate. ]} Another application is to couple the curve $\eta_{4,\nu}$ with Gaussian free field and to generalize the notion of level lines. This is the second goal of this article.

\subsection{Identification of the envelop when $\kappa=4$}

The (zero-boundary) Gaussian free field (GFF) in the unit disc $\D$ is a centered Gaussian process $\Phi$ indexed by the set of smooth functions with compact support in $\D$ where the covariance is given by the Green's function (see Section~\ref{Subsec Excursions}):
\[\E[(\Phi, f)(\Phi, g)]=\iint_{\D\times \D}f(z)G_{\D}(z,w)g(w)dzdw.\]

Suppose $\nu$ is a finite Radon measure on $\partial \D$.
We will again denote by $\nu$ the harmonic extension
of $\nu$ in $\D$:
\[\nu(z):=\int_{\partial D}H_{\D}(z, x)d\nu(x), \quad \forall z\in \D,\]
where $H_{\D}(z,x)$ is the Poisson kernel 
(see Section~\ref{Subsec Excursions}). 
Note also that any non-negative harmonic function on
$\D$ is a harmonic extension of a finite 
Radon measure on $\partial \D$;
see Appendix~\ref{appendix}.
The GFF in $\D$ with boundary condition $\nu$ is $\Phi+\nu$ where 
$\Phi$ is a zero-boundary GFF. The definition for GFF in a general simply connected domain can be passed by conformal invariance. 

Next, we introduce level lines of GFF. Fix $\lambda=\sqrt{\pi/8}$. Suppose $(\eta(t))_{t\ge 0}$ is a continuous simple curve in $\overline{\D}$ from $-i$ to $i$ and $\psi_0(\eta)$ has continuous driving function.  Define $\nu_t$ to be the harmonic extension of the following boundary data on $\D\setminus \eta[0,t]$: it is $2\lambda$ on the left side of $\eta[0,t]$ and is $0$ on the right side of $\eta[0,t]$, and it coincides with $\nu$ on $\partial\D$. That is 
\begin{equation}\label{eqn::def_boundarydata_Markov}
\nu_t(z)=
\int_{\partial \D}H_{\D\setminus\eta[0,t]}(z,x)d\nu(x)
+2\lambda H_{\D\setminus\eta[0,t]}(z, \LL(\eta[0,t])), \quad z\in \D\setminus\eta[0,t],
\end{equation}
where $\LL(\eta[0,t])$ denotes the left side of $\eta[0,t]$. 

\begin{definition}\label{def::GFFlevelline}
Suppose $\Phi$ is zero-boundary GFF in $\D$ and $\nu$ is a finite non-negative Radon measure on $A_{\rm L}$. 
Suppose $\eta$ is a continuous simple curve in $\overline{\D}$ from $-i$ to $i$.  
We say that $\eta$ is a level line of $\Phi+\nu$ if there exists a coupling $(\Phi, \eta)$ such that the following domain Markov property holds: for any finite $\eta$-stopping time $\tau$, the conditional law of $(\Phi+\nu)|_{\D\setminus \eta[0,\tau]}$ given $\eta[0,\tau]$ is equal to the law of GFF in $\D\setminus \eta[0,\tau]$ with boundary condition $\nu_{\tau}$ as defined in~\eqref{eqn::def_boundarydata_Markov}.   
\end{definition}
The definition in simply connected domains is given via conformal image. 

The notion of level lines of GFF originally appears in~\cite{DubedatSLEFreefield, SchrammSheffieldDiscreteGFF, SchrammSheffieldContinuumGFF}. 
In particular, the authors of~\cite{SchrammSheffieldContinuumGFF} prove that, in $\D$, the coupling exists when $\nu=2\lambda\one_{A_{\rm L}}\sigma_{\partial\D}$, and the law of the level line is an $\SLE_4$ in $\D$ from $-i$ to $+i$. 
In~\cite{WangWuLevellinesGFFI}, the authors give a survey on level lines of GFF when the boundary condition is piecewise constant; later, in~\cite{PowellWuLevellinesGFF}, the authors construct level lines of GFF when the boundary condition is regulated. 
In this article, we provide a more general coupling. 

First, we recall the result of 
\cite{AruLupuSepulvedaFPSGFFCVGISO} 
which relates some level lines of the GFF to envelops $\eta_{4,\nu}$ in the case of $\nu$ being a piecewise 
constant function; see \cite[Proposition~5.11]{AruLupuSepulvedaFPSGFFCVGISO}.

\begin{theorem}[Aru-Lupu-Sep\'ulveda 
\cite{AruLupuSepulvedaFPSGFFCVGISO}]
\label{thm::ALS}
Fix $\kappa=4$.
Let $u$ be a strictly positive piecewise constant 
function on $A_{\rm L}$ assuming finitely many values.
Denote 
$\eta_{4,u} := \eta_{4,u \one_{A_{\rm L}}\sigma_{\partial\D}}$. 
In this case, there exists a coupling $(\Phi, \eta_{4,u})$ such that 
$\eta_{4,u}$ is a level line of $\Phi+u$.  
\end{theorem}

We extend the result of \cite{AruLupuSepulvedaFPSGFFCVGISO}
beyond the piecewise constant case. Suppose $\nu$ is a finite non-negative Radon measure on 
$\overline{A_{\rm L}}$. 
Denote by $\Atom(\nu)$ the set of atoms of $\nu$,
if any.
Denote
\begin{equation}\label{eqn::atom_conv_LR}
\begin{split}
\Atom_{\rm conv}^{\rm l}(\nu) :=
\Big\{x\in \Atom(\nu) : 
\int_{\substack{y\in\overline{A_{\rm L}}
\\ \Im(y)>\Im(x)}}
\dfrac{1}
{\vert y-x\vert^{2}} d\nu(y)<+\infty
\Big\},
\\
\Atom_{\rm conv}^{\rm r}(\nu) :=
\Big\{x\in \Atom(\nu) : 
\int_{\substack{y\in\overline{A_{\rm L}}
\\ \Im(y)<\Im(x)}}
\dfrac{1}
{\vert y-x\vert^{2}} d\nu(y)<+\infty
\Big\}.
\end{split}
\end{equation}
Set 
\begin{equation}\label{eqn::atom_conv_star}
\Atom_{\rm conv}^{\ast}(\nu) :=
\Atom_{\rm conv}^{\rm l}(\nu)\cup
\Atom_{\rm conv}^{\rm r}(\nu).
\end{equation}

\begin{theorem}\label{thm::gfflevelline_coupling}
Fix $\kappa=4$ and $\nu$ a finite non-negative Radon measure on 
$\overline{A_{\rm L}}$ such that the support of $\nu$ equals $\overline{A_{\rm L}}$ and $\Atom_{\rm conv}^{\ast}(\nu)=\emptyset$.
Then we have the followings.
\begin{itemize}
\item The curve $\eta_{4,\nu}$ is a continuous simple curve with continuous driving function.
\item Suppose $\Phi$ is zero-boundary GFF in $\D$. 
There exists a coupling $(\Phi, \eta_{4,\nu})$ such that $\eta_{4,\nu}$ is a level line of $\Phi+\nu$.  
\end{itemize} 
\end{theorem}

The condition $\Atom_{\rm conv}^{\ast}(\nu)=\emptyset$ above is only
to ensure that $\eta_{4,\nu}$ a.s. does not hit an atom of $\nu$.
See Lemma \ref{Lem hit}.

\begin{theorem}\label{thm::determination}
If $(\Phi, \eta_{4,\nu})$ are coupled as in Theorem~\ref{thm::gfflevelline_coupling}, then $\eta_{4,\nu}$ is almost surely determined by $\Phi$. 
\end{theorem}

We will complete the proof of 
Theorems~\ref{thm::gfflevelline_coupling} and~\ref{thm::determination} in Section~\ref{Sec Level lines}. 
These two theorems are generalizations of~\cite[Theorems~1.2 and~1.3]{PowellWuLevellinesGFF} where the authors prove the same conclusion under the assumption that $\nu$ is a regulated function and is bounded away from $0$, i.e. the assumption~\eqref{assumptionPW}. Let us briefly summarize the proof in~\cite{PowellWuLevellinesGFF}. The idea is to approximate uniformly the regulated function by piecewise constant functions and then show that the level lines corresponding to piecewise constant boundary functions are convergent. For that one shows that the limit of the level lines gives the desired coupling. However, the limiting process is not automatically a continuous curve with continuous driving function. In order to guarantee that the limiting process can be encoded as a Loewner chain with continuous driving function, the authors use the conclusions from~\cite{KemppainenSmirnovRandomCurves}. 
The technical assumption from~\cite{KemppainenSmirnovRandomCurves},
related to the probabilities of crossings of quadrilaterals,
restricts the method in~\cite{PowellWuLevellinesGFF} and the authors are only able to show the above conclusion under the assumption~\eqref{assumptionPW}.  
We will see in Section~\ref{Sec Level lines} that our construction improves the conclusion from~\cite{PowellWuLevellinesGFF}. 
The continuity result of Theorem~\ref{thm::cvg_envelop}, which replaces the technical tool from~\cite{KemppainenSmirnovRandomCurves}, 
allows to replace the uniform convergence of the boundary data
by the weak convergence of measures,
and does not require the assumption~\eqref{assumptionPW}.
The above method works for any finite non-negative Radon measure $\nu$ whose support equals $\overline{A_{\rm L}}$ and 
such that $\Atom_{\rm conv}^{\ast}(\nu)=\emptyset$.

We end the introduction with a few interesting open questions. 
Our construction of level lines of thr GFF is a significant generalization of previous constructions, but there is still an interesting scenario that we do not understand. For instance, we do not know what happens to the coupling between the GFF and 
$\eta_{4,\nu}$ when the curve hits an atom of $\nu$ with positive probability, see open questions in Section~\ref{Sec questions}. 

Let us come back to the question at the end of Section~\ref{subsec::intro_envelop}. 
Fix $\kappa=4$ and $\nu$ a finite non-negative Radon measure on 
$\overline{A_{\rm L}}$ as in Theorem~\ref{thm::gfflevelline_coupling}. 
We parameterize the curve $\eta_{4,\nu}$ by the half-plane capacity. 
Then $\eta_{4,\nu}$ can be encoded as a generalization of $\SLE_4(\rho)$ process on the time interval when it is away from the boundary, see Section~\ref{Subsec equation kappa 4}. However, as it is possible for $\eta_{4,\nu}$ to intersect the boundary with a positive Lebesgue measure (see Section~\ref{subsec::hittingboundary}), we do not know what happens to the driving function when the curve hits the boundary, see open questions in Section~\ref{Sec questions}.

\section{Preliminaries}
\label{Sec Prelim}
\subsection{Local connectedness and cut points}
\label{Subsec Loc Connec}

We will introduce the notions of 
\textit{local connectedness} and
\textit{uniform local connectedness}, and cite and derive a few elementary properties which will be useful later. We refer interested readers to~\cite[Section~2.2]{Pommerenke} for more detail.

\begin{definition}
\label{Def Loc Connec}
Given $C$ a closed non-empty subset of $\C$,
and $z,z'\in C$, we say that
$z$ and $z'$ are $\varepsilon$-\textit{connected in} $C$ if
there is $K$ a compact connected subset of $C$
with $\diam(K)<\varepsilon$ such that $z,z'\in K$.

A closed non-empty subset $C\subset\C$ is \textit{locally connected}
if for every $\varepsilon>0$, there is $\delta>0$ such that for every $z,z'\in C$ with $\vert z'-z\vert<\delta$,
the points $z$ and $z'$ are $\varepsilon$-connected in $C$.

A family of closed non-empty subsets $(C_{n})_{n\geq 0}$ of $\C$
is \textit{uniformly locally connected}
if for every $\varepsilon>0$, there is $\delta>0$ such that for every
$n\geq 0$ and every $z,z'\in C_{n}$ with $\vert z'-z\vert<\delta$,
the points $z$ and $z'$ are $\varepsilon$-connected in $C_{n}$.
\end{definition}

\begin{lemma}
\label{Lem Loc Con Elem}
\begin{enumerate}[label=(\arabic*)]
\item \label{item::localcon1} If $\gamma : [0,1]\rightarrow\C$ is a continuous parametrized curve,
then $\Range(\gamma)$ is locally connected.
\item \label{item::localcon2} If $K_{1}$ and $K_{2}$ are two locally connected compact subsets of $\C$, then $K_{1}\cup K_{2}$ is locally connected.
\end{enumerate}
\end{lemma}

\begin{proof}
For~\ref{item::localcon1}, 
$\Range(\gamma)$ is locally connected since it is
the image of the compact locally connected set $[0,1]$ by a continuous map; see \cite[Theorem~8.2, Chapter~IV]{NewmanTopology} 
and \cite[Section~2.2]{Pommerenke}.
For~\ref{item::localcon2}, see
\cite[Theorem~8.1, Chapter~IV]{NewmanTopology} 
and the subsequent corollary,
and \cite[Section~2.2]{Pommerenke}.
\end{proof}

\begin{lemma}
\label{Lem Count Union Loc Con}
Let $(K_{n})_{n\geq 0}$ be a sequence of non-empty compact subsets of 
$\C$. Assume that the following four conditions hold:
\begin{enumerate}[label=(\alph*)]
\item \label{item::countunionlocalcon1} For each $n\geq 0$, $K_{n}$ is locally connected.
\item \label{item::countunionlocalcon2} For each $n\geq 1$, $K_{n}$ is connected.
\item \label{item::countunionlocalcon3} For each $n\geq 1$, $K_{n}\cap K_{0}\neq\emptyset$.
\item \label{item::countunionlocalcon4} $\diam(K_{n})\to 0$ as $n\to +\infty$.
\end{enumerate}
Then the union $\bigcup_{n\geq 0} K_{n}$ is compact and locally connected.
\end{lemma}

\begin{proof}
The compactness of $\bigcup_{n\geq 0} K_{n}$ is ensured by the compactness of each $K_{n}$ and the conditions~\ref{item::countunionlocalcon3} and~\ref{item::countunionlocalcon4}.
It remains to check the local connectedness.

For $N\geq 1$, denote
\begin{displaymath}
\widetilde{K}_{N}:
=\bigcup_{n=0}^{N-1} K_{n} .
\end{displaymath}
The condition~\ref{item::countunionlocalcon1} and Lemma~\ref{Lem Loc Con Elem}~\ref{item::localcon2} ensure that the compact sets $\widetilde{K}_{N}$ are locally connected.
Fix $\varepsilon>0$. 
Since $K_{0}$ is locally connected (the condition~\ref{item::countunionlocalcon1}),
there is $\delta_{0}>0$
such that for every $z,z'\in K_{0}$ with
$\vert z'-z\vert<\delta_{0}$,
$z$ and $z'$ are $\varepsilon/3$-connected in $K_{0}$.
The condition~\ref{item::countunionlocalcon4} ensures that there is $N_{0}\geq 1$
such that for every $n\geq N_{0}$, 
$\diam(K_{n})<(\varepsilon\wedge\delta_{0})/3$.
Further, there is $\delta_{1}>0$ such that
for every $z,z'\in \widetilde{K}_{N_{0}}$ with
$\vert z'-z\vert<\delta_{1}$,
$z$ and $z'$ are $\varepsilon/2$-connected in $\widetilde{K}_{N_{0}}$.
Then there is $N_{1}\geq N_{0}$ such that
for every $n\geq N_{1}$, 
$\diam(K_{n})<(\varepsilon\wedge\delta_{1})/2$.
Finally, there is $\delta_{2}>0$
such that for every $z,z'\in \widetilde{K}_{N_{1}}$ with
$\vert z'-z\vert<\delta_{2}$,
$z$ and $z'$ are $\varepsilon$-connected in 
$\widetilde{K}_{N_{1}}$.
Set
\begin{displaymath}
\delta : =
\dfrac{\delta_{0}}{3}
\wedge
\dfrac{\delta_{1}}{2}
\wedge
\delta_{2}
\wedge
\dfrac{\varepsilon}{3}.
\end{displaymath}

Take
$z,z'\in \bigcup_{j\geq 0} K_{j}$
with
$\vert z'-z\vert <\delta$.
Since $z$ and $z'$ play symmetric roles, 
there are three cases to consider:
\begin{itemize}
\item Case 1:
$z\in K_{n}$, $z'\in K_{n'}$,
with $n,n'\geq N_{0}$.
\item Case 2:
$z\in \widetilde{K}_{N_{0}}$
and $z'\in K_{n'}$
with $n'\leq N_{1}-1$.
\item Case 3:
$z\in \widetilde{K}_{N_{0}}$
and $z'\in K_{n'}$
with $n'\geq N_{1}$.
\end{itemize}

In Case 1, the condition~\ref{item::countunionlocalcon3} ensures that one can take points
$\tilde{z}\in K_{n}\cap K_{0}$ and
$\tilde{z}'\in K_{n'}\cap K_{0}$.
Then
\begin{displaymath}
\vert\tilde{z}'-\tilde{z}\vert < 
\vert z'-z\vert + \dfrac{2}{3}\delta_{0}< \delta_{0}.
\end{displaymath}
So there is $K$ a compact connected subset of 
$K_{0}$ with $\diam(K)<\varepsilon/3$
such that $\tilde{z},\tilde{z}'\in K$.
Then $K\cup K_{n} \cup K_{n'}$ is a compact subset of
$\bigcup_{j\geq 0} K_{j}$, 
by the condition~\ref{item::countunionlocalcon2}, it is connected,
it contains $z$ and $z'$, and
\begin{displaymath}
\diam(K\cup K_{n} \cup K_{n'})
\leq \diam(K)+\diam(K_{n})+\diam(K_{n'})<\varepsilon .
\end{displaymath}

In Case 2, we have $z,z'\in \widetilde{K}_{N_{1}}$
and $\vert z'-z\vert < \delta_{2}$.
So $z'$ and $z$ are $\varepsilon$-connected in 
$\widetilde{K}_{N_{1}}$, 
and thus in $\bigcup_{j\geq 0} K_{j}$.

In Case 3, consider $\tilde{z}'\in K_{n'}\cap K_{0}$.
Then 
\begin{displaymath}
\vert \tilde{z}' - z\vert
<\vert z' - z\vert + \dfrac{1}{2}\delta_{1} <\delta_{1} .
\end{displaymath}
Thus, there is $K$ a compact connected subset of 
$\widetilde{K}_{N_{0}}$
with $\diam(K)<\varepsilon/2$
such that $z,\tilde{z}'\in K$.
Then $K\cup K_{n'}$ is a 
compact connected subset of $\bigcup_{j\geq 0} K_{j}$
containing $z$ and $z'$,
and $\diam(K\cup K_{n'})<\varepsilon$.
\end{proof}

\begin{lemma}
\label{Lem Loc Con Boundary}
\begin{enumerate}[label=(\arabic*)]
\item \label{item::localconboundary1} Let $C$ and $\widetilde{C}$ be closed non-empty subsets of
$\C$, with $\widetilde{C}\subset C$.
Assume that $C\setminus \widetilde{C}$ is an open subset of $\C$.
Then, if $\widetilde{C}$ is locally connected, then so is $C$.
\item \label{item::localconboundary2} Let $(C_{n})_{n\geq 0}$ and $(\widetilde{C}_{n})_{n\geq 0}$
be two families of closed non-empty subsets of $\C$.
Assume that for every $n\geq 0$,
$\widetilde{C}_{n}\subset C_{n}$ and
$C_{n}\setminus \widetilde{C}_{n}$ is an open subset of $\C$.
Then, if the family $(\widetilde{C}_{n})_{n\geq 0}$
is uniformly locally connected, 
then so is the family $(C_{n})_{n\geq 0}$.
\end{enumerate}
\end{lemma}

\begin{proof}
Since~\ref{item::localconboundary2} clearly implies~\ref{item::localconboundary1}, it suffices to show~\ref{item::localconboundary2}.
Fix $\varepsilon>0$, $n\geq 0$ and
$z,z'\in C_{n}$
with $\vert z'-z\vert<\varepsilon/2$.
One can consider the straight line segment
$I_{z,z'}$ joining $z$ and $z'$.
If 
$I_{z,z'}\cap \widetilde{C}_{n}=\emptyset$,
then necessarily 
$I_{z,z'}\subset C_{n}\setminus \widetilde{C}_{n}
\subset C_{n}$,
because $C_{n}\setminus \widetilde{C}_{n}$ is open.
Otherwise, one can consider
$\tilde{z}$ the point of
$I_{z,z'}\cap \widetilde{C}_{n}$
which is the closest to $z$,
and
$\tilde{z}'$ the point of
$I_{z,z'}\cap \widetilde{C}_{n}$
which is the closest to $z'$.
By construction, 
$\vert\tilde{z}'-\tilde{z}\vert\leq \vert z'-z\vert$.
Let $I_{z,\tilde{z}}$, respectively
$I_{\tilde{z}',z'}$, be the subsegment joining
$z$ and $\tilde{z}$, respectively $z'$ and $\tilde{z}'$.
We have that 
$I_{z,\tilde{z}}\cup I_{\tilde{z}',z'}
\subset C_{n}$.
Thus, $z$ and $z'$ are $\varepsilon$-connected in
$C_{n}$ as soon as
$\tilde{z}$ and $\tilde{z}'$
are $\varepsilon/2$-connected in 
$\widetilde{C}_{n}$.
\end{proof}

Next, we will introduce the notion of
\textit{cut points}, and cite a few elementary properties which will be useful later.
We refer interested readers to~\cite[Section~2.3]{Pommerenke} for more detail.

\begin{definition}
\label{Def Cut Point}
Given $C$ a closed connected non-empty subset of $\C$,
a point $z\in C$ is said to be a
\textit{cut point} of $C$
if $C\setminus\{ z\}$ is not connected.
\end{definition}

Next lemma is standard and we state it without proof.

\begin{lemma}
\label{Lem disks}
Let $K$ be a compact connected subset of the Riemann sphere
$\widehat{\C}:=\C\cup\{\infty\}\cong\mathbb{S}^{2}$,
such that $K$ and $\widehat{\C}\setminus K$ are both non-empty.
Then for every $\LO$ connected component of
$\widehat{\C}\setminus K$, 
$\LO$ is simply connected, and in particular there is a conformal
transformation from $\D$ to $\LO$.
The boundary $\partial\LO$ is connected.
\end{lemma}

\begin{lemma}
\label{Lem K cuts}
Let $K$ be a compact connected non-empty subset of $\C$.
Assume that $K$ has no cut points.
Then for every $\LO$ connected component of $\C\setminus K$, 
$\partial\LO$ has no cut points.
\end{lemma}

\begin{proof}
Since we can always consider the Riemann sphere
$\widehat{\C}=\C\cup\{\infty\}$, we assume without loss of generality
that $\LO$ is bounded.
Assume that $z$ is a cut point of $\partial\LO$.
It follows from the proof of~\cite[Proposition~2.5]{Pommerenke} that
there are two points $z_{1},z_{2}\in\partial\LO$
that are in two distinct connected components of
$(\C\setminus\LO)\setminus\{z\}$.
Since $K\subset \C\setminus\LO$,
the points $z_{1}$ and $z_{2}$ are also in two distinct connected components of
$K\setminus\{z\}$.
This contradicts the assumptions.
\end{proof}

Next we recall Carathéodory's theorem on the extension of conformal maps to the boundary; see~\cite[Theorem~2.1, Theorem~2.6, Corollary~2.8]{Pommerenke}.

\begin{theorem}
\label{Thm Cara boundary}
Let $D$ be an open bounded simply connected domain in $\C$.
Let $\psi$ be a conformal map from $\D$ to $D$.
\begin{enumerate}[label=(\arabic*)]
\item If $\C\setminus D$ is locally connected, then
$\psi$ extends continuously to $\overline{\D}$.
In particular $\partial D$ can be parametrized as a continuous closed curve.
\item If on top of that, $\partial D$ has no cut points, 
then $\partial D$ is a Jordan curve, 
i.e. continuous closed simple curve,
and $\psi$ extends to a homeomorphism
from $\overline{\D}$ to $\overline{D}$.
\end{enumerate}
\end{theorem}

Next we recall the notion of \textit{Carathéodory convergence}.
See \cite[Section~1.4]{Pommerenke}.

\begin{definition}
\label{Def Cara}
Let $D$ and $(D_{n})_{n\geq 0}$ be open non-empty simply connected domains in $\C$, different from $\C$.
Let $w\in D$, respectively $w_{n}\in D_{n}$.
The sequence of marked domains
$((D_{n},w_{n}))_{n\geq 0}$ is said to converge to
$(D,w)$ in the Carathéodory sense if the following holds:
\begin{enumerate}[label=(\arabic*)]
\item \label{item::defCara1}
$w_{n}\to w$;
\item \label{item::defCara2}
for every $z\in D$, there is a neighborhood $U$ of $z$ in $D$
such that
\begin{displaymath}
U\subset \bigcap_{n\geq m} D_{n}
\end{displaymath}
for $m$ large enough.
\item \label{item::defCara3}
for every $z\in\partial D$, 
there exist $z_{n}\in D_{n}$ such that 
$z_{n}\to z$ as $n\to +\infty$.
\end{enumerate}
\end{definition}

Note that the Carathéodory convergence does not imply that
$D_{n}$ converges $D$ for the Hausdorff distance, even for $D$ bounded.

\subsection{Poisson point processes of boundary to boundary excursions}
\label{Subsec Excursions}
Recall that $\D$ denotes the unit disc and $A_L, A_R$ denote the left and right half-circles in $\partial\D$ as in~\eqref{eqn::def_halfcircle}. 
We first introduce Green's function and Poisson kernel. Denote by $G_{\D}(z,w)$ the Green's function on $\D$
with Dirichlet $0$ boundary conditions: 
\[
G_{\D}(z,w)=\frac{1}{2\pi}\log\left|\frac{1-\bar{z}w}{z-w}\right|, \quad z\neq w\in\D .
	\]
For any simply connected domain $D$, we define Green's function via conformal image. Let $\psi: \D\to D$ be any conformal map, we have 
\begin{align*}
G_D(z,w)=&G_{\D}(\psi(z), \psi(w)),\quad z\neq w\in D. 
\end{align*}
Denote by $H_{\D}(z,x)$ the Poisson kernel on $\D$: 
\begin{equation}
\label{Eq Poisson kernel}
H_{\D}(z,x)=\frac{1}{2\pi}\frac{1-|z|^2}{|x-z|^2}, \quad z\in \D, x\in \partial\D. 
\end{equation}
For any simply connected domain $D$ with a boundary point $x\in\partial D$ such that $\partial D$ is analytic in neighborhood of $x$, we define Poisson kernel via conformal image. Let $\psi: \D\to D$ be any conformal map, we have 
\begin{align*}
H_D(z, x)=&\vert\psi'\circ\psi^{-1}(x)\vert^{-1}
H_{\D}(\psi^{-1}(z), \psi^{-1}(x)), \quad z\in D, x\in\partial D. 
\end{align*}
Denote by $H_{\D}(x,y)$ the boundary Poisson kernel on 
$\partial\D$
(see \cite[Section~5.2]{LawlerConformallyInvariantProcesses}):
\[
H_{\D}(x,y) = \dfrac{1}{\pi \vert y - x\vert^{2}},
\quad x\neq y\in\partial\D .
\]
For any simply connected domain $D$ with two boundary points $x, y$ such that $\partial D$ is analytic in neighborhoods of $x$ and $y$, we define the boundary Poisson kernel via conformal image. Let $\psi: \D\to D$ be any conformal map, we have
\begin{align*}
H_D(x,y)=&\vert\psi'\circ\psi^{-1}(x)\vert^{-1}
\vert\psi'\circ\psi^{-1}(y)\vert^{-1}
H_{\D}(\psi^{-1}(x), \psi^{-1}(y)),\quad x\neq y\in \partial D. 
\end{align*}

%Recall that $\sigma_{\partial \D}$ denotes the arc-length measure on
%$\partial\D$ with
%$\sigma_{\partial \D}(\partial \D)=2\pi$. 
%The image of
%$\sigma_{\partial \D}$ by $\psi_{0}$ is
%\begin{displaymath}
%((\psi_{0})_{\ast}\sigma_{\partial \D})(dx)= 
%\dfrac{2 dx}{1+x^{2}}
%~~\text{on } \R .
%\end{displaymath}
%Next we describe the measures on Brownian excursions.
%Here and in the sequel "Brownian" will refer to the standard Brownian motion in $\C$. 
Next, we describe the measures on Brownian excursions. 
Given $x\neq y\in\partial\D$,
denote by $\mu_{x,y}^{\D,\#}$ the normalized probability measure
on Brownian excursions from $x$ to $y$ in $\D$;
see \cite[Section~5.2]{LawlerConformallyInvariantProcesses}. Denote by 
$\mu_{x,y}^{\D}$ the non-normalized measure
\[
\mu_{x,y}^{\D} := 
H_{\D}(x,y)\mu_{x,y}^{\D,\#}.
\]
For $x\in\partial\D$, let $\mu_{x,x}^{\D}$ denote the measure on Brownian excursions from to $x$ to $x$ in $\D$
\[
\mu_{x,x}^{\D} = 
\lim_{\substack{y\to x \\ y\in \partial\D\setminus\{x\}}}
\mu_{x,y}^{\D}.
\]
Note that $\mu_{x,x}^{\D}$ is up to a constant the Brownian bubble measure
of~\cite[Section~5.5]{LawlerConformallyInvariantProcesses}.
It has infinite total mass. However,
for every $\varepsilon>0$,
\[
\mu_{x,x}^{\D}(\{\gamma : \diam \Range(\gamma)>\varepsilon\})
< +\infty.
\]

For a general simply connected domain $D$ with two boundary points $x, y$ such that $\partial D$ is analytic in neighborhoods of $x, y$, we may extend the definition of Brownian excursion measure via conformal image: Let $\psi: \D\to D$ be any conformal map, 
\[\mu^D_{x,y}=
\vert\psi'\circ\psi^{-1}(x)\vert^{-1}
\vert\psi'\circ\psi^{-1}(y)\vert^{-1}
\psi_{\ast}
\mu_{\psi^{-1}(x),\psi^{-1}(y)}^{\D}.
\]
The total mass of $\mu^D_{x,y}$ is given by $H_D(x,y)$.

%Given a measurable function $u : A_{\rm L}\mapsto [0,+\infty)$,
%$u\in\mathbb{L}^{1}(A_{\rm L},\sigma_{\partial \D})$,
%$\mu^{\D}_{u}$ will denote the following measure on Brownian excursion
%\begin{displaymath}
%\mu^{\D}_{u}(\cdot) :=
%\dfrac{1}{2}\iint_{A_{\rm L}\times  A_{\rm L}}
%\sigma_{\partial \D}(dx)\sigma_{\partial \D}(dy)
%u(x) u(y)
%\mu_{x,y}^{\D}(\cdot).
%\end{displaymath}
%The case of $u$ piecewise constant was considered in
%\cite{AruLupuSepulvedaFPSGFFCVGISO} in relation with
%isomorphisms with the Gaussian free field.
%Here we will further extend the framework and also
%consider instead of non-negative functions $u$
%finite non-negative Radon measures 
%$\nu$ on $\overline{A_{\rm L}}$.
Suppose $\nu$ is a finite non-negative Radon measure on $\overline{A_L}$, we separate its 
atomic and non-atomic parts:
\[
\Atom (\nu):=
\{x\in\overline{A_{\rm L}} :
\nu(\{ x\})>0\},
\qquad
\hat{\nu}:=
\nu - \sum_{x\in \Atom (\nu)}\nu(\{ x\})\delta_{x}.
\]
Note that $\Atom (\nu)$ is at most countable.
We define $\mu^{\D}_{\nu}$ as in~\eqref{eqn::Brownianexcursionmeasurenu} and we see that 
\begin{eqnarray*}
\mu^{\D}_{\nu}(\cdot)&=&
\dfrac{1}{2}
\iint_{\overline{A_{\rm L}}\times \overline{A_{\rm L}}}
d(\nu\otimes\nu)(x,y)\mu_{x,y}^{\D}(\cdot)
\\
&=&
\dfrac{1}{2}
\iint_{A_{\rm L}\times A_{\rm L}}
d\hat{\nu}(x)d\hat{\nu}(y)\mu_{x,y}^{\D}(\cdot)
+
\dfrac{1}{2}
\sum_{x\in \Atom (\nu)} \nu(\{ x\})
\int_{A_{\rm L}}d\hat{\nu}(y)\mu_{x,y}^{\D}(\cdot)
\\ &&
+
\dfrac{1}{2}
\sum_{y\in \Atom (\nu)} \nu(\{ y\})
\int_{A_{\rm L}}d\hat{\nu}(x)\mu_{x,y}^{\D}(\cdot)
+
\dfrac{1}{2}
\sum_{(x,y)\in \Atom (\nu)^{2}} 
\nu(\{ x\})\nu(\{ y\})\mu_{x,y}^{\D}(\cdot).
\end{eqnarray*}
Note that the last of the four terms above 
also involves measures $\mu^{\D}_{x,x}$. All other terms
only involve measures $\mu^{\D}_{x,y}$ for $y\neq x$.
The measure $\mu^{\D}_{\nu}$ is conformally covariant in the following sense. 
If $D\subsetneq\C$ is an open simply connected domain with 
piecewise analytic boundary and $\psi$ is a conformal transformation from
$\D$ to $D$, then the image of $\mu^{\D}_{\nu}$
by $\psi$ is the measure
\begin{displaymath}
\iint_{\overline{\psi(A_{\rm L})}\times \overline{\psi(A_{\rm L})}}
d((\psi_{\ast}\nu)\otimes(\psi_{\ast}\nu))(x,y)
\vert\psi'\circ\psi^{-1}(x)\vert
\vert\psi'\circ\psi^{-1}(y)\vert
\mu_{x,y}^{D}(d\gamma)
\end{displaymath}
up to a change of time in excursions
$ds=\vert (\psi^{-1})'(\gamma(t))\vert^{2} dt$. 

Denote by $\Xi_{\nu}$ the Poisson point process of intensity
$\mu^{\D}_{\nu}$. We see it as a random at most countable collection of
time-parametrized Brownian boundary-to-boundary excursions in $\D$.
Given $\varepsilon>0$, denote 
\begin{equation}
\label{Eq exc eps}
\Xi_{\nu,\varepsilon}:=
\{\gamma\in\Xi_{\nu} : \diam \Range(\gamma)>\varepsilon\}.
\end{equation}

\begin{lemma}
\label{Lem Xi elem}
Let $\nu$ be a finite non-negative non-zero Radon measure
on $\overline{A_{\rm L}}$. Then $\Xi_{\nu}$ satisfies the following.
\begin{enumerate}[label=(\arabic*)]
\item \label{item::pppelem1} A.s. for every $\varepsilon >0$,
the subset
$\Xi_{\nu,\varepsilon}$
is finite.
\item \label{item::pppelem2} A.s. for every subarc $A$ of 
$\overline{A_{\rm L}}$ such that
$\nu(A)>0$, the subset
$\{\gamma\in\Xi_{\nu} : \gamma
\text{ has both ends in } A\}$ is infinite.
\end{enumerate}
\end{lemma}
\begin{proof}
The first point comes from that
\begin{displaymath}
\sup_{x,y\in\partial\D}
\mu_{x,y}^{\D}(\{\gamma : \diam \Range(\gamma)>\varepsilon\})
<+\infty .
\end{displaymath}
See \cite[Section~5.2]{LawlerConformallyInvariantProcesses}.

For the second point it is enough to restrict to a countable collection of subarcs.
If such a subarc $A$ contains an atom
$x_{0}$ of $\nu$, then
$\mu^{\D}_{\nu}\geq \frac{1}{2}\nu\{x_{0}\}^{2}\mu^{\D}_{x_{0},x_{0}}$,
and the measure $\mu^{\D}_{x_{0},x_{0}}$ on excursions with both
endpoints in $x_{0}$ has infinite total mass. If 
$A\cap\Atom (\nu)=\emptyset$, then one needs to check that
\begin{displaymath}
\iint_{A\times A}
\dfrac{1}{\vert y -x \vert^{2}} d\nu (x)d\nu (y) = +\infty.
\end{displaymath}
The integral above is the two-dimensional energy of the measure
$\one_{A}\nu$; see \cite[Definition 3.4.1]{BishopPeresFractals}.
Since the Hausdorff dimension of $A$ is 1, and in particular smaller than 2, the two-dimensional energy equals $+\infty$;
see \cite[Theorem 3.4.2]{BishopPeresFractals}.
\end{proof}

\begin{figure}[ht!]
\begin{center}
\includegraphics[width=0.3\textwidth]{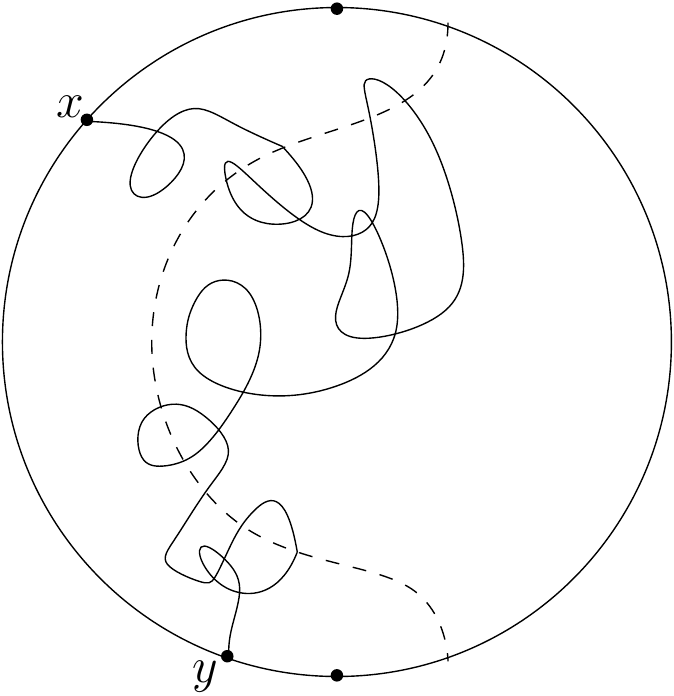}$\quad$
\includegraphics[width=0.3\textwidth]{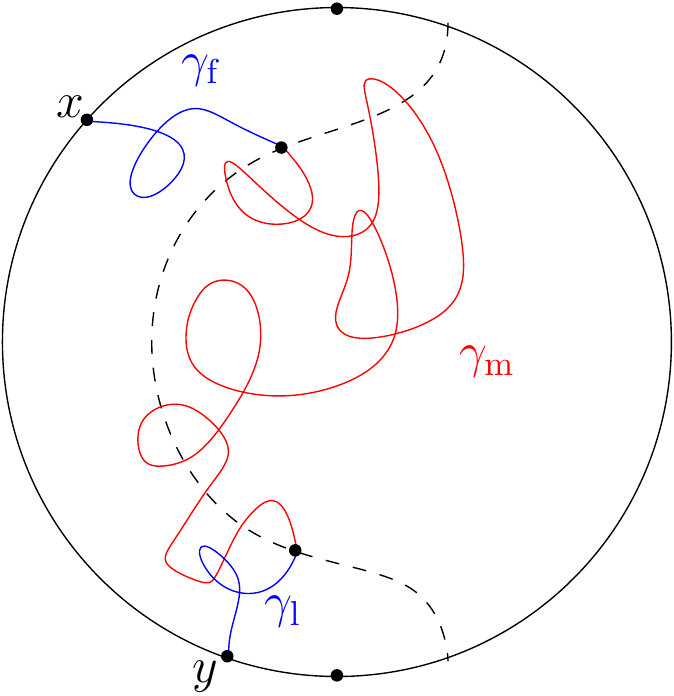}
\end{center}
\caption{\label{fig::Markovdecomposition} In the left panel, the dashed line indicates $\partial\D_{\varepsilon}\cap\D$ and the curve indicates a continuous path $\gamma$ intersecting $\D_{\varepsilon}$. In the right panel, the path $\gamma$ is decomposed into three pieces. The first piece is $\gamma_{\rm{f}}$ (in blue): it is the part of $\gamma$ from 0 to $T_{\gamma,\varepsilon}^{\rm{f}}$. The second piece is $\gamma_{\rm{m}}$ (in red): it is the part of $\gamma$ from $T_{\gamma,\varepsilon}^{\rm{f}}$ to $T_{\gamma,\varepsilon}^{\rm{l}}$. The last piece is $\gamma_{\rm{l}}$ (in blue): it is the part of $\gamma$ from $T_{\gamma,\varepsilon}^{\rm{l}}$ to $T_{\gamma}$. }
\end{figure}

Next we describe the Markovian decomposition of the measures on
Brownian excursions.
Given $z\neq w\in\D$,
denote by $\mu_{z,w}^{\D,\#}$ the normalized probability measure
on Brownian excursions from $z$ to $w$ in $\D$;
see \cite[Section~5.2]{LawlerConformallyInvariantProcesses}.
Denote by $\mu_{z,w}^{\D}$ the non-normalized measure
\[
\mu_{z,w}^{\D} := 
G_{\D}(z,w)\mu_{z,w}^{\D,\#}.
\] 

For $\varepsilon\in (0,1)$, denote
\begin{equation}
\label{Eq D eps}
\D_{\varepsilon} :=
\{z\in \D : \dist(z,A_{\rm L})>\varepsilon \},
\qquad
\widehat{\D}_{\varepsilon} :=
\{z\in \D : \dist(z,A_{\rm L})<\varepsilon \}.
\end{equation}
The domain $\widehat{\D}_{\varepsilon}$ 
is open and simply connected, with
piecewise analytic boundary. Recall that 
$\mu_{x,y}^{\widehat{\D}_{\varepsilon}}$ denotes the
non-normalized measure on Brownian excursions from $x$ to $y$
in $\widehat{\D}_{\varepsilon}$.
Denote by $\sigma_{\partial\widehat{\D}_{\varepsilon}}$ the
arc-length measure on $\partial\widehat{\D}_{\varepsilon}$.
Denote by $T_{\gamma}$ the total duration of a generic element 
$\gamma$ of $\Xi_{\nu}$.
Given $(\gamma(t))_{0\leq t\leq T_{\gamma}}$ a continuous path
intersecting $\D_{\varepsilon}$, denote
\begin{equation}
\label{Eq Tf Tl}
T_{\gamma,\varepsilon}^{\rm f}:=
\inf\{t\in (0,T_{\gamma}) : \gamma(t)\in \D_{\varepsilon}\},
\qquad
T_{\gamma,\varepsilon}^{\rm l}:=
\sup\{t\in (0,T_{\gamma}) : \gamma(t)\in \D_{\varepsilon}\}.
\end{equation}
The Markovian decomposition is as follows.
For details, see
\cite[Section~5.2]{LawlerConformallyInvariantProcesses}
and
\cite[Proposition~3.7]{AruLupuSepulvedaFPSGFFCVGISO}. See also Figure~\ref{fig::Markovdecomposition} for illustration.

\begin{proposition}
\label{Prop Markov}
Fix $\varepsilon\in (0,1)$.
We will denote by $F$ an arbitrary bounded measurable functional on the appropriate space.
For $x,y\in\overline{A_{\rm L}}$,
\begin{align*}
&\int\limits_{
\substack{\gamma \text{ s.t. }\\\Range(\gamma)\cap\D_{\varepsilon}
\neq\emptyset} }
F\big(\gamma(T_{\gamma,\varepsilon}^{\rm f}),
\gamma(T_{\gamma,\varepsilon}^{\rm l}),
(\gamma(t))_{0\leq t\leq T_{\gamma,\varepsilon}^{\rm f}},
(\gamma(T_{\gamma,\varepsilon}^{\rm f}+t))_{0\leq t\leq T_{\gamma,\varepsilon}^{\rm l}
-T_{\gamma,\varepsilon}^{\rm f}},
(\gamma(T_{\gamma,\varepsilon}^{\rm l}+t))_{0\leq t\leq T_{\gamma}-T_{\gamma,\varepsilon}^{\rm l}}
\big)
\mu^{\D}_{x,y}(d\gamma)
\\=&
\iint
\limits_{(\partial\widehat{\D}_{\varepsilon}\cap\D)^{2}}
\sigma_{\partial\widehat{\D}_{\varepsilon}}(dz)
\sigma_{\partial\widehat{\D}_{\varepsilon}}(dw)
\iiint
\mu_{x,z}^{\widehat{\D}_{\varepsilon}}(d\gamma_{\rm f})
\mu^{\D}_{z,w}(d\gamma_{\rm m})
\mu_{w,y}^{\widehat{\D}_{\varepsilon}}(d\gamma_{\rm l})
F(z,w,\gamma_{\rm f},\gamma_{\rm m},\gamma_{\rm l}).
\end{align*}
In particular,
\begin{align*}
&\int\limits_{
\substack{\gamma \text{ s.t. }\\\Range(\gamma)\cap\D_{\varepsilon}
\neq\emptyset} }
F\big(\gamma(T_{\gamma,\varepsilon}^{\rm f}),
\gamma(T_{\gamma,\varepsilon}^{\rm l}),
(\gamma(T_{\gamma,\varepsilon}^{\rm f}+t))_{0\leq t\leq T_{\gamma,\varepsilon}^{\rm l}
-T_{\gamma,\varepsilon}^{\rm f}}
\big)
\mu^{\D}_{x,y}(d\gamma)
\\
=&
\iint
\limits_{(\partial\widehat{\D}_{\varepsilon}\cap\D)^{2}}
\sigma_{\partial\widehat{\D}_{\varepsilon}}(dz)
\sigma_{\partial\widehat{\D}_{\varepsilon}}(dw)
H_{\widehat{\D}_{\varepsilon}}(x,z)
H_{\widehat{\D}_{\varepsilon}}(w,y)
\int
\mu^{\D}_{z,w}(d\gamma_{\rm m})
F(z,w,\gamma_{\rm m}).
\end{align*}
\end{proposition}

\subsection{Loewner chain and SLE}

Recall that $\psi_{0}$ is the Möbius transformation from
$\D$ to $\HH$ as in~\eqref{eqn::psi0}. 
Since we will deal with conformally invariant objects, working in
$\D$ or in $\HH$ will be equivalent, but it will be more convenient 
to handle Brownian excursions in
$\D$, and to work with Loewner chains in $\HH$. 

An $\HH$-hull is a compact subset $K$ of $\overline{\HH}$ such that $\HH\setminus K$ is simply connected. By Riemann's mapping theorem, there exists a unique conformal map $g_K$ from $\HH\setminus K$ onto $\HH$ with the hydrodynamic normalization $\lim_{z\to\infty}|g_K(z)-z|=0$. The quantity 
\[\hcap(K):=\lim_{z\to\infty}z(g_K(z)-z)\]
is non-negative and we call it the half-plane capacity of $K$ (seen from $\infty$). For background on the half-plane capacity, see 
see \cite[Section~3.4]{LawlerConformallyInvariantProcesses}
and \cite[Section~6.2]{BerestyckiNorrisSLE}.

Loewner chain is a collection of $\HH$-hulls $(K_t)_{t\ge 0}$ associated to the family of conformal maps $(g_t)_{ t\ge 0}$ which solves the following Loewner equation: for each $z\in\HH$,
\begin{equation}\label{eqn::Loewnerequation}
\partial_t g_t(z)=\frac{2}{g_t(z)-\xi_t},\quad g_0(z)=z,
\end{equation}
where $(\xi_t)_{t\ge 0}$ is a one-dimensional continuous function which we call the driving function. For $z\in\overline{\HH}$, the swallowing time of $z$ is defined to be $\sup\left\{t\ge 0: \min_{s\in[0,t]}|g_s(z)-\xi_s|>0\right\}$. Let $K_t$ be the closure of $\{z\in\HH: T_z\le t\}$. It turns out that $g_t$ is the unique conformal map from $\HH\setminus K_t$ onto $\HH$ with normalization $\lim_{z\to\infty}|g_t(z)-z|=0$. Since $\hcap(K_t)=\lim_{z\to\infty}z(g_t(z)-z)=2t$, we say that the process $(K_t)_{t\ge 0}$ is parameterized by the half-plane capacity. 
We say that $(K_t)_{t\ge 0}$ can be generated by continuous curve $(\eta(t))_{t\ge 0}$ if, for any $t$, the unbounded connected component of $\HH\setminus\eta[0,t]$ is the same as $\HH\setminus K_t$. 

The following proposition explains which kind of continuous curve enjoys continuous driving function. 
\begin{proposition}\label{prop::Loewnerchain_qualification}
Suppose $T\in (0,\infty]$. Let $\eta: [0,T)\to \overline{\HH}$ be a continuous curve with $\eta(0)=0$. Assume the following hold: for every $t\in (0,T)$, 
\begin{enumerate}[label=(\alph*)]
\item \label{item::Loewnerchain_qualifya} $\eta(t,T)$ is contained in the closure of the unbounded connected component of $\HH\setminus\eta[0,t]$,
\item \label{item::Loewnerchain_qualifyb} $\eta^{-1}(\eta[0,t]\cup\R)$ has empty interior in $(t,T)$. 
\end{enumerate}
For each $t>0$, let $g_t$ be the conformal map from the unbounded connected component of $\HH\setminus\eta[0,t]$ onto $\HH$ with normalization $\lim_{z\to\infty}|g_t(z)-z|=0$. After reparameterization, $(g_t)_{t\ge 0}$ solves~\eqref{eqn::Loewnerequation} with continuous driving function $(\xi_t)_{t\ge 0}$. \end{proposition}
\begin{proof}
See
\cite{PommerenkeLoewner}, 
\cite[Theorem~1.2]{KinnebergLoewner},
\cite[Section~4]{LawlerConformallyInvariantProcesses} and \cite[Proposition~6.12]{MillerSheffieldIG1}. 
\end{proof}

%\begin{lemma}\label{lem::Loewnerchainsimplicity}
%Suppose $\eta$ satisfies the assumptions in Proposition~\ref{prop::Loewnerchain_qualification} and assume the same notations. Suppose $\eta(t)$ is a simple point for $\eta$, then 
%\[\lim_{z\to\eta(t)}g_t(z)=\xi_t. \]
%\end{lemma}
%\begin{proof}
%This can be proved using~\cite[Proof of Lemma~4.2]{LawlerConformallyInvariantProcesses}.
%\end{proof}

Schramm Loewner evolution (SLE) is a Loewner chain with driving function equal a multiple of Brownian motion. For $\kappa>0$, $\SLE_{\kappa}$ is the Loewner chain with driving function $\xi_t=\sqrt{\kappa}B_t$ where $(B_t)_{t\ge 0}$ is a standard one-dimensional Brownian motion. It is known that $\SLE_{\kappa}$ is almost surely generated by a continuous curve for all $\kappa$; see~\cite{RohdeSchrammSLEBasicProperty}. In particular, when $\kappa\in (0,4]$, it is a simple curve.

\subsection{Gaussian free field and level lines}
\label{subsec::levellines}

In this section, we will collect some known results on level lines of GFF from~\cite{DubedatSLEFreefield, SchrammSheffieldDiscreteGFF, SchrammSheffieldContinuumGFF, WangWuLevellinesGFFI, PowellWuLevellinesGFF} and relate the level lines to variants of SLE$_4$ process. 
To this end, it is more convenient to work in $\HH$. 

We first consider the case when GFF has piecewise constant boundary data. Suppose $x_n<\cdots<x_1< 0$ and $\rho_n, \ldots, \rho_1\in\R$. 
Denote 
\[\bar{\rho}_k:=\sum_{j=1}^k \rho_j,\quad \text{for all }k\in \{1,\ldots, n\}.\]
Consider GFF on $\HH$ with the following boundary data: 
\begin{equation}\label{eqn::piecewisebc}
\zeta(x)=2\lambda\one_{(x_1, 0)}(x)+\sum_{k=1}^n \lambda(2+\bar{\rho}_k)\one_{(x_{k+1}, x_k]}(x), \quad x<0; \qquad \zeta(x)=0,\quad x>0, 
\end{equation}
where we use the convention that $x_{n+1}=-\infty$. 
Suppose $\Phi$ is zero-boundary GFF in $\HH$ and suppose $\bar{\rho}_k>-2$ for all $k\in\{1,\ldots, n\}$. 
Then the level line of $\Phi+\zeta$ exists and is uniquely determined by $\Phi$. Furthermore, it is a continuous curve with continuous driving function $(\xi_t)_{t\ge 0}$ which is the solution to the following SDE:  
\[
d\xi_t=2dB_t+\sum_{j=1}^n \frac{\rho_j dt}{\xi_t-V_t^j}, \quad dV_t^k =\frac{2dt}{V_t^k-\xi_t}, \text{ for }k\in\{1, \ldots, n\}, 
\]
with initial values $\xi(0)=0$ and $V_0^k=x_k$ for $k\in\{1, \ldots, n\}$. Note that the Lowener chain with the above driving function is called $\SLE_4(\rho_n, \ldots, \rho_1)$ process with force points $(x_n, \ldots, x_1)$. For more detail on $\SLE_{\kappa}(\rho)$ with multiple force points, see~\cite[Section~2.2]{MillerSheffieldIG1}. 

Next, we consider GFF with regulated boundary conditions. Suppose the boundary condition is a regulated function $\zeta$ on $\R$. Assume that there exists $\eps>0$ such that 
\begin{equation}\label{assumptionPW}
\zeta(x)\ge \eps, \quad x<0; \qquad \zeta(x)=0, \quad x>0.
\end{equation}
The authors in~\cite{PowellWuLevellinesGFF} prove that there exists a coupling $(\Phi, \eta)$ as in Definition~\ref{def::GFFlevelline} with boundary data $\zeta$ and $\eta$ is a continuous simple curve with continuous driving function. Furthermore, they identify the law of $\eta$ when $\zeta$ is of bounded variation. 

Suppose $\zeta$ is of bounded variation and $\zeta=0$ on $\R_+$. 
Such function can be described almost every as the integral of a finite signed Radon measure $\rho$ on $(-\infty,0]$: 
\begin{equation}
\label{Eq bc}
\zeta(x^{+})=2\lambda+\lambda\rho((x,0]), \quad x<0.
\end{equation}
Suppose that there exists $\eps>0$ such that 
\begin{equation}\label{assumption::PowellWu}
\rho((x,0]) \ge -2+\eps/\lambda, \quad x<0. 
\end{equation}
Under the assumption~\eqref{assumption::PowellWu}, the authors in~\cite{PowellWuLevellinesGFF} prove that the law of $\eta$ is an $\SLE_4(\rho)$ process defined as follows. 
\begin{definition}\label{def::SLE4rho}
Suppose $(B_t)_{t\ge 0}$ is one-dimensional Brownian motion. We say that the process \[(\xi_t, (V_t(x))_{x\le 0} )_{t\ge 0}\] describes an $\SLE_4(\rho)$ process if it is adapted to the filtration of $B$ and the following hold:
\begin{itemize}
\item We have $\xi_0=0$ and $V_0(x)=x$ for $x\le 0$. 
\item The processes $B_t, \xi_t, (V_t(x))_{x\le 0}$ satisfy the following SDE on time intervals where $\xi_t$ does not collide with any of the $V_t(x)$:
\begin{equation}
\label{Eq SDE}
d\xi_t=2dB_t+\left(\int_{(-\infty,0]}\frac{d\rho(x)}{\xi_t-V_t(x)}\right)dt, \quad dV_t(x)=\frac{2dt}{V_t(x)-\xi_t},\quad x\le 0.
\end{equation} 
\item We have instantaneous reflection of $\xi_t$ off the $V_t(x)$, i.e. it is almost surely the case that for Lebesgue almost all times $t$ we have that $\xi_t\neq V_t(x)$ for each $x\le 0$. 
\end{itemize}
The $\SLE_4(\rho)$ process is then defined to be the Loewner chain with driving function $(\xi_t)_{t\ge 0}$. 
\end{definition}
Note that the existence of $\SLE_4(\rho)$ is not clear from the above definition. It is part of the conclusion from~\cite{PowellWuLevellinesGFF} that there exists an $\SLE_4(\rho)$ process under the assumption~\eqref{assumption::PowellWu} and it is a continuous simple curve with continuous driving function. We emphasize that~\cite{PowellWuLevellinesGFF} only provides the existence of $\SLE_4(\rho)$, and it does not give the uniqueness in law.

\section{Construction of chordal curves}
\label{Sec Construction}

\subsection{Proof of Propositions~\ref{prop::envelop_parameterization} and~\ref{prop::envelop_drivingfunction}}
\label{SubSec Construction}
Let us recall the construction of $\eta_{\kappa, \nu}$ given in the introduction and provide more detail.  
Our construction of chordal curves in $\overline{\D}$ from $-i$ to $i$ involves two ingredients: Brownian excursions
introduced in Section~\ref{Subsec Excursions} 
and conformal loop ensembles
$\CLE_{\kappa}$ with $\kappa\in (8/3 , 4]$. 
For the construction of the CLE, 
see \cite{SheffieldExplorationTree,SheffieldWernerCLE}.
Note that according to \cite{SheffieldWernerCLE}, a 
$\CLE_{\kappa}$ is also the set of outermost boundaries of
clusters in Brownian loop-soups that were introduced in 
\cite{LawlerWernerBrownianLoopsoup}.
Here we emphasize that $\CLE_{\kappa}$ satisfies a local finiteness property:
a.s., for every $\varepsilon>0$, there are only finitely many loops of diameter greater than $\varepsilon$.

Here is our construction.
Fix $\kappa\in (8/3 , 4]$ and 
let $\FC_{\kappa}$ denote a $\CLE_{\kappa}$ loop ensemble.
Fix $\nu$ a finite non-negative Radon measure on $\overline{A_{\rm L}}$
and 
let $\Xi_{\nu}$ be a Poisson point process of excursions of intensity 
$\mu_{\nu}^{\D}$, independent of $\FC_{\kappa}$.
For $\gamma\in \Xi_{\nu}$,
denote
\begin{equation}
\label{Eq notations gamma}
\widetilde{\FC}_{\kappa}(\gamma):=
\{\tilde{\gamma}\in\FC_{\kappa} :
\Range(\tilde{\gamma})\cap \Range(\gamma)
\neq\emptyset\},
\qquad
\LS_{\kappa}(\gamma):=
\Big(\bigcup_{\tilde{\gamma}\in\widetilde{\FC}_{\kappa}(\gamma)}
\Range(\tilde{\gamma})\Big)
\cup
\Range(\gamma).
\end{equation}
Define
\begin{equation}
\label{Eq tilde FC}
\widetilde{\FC}_{\kappa,\nu}:=
\bigcup_{\gamma\in \Xi_{\nu}}\widetilde{\FC}_{\kappa}(\gamma)
=
\{\tilde{\gamma}\in\FC_{\kappa} :
\exists\gamma\in \Xi_{\nu}, 
\Range(\tilde{\gamma})\cap \Range(\gamma)
\neq\emptyset\}.
\end{equation}
Let $\LS_{\kappa,\nu}$ be the following random subset of
$\overline{\D}$:
\begin{displaymath}
\LS_{\kappa,\nu}:=
\bigcup_{\gamma\in\Xi_{\nu}}\LS_{\kappa}(\gamma)
=
\Big(\bigcup_{\tilde{\gamma}\in\widetilde{\FC}_{\kappa,\nu}}
\Range(\tilde{\gamma})\Big)
\cup
\Big(\bigcup_{\gamma\in\Xi_{\nu}}\Range(\gamma)\Big).
\end{displaymath}
In the limit case $\kappa=8/3$, we set
\begin{displaymath}
 \LS_{8/3,\nu}:=
 \bigcup_{\gamma\in\Xi_{\nu}}\Range(\gamma).
\end{displaymath}

By construction, 
$\LS_{\kappa,\nu}\cap A_{\rm R} = \emptyset$.
Let $\LD_{\rm R,\kappa,\nu}$ be the connected component of
$\overline{\D}\setminus
(\LS_{\kappa,\nu}\cup\overline{A_{\rm L}})$ 
that contains $A_{\rm R}$.
Then $\LD_{\rm R,\kappa,\nu}$ is of form
$\LO_{\kappa,\nu}\cup A_{\rm R}$,
where $\LO_{\kappa,\nu}$ is an open simply connected domain.
Set
\begin{displaymath}
\eta_{\kappa,\nu}:=
(\partial \LD_{\rm R,\kappa,\nu})\setminus A_{\rm R}.
\end{displaymath}
Informally, $\eta_{\kappa,\nu}$ is constructed as the
envelop from the right of the set 
$\LS_{\kappa,\nu}\cup \overline{A_{\rm L}}$.
First of all, we will show that $\eta_{\kappa, \nu}$ is a continuous curve and satisfies conformal covariance. 

\begin{proof}[Proof of Proposition~\ref{prop::envelop_parameterization}]
The conformal covariance in law of $\eta_{\kappa,\nu}$
follows form the conformal invariance in law of the
$\CLE_{\kappa}$ and the conformal covariance in law of
$\Xi_{\nu}$. It remains to show Proposition~\ref{prop::envelop_parameterization}~\ref{item::envelop_continuity}. 

For the continuity of the envelop we will use a somewhat different argument from~\cite[Section~2.4]{WernerWuCLEtoSLE},
relying on Lemma~\ref{Lem Count Union Loc Con}. 
If one shows that
$\partial\LO_{\kappa,\nu}$ is a continuous closed curve,
then one gets that
$\eta_{\kappa,\nu}=
\partial\LO_{\kappa,\nu}\setminus A_{\rm R}$
is a continuous curve. 
Its endpoints are $-i$ and $i$ since these are also the
endpoints of $A_{\rm R}$.
According to Theorem~\ref{Thm Cara boundary},
to show that $\partial\LO_{\kappa,\nu}$ is a continuous closed curve
one needs to check that $\C\setminus \LO_{\kappa,\nu}$
is locally connected.
According to Lemma~\ref{Lem Loc Con Boundary},
it is enough to check that
$\LS_{\kappa,\nu}\cup \partial \D$ is locally connected.
To this end, we will apply 
Lemma~\ref{Lem Count Union Loc Con} twice.
The first time, we apply it to
$K_{0}=\partial\D$ and
$(K_{n})_{n\geq 1}=\Xi_{\nu}$.
We get that 
$\partial \D\cup\bigcup_{\gamma\in \Xi_{\nu}}\Range(\gamma)$
is locally connected.
The second time, we apply it to
$K_{0}=\partial \D\cup\bigcup_{\gamma\in \Xi_{\nu}}\Range(\gamma)$
and $(K_{n})_{n\geq 1}=\widetilde{\FC}_{\kappa,\nu}$
and get that 
$\LS_{\kappa,\nu}\cup \partial \D$ is locally connected.
This completes the proof. 
\end{proof}

From the construction, the curves $\eta_{\kappa,\nu}$ satisfy an obvious monotonicity in 
$\nu$. 
Indeed, if $\nu_{1}\leq \nu_{2}$,
$\Xi_{\nu_{1}}$ can be realized as a subset of
$\Xi_{\nu_{2}}$.

\begin{proposition}
\label{Prop Monotonicity}
Fix $\kappa\in [8/3 , 4]$.
Let $\nu_{1}$ and $\nu_{2}$ be two finite non-negative Radon measures on 
$\overline{A_{\rm L}}$ such that
$\nu_{1}\leq \nu_{2}$,
i.e. $\nu_{2}-\nu_{1}$ is a non-negative measure.
Then $\eta_{\kappa,\nu_{1}}$ and $\eta_{\kappa,\nu_{2}}$
can be coupled on the same probability space such that
a.s., $\eta_{\kappa,\nu_{1}}$ is contained between
$\overline{A_{\rm L}}$ and $\eta_{\kappa,\nu_{2}}$,
and in particular,
\begin{displaymath}
\left(\eta_{\kappa,\nu_{2}}\cap A_{\rm L}\right)
\subset
\left(\eta_{\kappa,\nu_{1}}\cap A_{\rm L}\right).
\end{displaymath}
\end{proposition}

By construction, 
$\eta_{\kappa,\nu}\cap A_{\rm R}=\emptyset$.
However, $\eta_{\kappa,\nu}$ may intersect $A_{\rm L}$.
Next we give a condition under which
$\eta_{\kappa,\nu}$ may contain a whole subarc of $A_{\rm L}$.

\begin{proposition}
\label{Prop contain subarc}
Fix $\kappa\in [8/3 , 4]$ and 
$\nu$ a finite non-negative Radon measure on $\overline{A_{\rm L}}$.
Let $A$ be a non-empty open subarc of $A_{\rm L}$
and let $A_{1}$ and $A_{2}$ denote the two connected components of
$\overline{A_{\rm L}}\setminus A$.
Then $\PP(\overline{A}\subset \eta_{\kappa,\nu})>0$
if and only if $\nu(A)=0$.
Moreover, in the latter case,
the event $\overline{A}\subset \eta_{\kappa,\nu}$
coincides a.s. with the event defined by the following two conditions
(and only the first one if $\kappa=8/3$).
\begin{enumerate}[label=(\arabic*)]
\item There is no excursion in $\Xi_{\nu}$
joining $A_{1}$ and $A_{2}$.
\item The process $\Xi_{\nu}$ does not contain 
an excursion from $A_{1}$ to $A_{1}$ and an excursion from
$A_{2}$ to $A_{2}$ intersecting the same
$\CLE_{\kappa}$ loop in $\FC_{\kappa}$
(provided $\kappa\neq 8/3$).
\end{enumerate}
On the complementary event,
again in the case $\nu(A)=0$,
we have a.s. $A\cap \eta_{\kappa,\nu}=\emptyset$.
\end{proposition}

\begin{proof}
First assume that $\nu(A)>0$.
According to Lemma~\ref{Lem Xi elem},
the process $\Xi_{\nu}$ contains a.s. an excursion $\gamma$ with
both endpoints $x,y$ in $A$.
This implies that $A(x,y)\cap\eta_{\kappa,\nu} =\emptyset$,
where $A(x,y)$ is the open subarc of $A$ with endpoints $x,y$.
In particular, a.s. $A\not\subset \eta_{\kappa,\nu}$.

Next assume that $\nu(A)=0$.
Also take $\kappa \neq 8/3$.
The case $\kappa=8/3$ is actually simpler.
Since $\nu(A)=0$, the process $\Xi_{\nu}$ does not contain excursions
with one or both endpoints in $A$.
\begin{itemize}
\item Let $E_{0}$ denote the event that
$\overline{A}\subset \eta_{\kappa,\nu}$.
Let $\widetilde{E}_{0}$ be the event that 
$A\cap \eta_{\kappa,\nu}=\emptyset$.
Clearly, $\widetilde{E}_{0}\subset E_{0}^{\rm c}$.
\item Let $E_{1}$ denote the event that there is an excursion in 
$\Xi_{\nu}$ joining $A_{1}$ and $A_{2}$.
If $\gamma$ is such an excursion, its range disconnects in
$\overline{\D}$ the arc $A$ from $A_{\rm R}$.
Thus, $E_{1}\subset \widetilde{E}_{0}$.
\item Let $E_{2}$ denote the event that there is
an excursion from $A_{1}$ to $A_{1}$ and an excursion from
$A_{2}$ to $A_{2}$ intersecting the same
$\CLE_{\kappa}$ loop in $\FC_{\kappa}$.
If $\gamma_{1}$, respectively $\gamma_{2}$,
are such excursion from $A_{1}$, respectively $A_{2}$,
and $\tilde{\gamma}$ is the common loop in $\FC_{\kappa}$
they intersect,
then 
$\Range(\gamma_{1})\cup\Range(\gamma_{2})\cup\Range(\tilde{\gamma})$
disconnects in $\overline{\D}$
the arc $A$ from $A_{\rm R}$.
Thus, $E_{2}\subset \widetilde{E}_{0}$.
\end{itemize}

Let us further show that
$E_{1}^{\rm c}\cap E_{2}^{\rm c}
\subset E_{0}$,
which will establish 
$E_{0}=\widetilde{E}_{0}^{\rm c}=E_{1}^{\rm c}\cap E_{2}^{\rm c}$.
Set
\begin{displaymath}
\Xi_{1}:=
\{\gamma\in\Xi_{\nu}\vert
\gamma \text{ has both endpoints in } A_{1}\},
\qquad
\Xi_{2}:=
\{\gamma\in\Xi_{\nu}\vert
\gamma \text{ has both endpoints in } A_{2}\},
\end{displaymath}
\begin{displaymath}
\widetilde{\FC}_{\kappa,1}:=
\{\tilde{\gamma}\in\FC_{\kappa} :
\exists\gamma\in \Xi_{1}, 
\Range(\tilde{\gamma})\cap \Range(\gamma)
\neq\emptyset\},
\end{displaymath}
\begin{displaymath}
\widetilde{\FC}_{\kappa,2}:=
\{\tilde{\gamma}\in\FC_{\kappa} :
\exists\gamma\in \Xi_{2}, 
\Range(\tilde{\gamma})\cap \Range(\gamma)
\neq\emptyset\},
\end{displaymath}
\begin{displaymath}
\LS_{1}:=
\Big(\bigcup_{\tilde{\gamma}\in\widetilde{\FC}_{\kappa,1}}
\Range(\tilde{\gamma})\Big)
\cup
\Big(\bigcup_{\gamma\in\Xi_{1}}\Range(\gamma)\Big),
\qquad
\LS_{2}:=
\Big(\bigcup_{\tilde{\gamma}\in\widetilde{\FC}_{\kappa,2}}
\Range(\tilde{\gamma})\Big)
\cup
\Big(\bigcup_{\gamma\in\Xi_{2}}\Range(\gamma)\Big).
\end{displaymath}
The local finiteness of the CLE ensures that
$\LS_{1}\cup A_{1}$ and $\LS_{2}\cup A_{2}$
are closed subsets of $\overline{\D}$ and that
$\partial\D\cap\LS_{1}\subset A_{1}$,
$\partial\D\cap\LS_{2}\subset A_{2}$.
Thus, neither $\LS_{1}\cup A_{1}$ nor $\LS_{2}\cup A_{2}$
disconnect $A$ from $A_{\rm R}$ in $\overline{\D}$.
Moreover, on the event $E_{2}^{\rm c}$,
we have 
\begin{displaymath}
(\LS_{1}\cup A_{1})\cap(\LS_{2}\cup A_{2})
=\emptyset.
\end{displaymath}
It follows from the Janiszewski's theorem, 
on the event $E_{2}^{\rm c}$, we have that 
$(\LS_{1}\cup A_{1})\cup(\LS_{2}\cup A_{2})$
does not disconnect $A$ from $A_{\rm R}$ in $\overline{\D}$;
see \cite[Section~1.1]{Pommerenke}.
Moreover, on the event $E_{1}^{\rm c}$, we have
$\LS_{\kappa,\nu}=\LS_{1}\cup\LS_{2}$.
So on the event $E_{1}^{\rm c}\cap E_{2}^{\rm c}$,
the set $\LS_{\kappa,\nu}$ does not
disconnect $A$ from $A_{\rm R}$ in $\overline{\D}$
and thus, $\overline{A}\subset \eta_{\kappa,\nu}$.

Finally, let us check that $\PP(E_{1}^{\rm c}\cap E_{2}^{\rm c})>0$.
Actually, the two events are independent and it suffices to show $\PP(E_1^{\rm c})>0$ and $\PP(E_2^{\rm c})>0$.
\begin{itemize}
\item We have $\PP(E_{1})<1$ because the intensity measure of excursions from
$A_{1}$ to $A_{2}$ is finite.
\item Let $U_{1}$ be an open neighborhood of $A_{1}$
such that $\overline{U_{1}}\cap A_{2}=\emptyset$.
The local finiteness and the fact that the CLE loops do not hit the boundary guarantee that there is 
$U_{2}$ an open neighborhoods of $A_{2}$,
such that $\overline{U_{1}}\cap\overline{U_{2}}=\emptyset$
and such that with positive probability no loop in 
$\FC_{\kappa}$ intersects both $U_{1}$ and $U_{2}$.
With positive probability,
all the excursions of $\Xi_{1}$ are contained in $U_{1}$,
and all the excursions of $\Xi_{2}$ are contained in $U_{2}$,
and these are independent events;
see Lemma \ref{Lem Xi elem}.
Thus, $\PP(E_{2}^{\rm c})>0$. 
\end{itemize}
These complete the proof. 
\end{proof}

Now, we are ready to complete the proof of Proposition~\ref{prop::envelop_drivingfunction}. 
\begin{proof}[Proof of Proposition~\ref{prop::envelop_drivingfunction}]
Recall that $\psi_0$ is the conformal transformation from $\D$ to $\HH$ given by~\eqref{eqn::psi0} and $
\tilde{\eta}_{\kappa,\nu} := \psi_{0}(\eta_{\kappa,\nu})$. It suffices to check that $\eta_{\kappa, \nu}$ satisfies the conditions in Proposition~\ref{prop::Loewnerchain_qualification}. 
\begin{itemize}
\item From the construction, $\eta_{\kappa, \nu}$ clearly satisfies Proposition~\ref{prop::Loewnerchain_qualification}~\ref{item::Loewnerchain_qualifya}. 
\item As the support of $\nu$ equals $\overline{A_L}$, $\eta_{\kappa, \nu}$ satisfies Proposition~\ref{prop::Loewnerchain_qualification}~\ref{item::Loewnerchain_qualifyb} due to Proposition~\ref{Prop contain subarc}. 
%\item As $\nu$ has no atoms, $\eta_{\kappa, \nu}$ is simple due to Lemma~\ref{lem::simple}. 
\end{itemize}
Therefore, 
$\tilde{\eta}_{\kappa,\nu}$ is a continuous curve in
$\overline{\HH}$ from $0$ to $\infty$ with continuous driving function.
The half-plane capacity is a continuous strictly increasing function on $\tilde{\eta}_{\kappa,\nu}$.
So one can parametrize $\tilde{\eta}_{\kappa,\nu}$
as $(\tilde{\eta}_{\kappa,\nu}(t))_{0\leq t<T_{\rm max}}$,
with $T_{\rm max}\in (0,+\infty]$,
such that~\eqref{eqn::envelop_hcap} holds. 
Note that one does not necessarily have $T_{\rm max}=+\infty$.
\end{proof}

Next, we consider the simplicity of $\eta_{\kappa,\nu}$. 
\begin{lemma}
\label{lem::simple}
Fix $\kappa\in [8/3 , 4]$ and 
$\nu$ a finite non-negative Radon measure on $\overline{A_{\rm L}}$.
 Assume further that $\nu$ has no atoms. Then the curve $\eta_{\kappa,\nu}$ is a.s. simple.
\end{lemma}

\begin{proof}

According to Theorem \ref{Thm Cara boundary},
to show that $\partial\LO_{\kappa,\nu}$ is a Jordan curve one
additionally needs to check that $\partial\LO_{\kappa,\nu}$
has no cut points.
Since $\nu$ has no atoms, for every $\gamma\in \Xi_{\nu}$,
the two endpoints of $\gamma$ are distinct.
Denote by $\LR(\gamma)$ the right boundary of the excursion
$\gamma$,
defined rigorously as a portion of the boundary of the
connected component of 
$\overline{\D}\setminus\Range(\gamma)$ that contains $A_{\rm R}$.
According to 
\cite[Corollary~8.5]{LawlerSchrammWernerConformalRestriction},
$\LR(\gamma)$ is a continuous simple curve joining
the two endpoints of $\gamma$,
more specifically distributed as a
chordal $\SLE_{8/3}(\rho)$ process,
with $\rho=2/3$. Denote
\begin{displaymath}
\widehat{\FC}_{\kappa,\nu}:=
\{\tilde{\gamma}\in\FC_{\kappa} :
\exists\gamma\in \Xi_{\nu}, 
\Range(\tilde{\gamma})\cap \LR(\gamma)
\neq\emptyset\},
\qquad
\widehat{\LS}_{\kappa,\nu}:=
\Big(\bigcup_{\tilde{\gamma}\in\widehat{\FC}_{\kappa,\nu}}
\Range(\tilde{\gamma})\Big)
\cup
\Big(\bigcup_{\gamma\in\Xi_{\nu}}\LR(\gamma)\Big).
\end{displaymath}
Then $\LO_{\kappa,\nu}$ is a connected component of
$\C\setminus(\widehat{\LS}_{\kappa,\nu}\cup\partial\D)$.
According to Lemma \ref{Lem K cuts},
it is enough to check that $\widehat{\LS}_{\kappa,\nu}\cup\partial\D$
has no cut points.
To this end, we classify the points of 
$\widehat{\LS}_{\kappa,\nu}\cup\partial\D$ as follows:
\begin{enumerate}[label=(\arabic*)]
\item \label{item::simple1}
The points of $\partial \D$ that are not an endpoints of an excursion $\gamma\in \Xi_{\nu}$.
\item \label{item::simple2}
The endpoints of excursions $\gamma\in \Xi_{\nu}$.
\item \label{item::simple3}
The points on $\LR(\gamma)$, 
for $\gamma\in \Xi_{\nu}$,
that are not endpoints and do not lie on
$\Range(\tilde{\gamma})$ for 
$\tilde{\gamma}\in \widehat{\FC}_{\kappa,\nu}$.
\item \label{item::simple4}
The points on $\Range(\tilde{\gamma})$, for 
$\tilde{\gamma}\in \widehat{\FC}_{\kappa,\nu}$,
that do not lie on $\LR(\gamma)$
for $\gamma\in \Xi_{\nu}$.
\item \label{item::simple5}
The points that belong to an intersection
$\LR(\gamma)\cap \Range(\tilde{\gamma})$
for $\gamma\in \Xi_{\nu}$ and
$\tilde{\gamma}\in \widehat{\FC}_{\kappa,\nu}$.
\end{enumerate}
It is clear that the points of type~\ref{item::simple1} cannot be cut points.
The points of type~\ref{item::simple2} cannot be cut points because each $\LR(\gamma)$ has two distinct endpoints. 
The points of type~\ref{item::simple3} cannot be cut points because the curve $\LR(\gamma)$ is simple.
The points of type~\ref{item::simple4} cannot be cut points because $\CLE_{\kappa}$ loops are Jordan curves.
Regarding the points of type~\ref{item::simple5}, 
they cannot be cut points because for
every $\gamma\in \Xi_{\nu}$ and
$\tilde{\gamma}\in \widehat{\FC}_{\kappa,\nu}$,
the intersection $\LR(\gamma)\cap \Range(\tilde{\gamma})$
is either empty or contains at least two points.
Indeed, given an independent Brownian motion that hits a 
$\CLE_{\kappa}$ loop, 
it will a.s. enter the interior surrounded by the loop.
\end{proof}

Let us explore more the multiple points of the curve
$\eta_{\kappa,\nu}$ in the case the measure $\nu$ has atoms.
Denote
\begin{displaymath}
\Atom_{\rm isol}(\nu) :=
\{x\in \Atom(\nu) : \exists A \subset \partial \D, A \text{ open arc, }
x\in A, \one_{A}\nu = \nu(\{ x\})\delta_{x}\},
\end{displaymath}
\begin{displaymath}
\Atom_{\rm conv}(\nu) :=
\Big\{x\in \Atom(\nu) : 
\int_{\overline{A_{\rm L}}\setminus \{ x\}}
\dfrac{1}{\vert y-x\vert^{2}} d\nu(y)<+\infty
\Big\}.
\end{displaymath}
Note that $\Atom_{\rm isol}(\nu)\subset \Atom_{\rm conv}(\nu)$.

\begin{proposition}
\label{Prop double points}
Fix $\kappa\in [8/3 , 4]$ and 
$\nu$ a finite non-negative Radon measure on $\overline{A_{\rm L}}$.
Let $\psi$ be an uniformizing map from
$\D$ to $\LO_{\kappa,\nu}$.
Then a.s., for every $x\in\eta_{\kappa,\nu}$,
the number $\Card(\psi^{-1}(\{x\}))$
(which does not depend on the choice of $\psi$)
is either $1$ or $2$.
Moreover, a.s.,
\begin{equation}\label{eqn::doublepoints1}
\{x\in \eta_{\kappa,\nu} : \Card(\psi^{-1}(\{x\}))=2\}
\subset \Atom_{\rm conv}(\nu)
\subset \overline{A_{\rm L}}.
\end{equation}
In particular, if $\Atom_{\rm conv}(\nu)=\emptyset$,
then the curve $\eta_{\kappa,\nu}$ is a.s. simple.
As a partial converse, 
we have that for every $x\in \Atom_{\rm isol}(\nu)$,
\begin{equation}\label{eqn::doublepoints2}
\PP(\Card(\psi^{-1}(\{x\}))=2)>0.
\end{equation}
\end{proposition}

\begin{proof}
According to \cite[Proposition~2.5]{Pommerenke}
and the corresponding proof,
for every $x\in \eta_{\kappa,\nu}$,
the number $\Card(\psi^{-1}(\{x\}))$
equals the number of connected components in
$\C\setminus(\LO_{\kappa,\nu}\cup\{ x\})$.
Similarly to Lemma~\ref{lem::simple},
one can show that $\C\setminus\LO_{\kappa,\nu}$ does not have
cut points in $\eta_{\kappa,\nu}\setminus\Atom(\nu)$. 
Thus $\Card(\psi^{-1}(\{x\}))$ is either 1 or 2 and it can be 2 only for points $x\in\Atom(\nu)$. 

Then we show~\eqref{eqn::doublepoints1}. From Lemma~\ref{Lem Intersect} below
and an absolute continuity argument near an endpoint applied to 
Brownian excursions,
it follows that for every
$\gamma,\gamma' \in \Xi_{\nu}$,
such that $\Range(\gamma)\cap\Range(\gamma')\neq\emptyset$,
one also has
$(\Range(\gamma)\cap\Range(\gamma'))\setminus \partial\D\neq\emptyset$.
If $x\in\Atom(\nu)\setminus\Atom_{\rm conv}(\nu)$,
then a.s. there is an excursion $\gamma\in \Xi_{\nu}$
with one endpoint $x$ and the other endpoint different from $x$.
For every $\gamma'$ excursion in $\Xi_{\nu}$
with both endpoints in $x$, we have
$(\Range(\gamma)\cap\Range(\gamma'))\setminus \{x\}\neq\emptyset$.
This implies that
$\LS_{\kappa,\nu}\setminus \{x\}$ is connected a.s.
Thus,
for every $x\in\eta_{\kappa,\nu}\setminus\Atom_{\rm conv}(\nu)$,
we have $\Card(\psi^{-1}(\{x\}))=1$. This completes the proof for~\eqref{eqn::doublepoints1}. 

Now, consider $x\in \Atom_{\rm isol}(\nu)$.
Let $\Xi_{x}$ be the subset of $\Xi_{\nu}$
made of all the excursions with both endpoints in $x$.
Denote
\begin{displaymath}
\LS_{x}:=\bigcup_{\gamma\in\Xi_{x}}\LS_{\kappa}(\gamma).
\end{displaymath}
Since for all $\gamma,\gamma'\in \Xi_{x}$,
we have $(\Range(\gamma)\cap\Range(\gamma'))\setminus \{x\}\neq\emptyset$. Thus $\LS_{x}\setminus \{x\}$ is connected a.s.
Define the event \[E_{x}:=\{\text{for every }\gamma\in\Xi_{\nu}\setminus\Xi_{x},\text{ we have }
\Range(\gamma)\cap\LS_{x}=\emptyset\}.\]
One the event $E_{x}$,
the set $\C\setminus(\LO_{\kappa,\nu}\cup\{ x\})$
has exactly two connected components,
one containing $\LS_{x}\setminus \{x\}$,
and the other $\C\setminus\overline{\D}$.
So, on the event $E_{x}$,
we have $\Card(\psi^{-1}(\{x\}))=2$.
On the complementary event $E_{x}^{\rm c}$, 
the set $\C\setminus(\LO_{\kappa,\nu}\cup\{ x\})$ is connected and 
$\Card(\psi^{-1}(\{x\}))=1$.
Further, it is easy to see that for
$x\in \Atom_{\rm isol}(\nu)$,
we have $\PP(E_{x})>0$. This completes the proof for~\eqref{eqn::doublepoints2}. 
\end{proof}

\begin{lemma}
\label{Lem Intersect}
Let $\wp_{1}$ and $\wp_{2}$ be two i.i.d. Brownian excursions 
from $-i$ to $i$ in $\D$, 
sampled according to $\mu^{\D,\#}_{-i,i}$.
Then
\begin{displaymath}
\dist(-i,(\Range(\wp_{1})\cap\Range(\wp_{2}))\setminus\{ -i\}) = 0 ~~a.s.
\end{displaymath}
\end{lemma}

\begin{proof}
Because of the conformal covariance, 
one can consider $\tilde{\wp}_{1}$ and $\tilde{\wp}_{2}$
two i.i.d. Brownian excursions from $0$ to $\infty$
in the upper half-plane $\HH$.
Define the random variable
\begin{displaymath}
\tilde{\delta} :=
\dist(0,(\Range(\tilde{\wp}_{1})\cap\Range(\tilde{\wp}_{2}))
\setminus\{ 0\})
\end{displaymath}
with values in $[0,+\infty]$.
It is enough to show that $\tilde{\delta}=0$ a.s.
Since the law of $\tilde{\wp}_{1}$ and $\tilde{\wp}_{2}$
is invariant under Brownian scaling,
we have that for every $c>0$,
$c\tilde{\delta}$ is distributed as $\tilde{\delta}$.
This in turn implies that
$\tilde{\delta}\in\{0,+\infty\}$ a.s.
So we have only to show that
$\PP(\tilde{\delta} = +\infty)=0$.
On the event that 
$\tilde{\delta} = +\infty$, we in particular have that
$\tilde{\wp}_{1}([0,1])\cap\tilde{\wp}_{2}([0,1])
= \{0\}$.
However, 
\begin{displaymath}
\PP\big(\tilde{\wp}_{1}([1,+\infty))\cap\tilde{\wp}_{2}([1,+\infty))
\neq\emptyset \vert
(\tilde{\wp}_{1}(t),\tilde{\wp}_{2}(t))_{0\leq t\leq 1}
\big)
> 0 ~~a.s.
\end{displaymath}
So, if the probability $\PP(\tilde{\delta} = +\infty)$ were positive, then we would have that
$\PP(\tilde{\delta}\in (0,+\infty))>0$,
which is a contradiction.
Thus, $\PP(\tilde{\delta} = +\infty)=0$.
\end{proof}

Next we complement Proposition \ref{Prop double points}
and give a partial results on atoms of $\nu$
which cannot be hit by $\eta_{\kappa,\nu}$.
We will need this result in Section~\ref{Sec Level lines}.
Recall that $\Atom_{\rm conv}^{\rm l}(\nu), \Atom_{\rm conv}^{\rm r}(\nu)$ and $\Atom_{\rm conv}^{\ast}(\nu)$ are defined in~\eqref{eqn::atom_conv_LR} and~\eqref{eqn::atom_conv_star}. 
%\begin{displaymath}
%\Atom_{\rm conv}^{\rm l}(\nu) :=
%\Big\{x\in \Atom(\nu) : 
%\int_{\substack{y\in\overline{A_{\rm L}}
%\\ \Im(y)>\Im(x)}}
%\dfrac{1}
%{\vert y-x\vert^{2}} d\nu(y)<+\infty
%\Big\},
%\end{displaymath}
%\begin{displaymath}
%\Atom_{\rm conv}^{\rm r}(\nu) :=
%\Big\{x\in \Atom(\nu) : 
%\int_{\substack{y\in\overline{A_{\rm L}}
%\\ \Im(y)<\Im(x)}}
%\dfrac{1}
%{\vert y-x\vert^{2}} d\nu(y)<+\infty
%\Big\}.
%\end{displaymath}
%Note that
%\begin{displaymath}
%\Atom_{\rm conv}(\nu)=
%\Atom_{\rm conv}^{\rm l}(\nu)\cap
%\Atom_{\rm conv}^{\rm r}(\nu).
%\end{displaymath}
%Let be
%\begin{displaymath}
%\Atom_{\rm conv}^{\ast}(\nu) :=
%\Atom_{\rm conv}^{\rm l}(\nu)\cup
%\Atom_{\rm conv}^{\rm r}(\nu).
%\end{displaymath}
Note that by construction,
$\Atom(\nu)\cap\{ -i,i\}\subset \Atom_{\rm conv}^{\ast}(\nu)$.

\begin{lemma}
\label{Lem hit}
For every $x\in \Atom(\nu)\setminus \Atom_{\rm conv}^{\ast}(\nu)$,
$\PP (x\in \eta_{\kappa,\nu})=0$.
\end{lemma}

\begin{proof}
Let $x\in \Atom(\nu)\setminus \Atom_{\rm conv}^{\ast}(\nu)$.
Then a.s., there is an excursion
$\gamma\in \Xi_{\nu}$
with one endpoint $x$ and the other endpoint
strictly to the left of $x$,
and an other excursion $\gamma'\in \Xi_{\nu}$
with one endpoint $x$ and the other endpoint
strictly to the right of $x$.
Moreover, by Lemma \ref{Lem Intersect},
$\gamma$ and $\gamma'$ a.s. intersect near $x$.
Thus, $\Range(\gamma)\cup\Range(\gamma')$ disconnects a 
neighborhood of $x$ in $\overline{\D}$
from $A_{\rm R}$.
\end{proof}

\subsection{Local absolute continuity with respect to $\SLE_{\kappa}$ away from the boundary}
\label{SubSec Abs Cont}

In this section we will show that for $\nu\neq 0$,
the curve $\eta_{\kappa,\nu}$
is in some sense absolutely continuous with respect to a
chordal $\SLE_{\kappa}$ away from the boundary.
First we recall the result of
\cite{WernerWuCLEtoSLE}
which identifies the law of $\eta_{\kappa,\nu}$
when $\nu$ is a constant on $A_{\rm L}$.

\begin{theorem}[Werner-Wu \cite{WernerWuCLEtoSLE}]
\label{Thm Werner Wu}
Let $\kappa\in [8/3 , 4]$.
Assume that
$\nu = a\one_{A_{\rm L}}\sigma_{\partial\D}$,
with $a>0$ a constant.
Let $\tilde{\eta}_{\kappa,a}$
denote 
$\tilde{\eta}_{\kappa,\nu}=\psi_{0}(\eta_{\kappa,\nu})$
in this case.
Then $\tilde{\eta}_{\kappa,a}$ is distributed as
a chordal $\SLE_{\kappa}(\rho)$ curve in
$\mathbb{H}$
from $0$ to $\infty$,
with one force point at $0^{-}$,
with $\rho$ be the unique real in 
$(-2,+\infty)$ satisfying
\begin{displaymath}
a^2 = 
\dfrac{\pi}{2}
\dfrac{(\rho+2)(\rho + 6 - \kappa)}{\kappa} .
\end{displaymath}
In particular, if $a=\sqrt{\pi(6-\kappa)/\kappa}$, 
then $\tilde{\eta}_{\kappa,a}$
is distributed as a chordal $\SLE_{\kappa}$.
\end{theorem}

\begin{figure}[ht!]
\begin{center}
\includegraphics[width=0.3\textwidth]{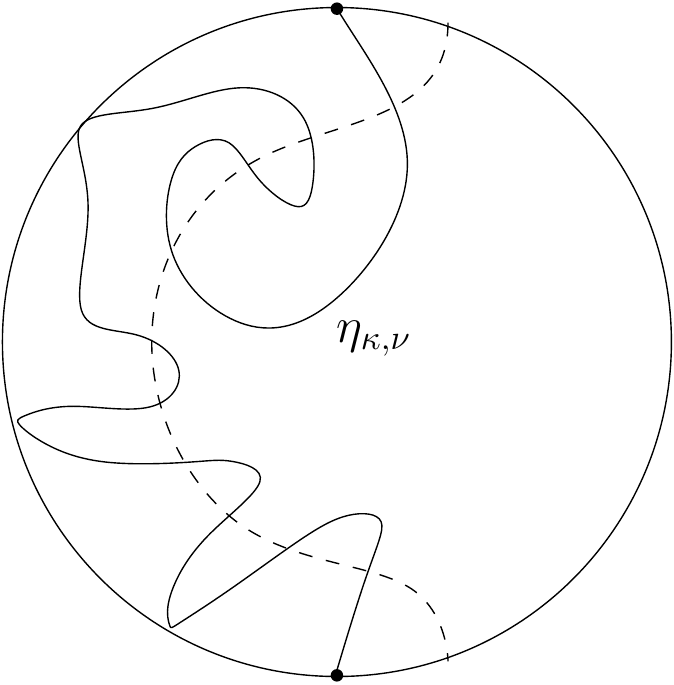}$\quad$
\includegraphics[width=0.3\textwidth]{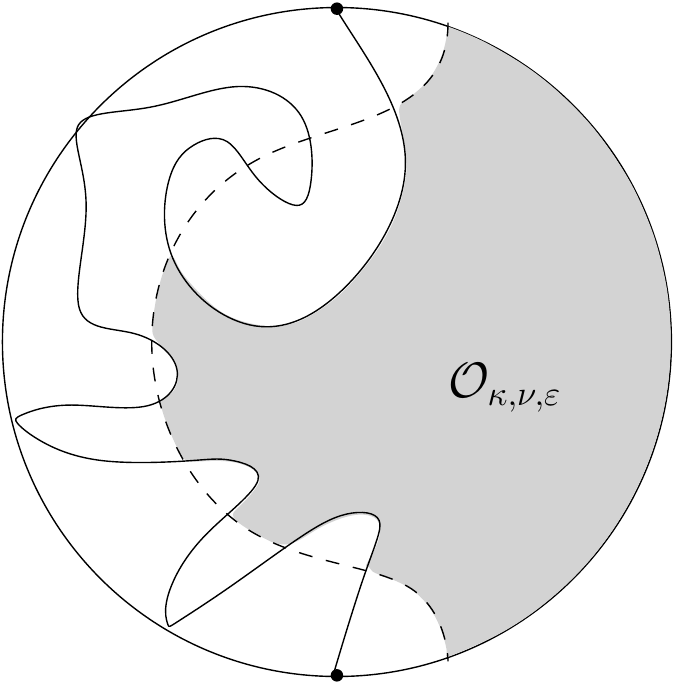}
\end{center}
\caption{\label{fig::absolutecontinuity} In the left panel, the dashed line indicates $\partial\D_{\varepsilon}\cap\D$ and the curve indicates $\eta_{\kappa, \nu}$ intersecting $\D_{\varepsilon}$. The set $\LO_{\kappa,\nu}$ is the connected component of $\D\setminus\eta_{\kappa,\nu}$ adjacent to $A_{\rm R}$. Consider $\LO_{\kappa, \nu}\cap\D_{\varepsilon}$, there are many connected components, and $\LO_{\kappa, \nu,\varepsilon}$ is the one adjacent to $A_{\rm R}\cap\overline{\D_{\varepsilon}}$. In the right panel, the region in gray indicates $\LO_{\kappa, \nu,\varepsilon}$. }
\end{figure}

For $\varepsilon\in (0,1)$, let 
$\LO_{\kappa,\nu,\varepsilon}$ denote the connected component of
$\LO_{\kappa,\nu}\cap\D_{\varepsilon}$
(see \eqref{Eq D eps})
adjacent to $A_{\rm R}\cap\overline{\D_{\varepsilon}}$; see Figure~\ref{fig::absolutecontinuity}. 
To motivate what will follow,
we state the next proposition.

\begin{proposition}
\label{Prop O k nu eps}
Let be $\kappa\in [8/3 , 4]$ and 
$\nu$ a finite non-negative Radon measure on $\overline{A_{\rm L}}$.
Assume that $\nu\neq 0$.
Then a.s., for every
$z\in \eta_{\kappa,\nu}\setminus\overline{A_{\rm L}}$,
there is $U$ a neighborhood of $z$ in
$\eta_{\kappa,\nu}\setminus\overline{A_{\rm L}}$
and $\varepsilon\in (0,1)$ such that
$U\subset \partial \LO_{\kappa,\nu,\varepsilon}$.
\end{proposition}

\begin{proof}
For $w\in \eta_{\kappa,\nu}\setminus\overline{A_{\rm L}}$,
let $I_{w}$ denote the straight line segment in $\overline{\D}$
with endpoints $1$ and
$\psi_{\kappa,\nu}^{-1}(w)$,
where $\psi_{\kappa,\nu}$ is the conformal map
from $\D$ to $\LO_{\kappa,\nu}$ 
defined in Section \ref{Sec Intro}.
Take $z\in\eta_{\kappa,\nu}\setminus\overline{A_{\rm L}}$.
Let $A$ be an open subarc of
$A_{\rm L}$ containing $\psi_{\kappa,\nu}^{-1}(z)$,
such that $\psi_{\kappa,\nu}(\overline{A})\subset 
\eta_{\kappa,\nu}\setminus\overline{A_{\rm L}}$.
Let $K$ be
\begin{displaymath}
K:=\bigcup_{w\in \overline{A}} \psi_{\kappa,\nu}(I_{w}).
\end{displaymath}
$K$ is a connected compact subset of
$\overline{\LO_{\kappa,\nu}}$ containing $1$,
$z$ and a neighborhood of $z$ in 
$\eta_{\kappa,\nu}\setminus\overline{A_{\rm L}}$.
By construction, $K\cap\overline{A_{\rm L}}=\emptyset$.
Thus, for $\varepsilon\in(0,\dist(K, A_{\rm L})\wedge 1)$,
$K\subset\overline{\LO_{\kappa,\nu,\varepsilon}}$.
\end{proof}

Next we state the absolute continuity result.

\begin{proposition}
\label{Prop abs cont}
Let $\kappa\in [8/3 , 4]$ and 
$\nu_{1}, \nu_{2}$ be two finite non-negative Radon measure on 
$\overline{A_{\rm L}}$.
Assume that both $\nu_{1}$ and $\nu_{2}$
are non-zero.
Then, for every $\varepsilon\in (0,1)$,
the laws of $\LO_{\kappa,\nu_{1},\varepsilon}$
and $\LO_{\kappa,\nu_{2},\varepsilon}$
are mutually absolutely continuous.
In particular, if $\eta_{\kappa}$ denotes
a chordal $\SLE_{\kappa}$ curve in $\D$
from $-i$ to $i$ and $\LO_{\kappa,\varepsilon}$
denotes the connected component of $\D\setminus\eta_{\kappa}$
adjacent to $A_{\rm R}\cap\overline{\D_{\varepsilon}}$,
then
for every $\nu$
non-negative non-zero Radon measure on 
$\overline{A_{\rm L}}$,
the law of $\LO_{\kappa,\nu,\varepsilon}$
is absolutely continuous with respect to that
of $\LO_{\kappa,\varepsilon}$.
\end{proposition}

\begin{proof}
The second part of the statement follows from the first part and
Theorem \ref{Thm Werner Wu}.

For the first part, fix $\varepsilon\in (0,1)$.
We assume that $\kappa\neq 8/3$,
the case $\kappa = 8/3$ being simpler.
Let $\FC_{\kappa,\varepsilon}$ be the subset of
$\FC_{\kappa}$ made of $\CLE$ loops
intersecting $\D_{\varepsilon}$.
Let $\tilde{\varepsilon}$ be the following r.v.:
\begin{displaymath}
\tilde{\varepsilon} :=
\inf_{\tilde{\gamma}\in\FC_{\kappa,\varepsilon}}
\dist(\Range(\tilde{\gamma}),A_{\rm L}) .
\end{displaymath}
We have that $\tilde{\varepsilon}\in (0,\varepsilon)$ a.s.
Define
\begin{displaymath}
\widetilde{\Xi}_{\nu_{1},\tilde{\varepsilon}}:=
\left\{(\gamma(T_{\gamma,\tilde{\varepsilon}}^{\rm f}+t))
_{0\leq t\leq 
T_{\gamma,\tilde{\varepsilon}}^{\rm l}
-T_{\gamma,\tilde{\varepsilon}}^{\rm f}}
: \gamma\in \Xi_{\nu_{1}},
\Range(\gamma)\cap \D_{\tilde{\varepsilon}}
\neq \emptyset
\right\},
\end{displaymath}
where $T_{\gamma,\tilde{\varepsilon}}^{\rm f}$
and $T_{\gamma,\tilde{\varepsilon}}^{\rm l}$ are given
by \eqref{Eq Tf Tl},
and where $\Xi_{\nu_{1}}$ is independent from $\FC_{\kappa}$.
Similarly define $\widetilde{\Xi}_{\nu_{2},\tilde{\varepsilon}}$.
We have that the law of
$(\FC_{\kappa},\widetilde{\Xi}_{\nu_{2},\tilde{\varepsilon}})$
is absolutely continuous with respect to that of
$(\FC_{\kappa},\widetilde{\Xi}_{\nu_{1},\tilde{\varepsilon}})$.
This follows from the
Markovian decomposition of Proposition \ref{Prop Markov}.
Further, $\LO_{\kappa,\nu_{1},\varepsilon}$,
respectively $\LO_{\kappa,\nu_{2},\varepsilon}$,
is measurable with respect to
$(\FC_{\kappa},\widetilde{\Xi}_{\nu_{1},\tilde{\varepsilon}})$,
respectively
$(\FC_{\kappa},\widetilde{\Xi}_{\nu_{2},\tilde{\varepsilon}})$.
This concludes the proof.
\end{proof}

\subsection{Curves hitting the boundary with positive measure}
\label{subsec::hittingboundary}
Fix $\kappa\in [8/3, 4]$.
Assume that the measure $\nu$ has full support on
$\overline{A_{\rm L}}$.
By Proposition \ref{Prop contain subarc},
$\eta_{\kappa,\nu}\cap \overline{A_{\rm L}}$
has a.s. empty interior.
However, $\eta_{\kappa,\nu}$ may still hit
$\overline{A_{\rm L}}$, depending on $\nu$.
Here we will show that actually for some $\nu$-s,
$\eta_{\kappa,\nu}\cap \overline{A_{\rm L}}$
may have a.s. empty interior,
yet have, with positive probability, 
a positive mass for the arc-length measure
$\sigma_{\partial\D}$.
We will construct examples with
$\nu$
actually being a continuous
function $u:\overline{A_{\rm L}}\rightarrow [0,+\infty)$.
We will write $\eta_{\kappa,u}$
in this case.

For $k\geq 0$, denote
\begin{displaymath}
Q_{k}
:=
\Big\{
\dfrac{\pi}{2}
+(2j+1)\dfrac{\pi}{2^{k+1}} :
0\leq j\leq 2^{k} -1
\Big\} .
\end{displaymath}
Set
\begin{displaymath}
Q := \bigcup_{k\geq 0} Q_{k}.
\end{displaymath}
Note that $Q_{k}\cap Q_{k'} = \emptyset$
for $k\neq k'$ and that
$Q$ is everywhere dense in 
$\big[\frac{1}{2}\pi, \frac{3}{2}\pi\big]$.
Given
$\varepsilon\in (0,1)$,
denote
$f_{k,\varepsilon}:\big[\frac{1}{2}\pi, \frac{3}{2}\pi\big]
\rightarrow [0,+\infty)$
the following function:
\begin{displaymath}
f_{k,\varepsilon}(x):=
(2^{-(k+1)}\pi - \varepsilon^{-1}\dist(x,Q_{k}))
\vee 0 .
\end{displaymath}
The function $f_{k,\varepsilon}$ is continuous and
bounded from above by $2^{-(k+1)}\pi$.
Moreover, given $\varepsilon\leq\varepsilon'\in (0,1)$,
we have that 
$f_{k,\varepsilon}\leq f_{k,\varepsilon'}$.

Given a sequence
$(\varepsilon_{k})_{k\geq 0}$ in $(0,1)$,
let $f = f_{(\varepsilon_{k})_{k\geq 0}}$
be the following function on 
$\big[\frac{1}{2}\pi, \frac{3}{2}\pi\big]$:
\begin{displaymath}
f = \sum_{k\geq 0}f_{k,\varepsilon_{k}} .
\end{displaymath}
The function $f$ is non-negative,
continuous, and positive on $Q$
whatever the choice of $(\varepsilon_{k})_{k\geq 0}$.
Moreover, $f\leq \pi$.
Let $u=u_{(\varepsilon_{k})_{k\geq 0}}$
be the function on $\overline{A_{\rm L}}$
defined by
$u(e^{i\theta}) = f(\theta)$
for $\theta\in\big[\frac{1}{2}\pi, \frac{3}{2}\pi\big]$.

\begin{proposition}
\label{Prop pos meas}
Fix $\kappa\in[8/3, 4]$.
There is a sequence
$(\varepsilon_{k})_{k\geq 0}$ in $(0,1)$
such that
\begin{displaymath}
\PP(\sigma_{\partial\D}(\eta_{\kappa,u}
\cap\overline{A_{\rm L}}) >0) > 0,
\end{displaymath}
where $u=u_{(\varepsilon_{k})_{k\geq 0}}$.
\end{proposition}

\begin{proof}
First note that given the measurable functions
$v : \overline{A_{\rm L}}\rightarrow [0,\pi]$,
the Poisson point processes of excursions
$\Xi_{v}$ are all naturally coupled on the same probability space. First one takes $\Xi_{\pi}$, which contains countably many excursions. For each $\gamma\in \Xi_{\pi}$,
one takes two i.i.d. random variables $U_{\gamma}^{\rm f}$ and 
$U_{\gamma}^{\rm l}$,
uniform in $(0,1)$.
Given $v : \overline{A_{\rm L}}\rightarrow [0,\pi]$,
one gets $\Xi_{v}$ by keeping an excursion
$\gamma\in \Xi_{\pi}$ if
\begin{displaymath}
\dfrac{1}{\pi}v(\gamma(0))>U_{\gamma}^{\rm f}
\qquad
\text{and}
\qquad
\dfrac{1}{\pi}v(\gamma(T_{\gamma}))>U_{\gamma}^{\rm l} .
\end{displaymath}
In this way, $\Xi_{v}\subset \Xi_{\pi}$ a.s.
This coupling of the $\Xi_{v}$-s
induces a coupling of the curves
$\eta_{\kappa,v}$, 
by taking the same $\CLE_{\kappa}$
for different $v$-s.
We will further consider this coupling.

For $k\geq 0$ and
$\varepsilon_{0},\dots,\varepsilon_{k}\in (0,1)$,
denote
\begin{displaymath}
f_{\varepsilon_{0},\dots,\varepsilon_{k}}
:=
\sum_{j=0}^{k} f_{k,\varepsilon_{j}},
\end{displaymath}
and define
$u_{\varepsilon_{0},\dots,\varepsilon_{k}}$
by 
$u_{\varepsilon_{0},\dots,\varepsilon_{k}}(e^{i\theta}) = f_{\varepsilon_{0},\dots,\varepsilon_{k}}(\theta)$.

Given a sequence $(\varepsilon_{k})_{k\geq 0}$ in $(0,1)$,
we have that
\begin{displaymath}
\Xi_{u_{(\varepsilon_{k})_{k\geq 0}}}
=\bigcup_{k\geq 0}
\Xi_{u_{\varepsilon_{0},\dots,\varepsilon_{k}}}
\qquad \text{a.s.},
\end{displaymath}
and
\begin{displaymath}
\eta_{\kappa,u_{(\varepsilon_{k})_{k\geq 0}}}
\cap\overline{A_{\rm L}}
=
\bigcap_{k\geq 0}
(\eta_{\kappa,u_{\varepsilon_{0},\dots,\varepsilon_{k}}}
\cap\overline{A_{\rm L}})
\qquad \text{a.s.},
\end{displaymath}
where the last intersection is non-increasing.
In particular,
\begin{displaymath}
\sigma_{\partial\D}
(\eta_{\kappa,u_{(\varepsilon_{k})_{k\geq 0}}}
\cap\overline{A_{\rm L}})
=
\lim_{k\to +\infty}
\sigma_{\partial\D}
(\eta_{\kappa,u_{\varepsilon_{0},\dots,\varepsilon_{k}}}
\cap\overline{A_{\rm L}})
\qquad
\text{a.s.}
\end{displaymath}
In particular, for any
$\delta>0$,
\begin{displaymath}
\PP(
\sigma_{\partial\D}
(\eta_{\kappa,u_{(\varepsilon_{k})_{k\geq 0}}}
\cap\overline{A_{\rm L}})\geq \delta
)
=
\lim_{k\to +\infty}
\PP(
\sigma_{\partial\D}
(\eta_{\kappa,u_{\varepsilon_{0},\dots,\varepsilon_{k}}}
\cap\overline{A_{\rm L}})\geq \delta
).
\end{displaymath}

Further, fix $k\geq 1$.
Consider the values of $\varepsilon_{k}$
of form $2^{-n}$.
We have that
\begin{displaymath}
\bigcup_{n\geq 1}
(\eta_{\kappa,u_{\varepsilon_{0},\dots,\varepsilon_{k-1},
2^{-n}}}
\cap\overline{A_{\rm L}})
=
(\eta_{\kappa,u_{\varepsilon_{0},\dots,\varepsilon_{k-1}}}
\cap\overline{A_{\rm L}})
\setminus Q_{k}
\qquad
\text{a.s.}
\end{displaymath}
Therefore,
\begin{displaymath}
\lim_{\varepsilon_{k}\to 0}
\sigma_{\partial\D}
(\eta_{\kappa,u_{\varepsilon_{0},\dots,\varepsilon_{k-1},
\varepsilon_{k}}}
\cap\overline{A_{\rm L}})
=
\sigma_{\partial\D}
(\eta_{\kappa,u_{\varepsilon_{0},\dots,\varepsilon_{k-1}}}
\cap\overline{A_{\rm L}})
\qquad
\text{a.s.}
\end{displaymath}
Thus, for every $\delta>0$,
\begin{equation}
\label{Eq lim 1}
\lim_{\varepsilon_{k}\to 0}
\PP(
\sigma_{\partial\D}
(\eta_{\kappa,u_{\varepsilon_{0},\dots,\varepsilon_{k-1},
\varepsilon_{k}}}
\cap\overline{A_{\rm L}})
> \delta
)
=
\PP(
\sigma_{\partial\D}
(\eta_{\kappa,u_{\varepsilon_{0},\dots,\varepsilon_{k-1}}}
\cap\overline{A_{\rm L}})
> \delta
).
\end{equation}
Similarly,
\begin{displaymath}
\lim_{\varepsilon_{0}\to 0}
\sigma_{\partial\D}
(\eta_{\kappa,u_{\varepsilon_{0}}}
\cap\overline{A_{\rm L}})
=
\sigma_{\partial\D}(\overline{A_{\rm L}})
=
\pi
\qquad
\text{a.s.},
\end{displaymath}
and for every $\delta\in (0,\pi)$,
\begin{equation}
\label{Eq lim 2}
\lim_{\varepsilon_{0}\to 0}
\PP(\sigma_{\partial\D}
(\eta_{\kappa,u_{\varepsilon_{0}}}
\cap\overline{A_{\rm L}})
>\delta) = 1.
\end{equation}
Therefore,
\eqref{Eq lim 1} and \eqref{Eq lim 2}
ensure that
one can choose the sequence
$(\varepsilon_{k})_{k\geq 0}$ in $(0,1)$
such that for every $k\geq 0$,
\begin{displaymath}
\PP\Big(
\sigma_{\partial\D}
(\eta_{\kappa,u_{\varepsilon_{0},\dots,\varepsilon_{k}}}
\cap\overline{A_{\rm L}})
\geq \Big(\dfrac{1}{2}+\dfrac{1}{2^{k+1}}\Big)\pi
\Big)
\geq 
\dfrac{1}{2}+\dfrac{1}{2^{k+1}} .
\end{displaymath}
For such a sequence,
\begin{displaymath}
\PP\Big(
\sigma_{\partial\D}
(\eta_{\kappa,u_{(\varepsilon_{k})_{k\geq 0}}}
\cap\overline{A_{\rm L}})\geq \dfrac{\pi}{2}
\Big)
\geq\dfrac{1}{2}.
\qedhere
\end{displaymath}
\end{proof}

\section{Continuous dependence on boundary conditions}
\label{Sec Continuity}

\subsection{Continuous dependence of the Poisson point process of excursions}
\label{SubSec Convergence exc}

In this section, we deal with the continuity in $\nu$ of the Poisson point process 
$\Xi_{\nu}$.
Suppose $S_{1}$ and $S_{2}$ are two finite sets
of continuous paths $(\gamma(t))_{0\leq t\leq T_{\gamma}}$
in $\C$ with $T_{\gamma}<+\infty$. We define 
\begin{align*}
& d_{\rm curves}(S_{1},S_{2}):=
\\
& \begin{cases}
\min\limits_{\sigma \in \Bij (S_{1},S_{2})}
\sum\limits_{\gamma\in S_{1}}
\left(
\vert T_{\gamma} - T_{\sigma(\gamma)}\vert
+
\max_{s\in[0,1]}
\vert \gamma(sT_{\gamma}) - \sigma(\gamma)(sT_{\sigma(\gamma)})\vert
\right),
& \text{ if } \Card (S_{1})=\Card (S_{2}), \\ 
+\infty, & \text{ if } \Card (S_{1})\neq\Card (S_{2}).
\end{cases}
\end{align*}
Note that $d_{\rm curves}$ is a distance.
By definition, the distance of the empty set to any non-empty set
is $+\infty$.

In the following, we will consider distance between Poisson point processes. Although the Poisson point process $\Xi_{\nu}$ contains infinitely many excursions, its cutoff is finite~\eqref{Eq exc eps}. In this section, we will consider the following three types of cutoff. 
Recall from~\eqref{Eq exc eps} that 
\[\Xi_{\nu,\varepsilon}:=
\{\gamma\in\Xi_{\nu} : \diam \Range(\gamma)>\varepsilon\}.
\]
We also define, for $\varepsilon>0$, 
\begin{align*}
\widehat{\Xi}_{\nu,\varepsilon}
:=
\left\{\gamma\in \Xi_{\nu} :
\Range(\gamma)\cap \D_{\varepsilon}\neq\emptyset
\right\}, \qquad 
\widetilde{\Xi}_{\nu,\varepsilon}:=
\left\{(\gamma(T_{\gamma,\varepsilon}^{\rm f}+t))
_{0\leq t\leq 
T_{\gamma,\varepsilon}^{\rm l}-T_{\gamma,\varepsilon}^{\rm f}}
: \gamma\in \widehat{\Xi}_{\nu,\varepsilon}
\right\},
\end{align*}
where $\D_{\varepsilon}$ is given by \eqref{Eq D eps},
and $T_{\gamma,\varepsilon}^{\rm f}$
and $T_{\gamma,\varepsilon}^{\rm l}$
by \eqref{Eq Tf Tl}. 

\begin{proposition}
\label{Prop Coupling Exc}
Fix $\nu$ a finite non-negative Radon measure on $\overline{A_L}$. Let $(\nu_{n})_{n\geq 0}$ be a sequence of finite non-negative Radon measures on $\overline{A_{\rm L}}$, 
converging weakly to $\nu$. 
Then, for every $\varepsilon>0$,
$(\Xi_{\nu_{n},\varepsilon})_{n\geq 0}$ 
converges in law to $\Xi_{\nu,\varepsilon}$
 for $d_{\rm curves}$.
Moreover, it is possible to couple on the same probability space all the processes
$(\Xi_{\nu_{n}})_{n\geq 0}$ 
and $\Xi_{\nu}$
such that the following two conditions hold a.s.
\begin{enumerate}[label=(\arabic*)]
\item \label{item::couplingexc1}
For every $\varepsilon\in (0,1)$,
$\lim_{n\to +\infty}d_{\rm curves}
(\Xi_{\nu_{n},\varepsilon},\Xi_{\nu,\varepsilon})=0$.
\item \label{item::couplingexc2}
For every $\varepsilon\in (0,1)$,
there is $n_{\varepsilon}\in\N$, 
such that 
$\widetilde{\Xi}_{\nu_{n},\varepsilon}
=\widetilde{\Xi}_{\nu,\varepsilon}$ for every $n\geq n_{\varepsilon}$.
\end{enumerate}
\end{proposition}

The proof of Proposition~\ref{Prop Coupling Exc} will be split into several lemmas. In the rest of this section, we fix the following assumptions: \textit{Fix $\nu$ a finite non-negative Radon measure on $\overline{A_L}$. Let $(\nu_{n})_{n\geq 0}$ be a sequence of finite non-negative Radon measures on $\overline{A_{\rm L}}$, 
converging weakly to $\nu$. }

\begin{lemma}
\label{Lem TV Poisson}
Fix $\varepsilon\in (0,1)$.
\begin{enumerate}[label=(\arabic*)]
\item If $\nu\neq 0$,
then for every $n\geq 0$,
the law of $\widetilde{\Xi}_{\nu_{n},\varepsilon}$
is absolutely continuous with respect to that of
$\widetilde{\Xi}_{\nu,\varepsilon}$.
Moreover, the corresponding density
$Y_{\varepsilon,\nu,\nu_{n}}$,
converges a.s. to $1$
as $n\to +\infty$.
\item If $\nu=0$, then 
\begin{equation}
\label{Eq empty}
\lim_{n\to +\infty}
\PP(\widetilde{\Xi}_{\nu_{n},\varepsilon}=\emptyset)=1.
\end{equation}
\end{enumerate}
\end{lemma}

\begin{proof}
Both $\widetilde{\Xi}_{\nu_{n},\varepsilon}$ and 
$\widetilde{\Xi}_{\nu,\varepsilon}$ are
a.s. finite Poisson point processes.
According to Proposition~\ref{Prop Markov},
the intensity measure of $\widetilde{\Xi}_{\nu,\varepsilon}$ is
\begin{displaymath}
\dfrac{1}{2}\iint
\limits_{(\partial\widehat{\D}_{\varepsilon}\cap\D)^{2}}
\sigma_{\partial\widehat{\D}_{\varepsilon}}(dz)
\sigma_{\partial\widehat{\D}_{\varepsilon}}(dw)
\mu_{z,w}^{\D}
\iint\limits_{\overline{A_{\rm L}}\times\overline{A_{\rm L}}}
H_{\widehat{\D}_{\varepsilon}}(x,z)
H_{\widehat{\D}_{\varepsilon}}(w,y)
d(\nu\otimes\nu)(x,y).
\end{displaymath}
The intensity measure for $\widetilde{\Xi}_{\nu_{n},\varepsilon}$
has same expression, with $\nu_{n}$ instead of $\nu$.

If $\nu\neq 0$, then the intensity measure for $\nu_{n}$
is absolutely continuous with respect to that for $\nu$,
both being absolutely continuous with respect to
\begin{equation}
\label{Eq ref meas}
\iint
\limits_{(\partial\widehat{\D}_{\varepsilon}\cap\D)^{2}}
\sigma_{\partial\widehat{\D}_{\varepsilon}}(dz)
\sigma_{\partial\widehat{\D}_{\varepsilon}}(dw)
\mu_{z,w}^{\D}.
\end{equation}
The density from $\nu$ to $\nu_{n}$ is
\begin{displaymath}
\dfrac{\iint\limits_{\overline{A_{\rm L}}\times\overline{A_{\rm L}}}
H_{\widehat{\D}_{\varepsilon}}(x,z)
H_{\widehat{\D}_{\varepsilon}}(w,y)
d(\nu_{n}\otimes\nu_{n})(x,y)}
{\iint\limits_{\overline{A_{\rm L}}\times\overline{A_{\rm L}}}
H_{\widehat{\D}_{\varepsilon}}(x,z)
H_{\widehat{\D}_{\varepsilon}}(w,y)
d(\nu\otimes\nu)(x,y)},
\end{displaymath}
where $z,w\in\partial\widehat{\D}_{\varepsilon}\cap\D$
are the two endpoints of the path.
This density converges to $1$ almost everywhere
and in $\mathbb{L}^{1}$.
This follows from the weak convergence of $(\nu_{n})_{n\geq 0}$ to $\nu$ together with the continuity of the boundary Poisson kernel 
$H_{\widehat{\D}_{\varepsilon}}$.
Since the Poisson point processes are a.s. finite, 
this implies the absolute continuity of Poisson point processes
and the a.s. convergence of the density to $1$.

If $\nu = 0$, then the total mass of the intensity for
$\nu_{n}$ converges to $0$,
which implies \eqref{Eq empty}.
\end{proof}

\begin{lemma}
\label{Lem abstract}
Let $S$ be an abstract Polish space and $\LB$ be its
Borel $\sigma$-algebra.
Let $X$ and $X_{n}$, for $n\geq 0$,
be random variables taking values in $(S,\LB)$.
Assume that for every $n\geq 0$, the law of
$X_{n}$ is absolutely continuous with respect to that of
$X$, with density denoted by $Y_{n}$.
Assume moreover that $(Y_{n})_{n\geq 0}$ converges
$d\PP_{X}$-a.s. to $1$ as $n\to +\infty$.
Then it is possibles to couple
$X$ and all $X_{n}$ for $n\geq 0$
on the same probability space such that
a.s. $X_{n}=X$ for every $n$ large enough.
\end{lemma}

\begin{proof}
Note that the sequence $(Y_{n})_{n\geq 0}$
is naturally defined on the same probability space as $X$ and 
is measurable with respect to $X$.
For $n\geq 0$ such that $\PP(Y_{n}> 1)>0$,
let $\widetilde{X}_{n}$ be a random variable taking values in
$(S,\LB)$,
with density
\begin{displaymath}
\dfrac{(Y_{n} - 1)_{+}}{\E[(Y_{n} - 1)_{+}]}
\end{displaymath}
with respect to $X$.
We also take $X$ and all the $\widetilde{X}_{n}$
to be independent.
Let $U$ be a uniform random variable on $(0,1)$,
independent from $(X,(\widetilde{X}_{n})_{n\geq 0})$.
We construct the sequence $(\widehat{X}_{n})_{n\geq 0}$
as follows.
On the event $\{Y_{n}\geq U\}$,
we set $\widehat{X}_{n}=X$.
On the event $\{Y_{n}< U\}$, 
we set $\widehat{X}_{n}=\widetilde{X}_{n}$.
It is easy to check that for every
$n\geq 0$,
$\widehat{X}_{n}$
has same distribution as $X_{n}$.
Moreover, a.s. for every $n$ large enough,
$Y_{n}\geq U$ and $\widehat{X}_{n}=X$.
\end{proof}

\begin{lemma}
\label{Lem TV approx}
It is possible to couple
$(\widehat{\Xi}_{\nu_{n},\varepsilon})_{n\geq 0}$
and $\widehat{\Xi}_{\nu,\varepsilon}$ on the same probability space
such that the following conditions hold a.s.
\begin{enumerate}[label=(\arabic*)]
\item \label{item::TVapprox1}
$\lim_{n\to +\infty}d_{\rm curves}
(\widehat{\Xi}_{\nu_{n},\varepsilon},
\widehat{\Xi}_{\nu,\varepsilon})=0$.
\item \label{item::TVapprox2}
For every $n$ large enough, 
$\widetilde{\Xi}_{\nu_{n},\varepsilon}
=\widetilde{\Xi}_{\nu,\varepsilon}$.
\end{enumerate}
\end{lemma}

\begin{proof}
We can assume that $\nu\neq 0$.
The case $\nu=0$ is trivial by \eqref{Eq empty}.
The fact that there is a coupling such that the condition~\ref{item::TVapprox2} is satisfied
follows from Lemmas~\ref{Lem TV Poisson} and~\ref{Lem abstract}.
It remains to couple the slices
\begin{displaymath}
\left\{(\gamma(t))_{0\leq t\leq T_{\gamma,\varepsilon}^{\rm f}}
: \gamma\in \widehat{\Xi}_{\nu_{n},\varepsilon}\right\},
\qquad
\left\{(\gamma(T_{\gamma,\varepsilon}^{\rm l}+t))
_{0\leq t\leq T_{\gamma} - T_{\gamma,\varepsilon}^{\rm l}}
: \gamma\in \widehat{\Xi}_{\nu_{n},\varepsilon}\right\},
\end{displaymath}
and
\begin{displaymath}
\left\{(\gamma(t))_{0\leq t\leq T_{\gamma,\varepsilon}^{\rm f}}
: \gamma\in \widehat{\Xi}_{\nu,\varepsilon}\right\},
\qquad
\left\{(\gamma(T_{\gamma,\varepsilon}^{\rm l}+t))
_{0\leq t\leq T_{\gamma} - T_{\gamma,\varepsilon}^{\rm l}}
: \gamma\in \widehat{\Xi}_{\nu,\varepsilon}\right\},
\end{displaymath}
in a way that the condition~\ref{item::TVapprox1} holds.
We only construct coupling of the slices $(\gamma(t))_{0\leq t\leq T_{\gamma,\varepsilon}^{\rm f}}$. The coupling for the slices
$(\gamma(T_{\gamma,\varepsilon}^{\rm l}+t))
_{0\leq t\leq T_{\gamma}-T_{\gamma,\varepsilon}^{\rm l}}$
can be obtained similarly.

According to Proposition~\ref{Prop Markov},
given $\gamma\in\widehat{\Xi}_{\nu,\varepsilon}$,
conditionally on 
$(\gamma(T_{\gamma,\varepsilon}^{\rm f}),
\gamma(T_{\gamma,\varepsilon}^{\rm l}))$,
the three slices
$(\gamma(t))_{0\leq t\leq T_{\gamma,\varepsilon}^{\rm f}}$,
$(\gamma(T_{\gamma,\varepsilon}^{\rm f}+t))
_{0\leq t
\leq T_{\gamma,\varepsilon}^{\rm l}-T_{\gamma,\varepsilon}^{\rm f}}$
and 
$(\gamma(T_{\gamma,\varepsilon}^{\rm l}+t))
_{0\leq t\leq T_{\gamma}-T_{\gamma,\varepsilon}^{\rm l}}$
are independent.
The conditional distribution of 
$(\gamma(t))_{0\leq t\leq T_{\gamma,\varepsilon}^{\rm f}}$ is
\begin{displaymath}
\left(
\int_{\overline{A_{\rm L}}}
d\nu(x)
H_{\widehat{\D}_{\varepsilon}}(x,\gamma(T_{\gamma,\varepsilon}^{\rm f}))
\right)^{-1}
\int_{\overline{A_{\rm L}}}
d\nu(x)\mu^{\widehat{\D}_{\varepsilon}}
_{x,\gamma(T_{\gamma,\varepsilon}^{\rm f})}. 
\end{displaymath}
In the case of $\widehat{\Xi}_{\nu_{n},\varepsilon}$ the distribution is the same,
with $\nu_{n}$ instead of $\nu$.

Given $\theta \in [\frac{1}{2}\pi,\frac{3}{2}\pi]$,
let $A[i,e^{i\theta}]$ denote the closed subarc of
$\overline{A_{\rm L}}$ with endpoints $i$ and $e^{i\theta}$.
Given $z\in\partial\widehat{\D}_{\varepsilon}\cap\D$,
let $\vartheta_{z,\nu}$ be the following function from
$[0,1]$ to $[\frac{1}{2}\pi,\frac{3}{2}\pi]$: 
\begin{displaymath}
\vartheta_{z,\nu}(u):=
\inf\left\{
\theta\in \left[\frac{1}{2}\pi,\frac{3}{2}\pi\right] :
\int_{A[i,e^{i\theta}]}
d\nu(x)
H_{\widehat{\D}_{\varepsilon}}(x,z)
\geq
u\int_{\overline{A_{\rm L}}}
d\nu(x)
H_{\widehat{\D}_{\varepsilon}}(x,z)
\right\}.
\end{displaymath}
Suppose $U$ is a uniform random variable on $(0,1)$,
then $e^{i\vartheta_{z,\nu}(U)}$ has the distribution
\begin{displaymath}
\left(
\int_{\overline{A_{\rm L}}}
d\nu(\tilde{x})
H_{\widehat{\D}_{\varepsilon}}(\tilde{x},z)
\right)^{-1}
H_{\widehat{\D}_{\varepsilon}}(x,z)d\nu(x).
\end{displaymath}
The functions $\vartheta_{z,\nu_{n}}$ are defined similarly,
with $\nu_{n}$ instead of $\nu$.
We have that $(\vartheta_{z,\nu_{n}}(U))_{n\geq 0}$
converges a.s. to $\vartheta_{z,\nu}(U)$.

Given $x\in\overline{A_{\rm L}}$ and
$z\in \partial\widehat{\D}_{\varepsilon}\cap\D$,
let $\psi_{x,z}$ be the conformal map from
$\D$ to $\widehat{\D}_{\varepsilon}$,
uniquely defined by
\begin{displaymath}
\psi_{x,z}(-i)=x,
\qquad
\psi_{x,z}(i)=z,
\qquad
\vert\psi_{x,z}'(-i)\vert=1.
\end{displaymath}
Given $(\wp(t))_{0\leq t\leq T_{\wp}}$
a continuous curve in $\overline{\D}$ from $-i$ to $i$,
let $\LT_{x,z}(\wp)$ denote the continuous curve
in $\widehat{\D}_{\varepsilon}$ from $x$ to $z$
obtained by applying to the curve $\wp$ the conformal map
$\psi_{x,z}$ and the change of time
$ds =\vert \psi_{x,z}'(\wp(t))\vert^{2} dt$.
The image of the normalized excursion probability measure
$\mu^{\D,\#}_{-i,i}$ under the map $\LT_{x,z}$ 
is the normalized excursion probability measure
$\mu^{\widehat{\D}_{\varepsilon},\#}_{x,z}$.

Now fix $z\in\partial\widehat{\D}_{\varepsilon}\cap\D$.
Let $\wp$ be a Brownian excursion from $-i$ to $i$ in
$\D$, sampled according to $\mu^{\D,\#}_{-i,i}$,
and let $U$ be an independent random variable uniform on $(0,1)$.
Then the random curve
$\LT_{e^{i\vartheta_{z,\nu}(U)},z}(\wp)$
is distributed according to the probability measure
\begin{displaymath}
\left(
\int_{\overline{A_{\rm L}}}
d\nu(x)
H_{\widehat{\D}_{\varepsilon}}(x,z)
\right)^{-1}
\int_{\overline{A_{\rm L}}}
d\nu(x)\mu^{\widehat{\D}_{\varepsilon}}
_{x,z}.
\end{displaymath}
The curve $\LT_{e^{i\vartheta_{z,\nu_{n}}(U)},z}(\wp)$
has a similar distribution, with $\nu_{n}$ instead of $\nu$.
Moreover, as $n\to +\infty$,
the sequence 
$(\LT_{e^{i\vartheta_{z,\nu_{n}}(U)},z}(\wp))_{n\geq 0}$
converges a.s. to
$\LT_{e^{i\vartheta_{z,\nu_{n}}(U)},z}(\wp)$.
So, this construction provides a way to couple
the slices
$(\gamma(t))_{0\leq t\leq T_{\gamma,\varepsilon}^{\rm f}}$
for $\gamma\in \widehat{\Xi}_{\nu,\varepsilon}$,
respectively $\gamma\in\widehat{\Xi}_{\nu_{n},\varepsilon}$,
so that the a.s. convergence holds.
\end{proof}

\begin{lemma}
\label{Lem diam level}
Fix $\tilde{\varepsilon}>0$.
Then 
\begin{displaymath}
\lim_{\varepsilon\to 0}
\PP(
\Xi_{\nu,\tilde{\varepsilon}}\setminus
\widehat{\Xi}_{\nu,\varepsilon}
\neq\emptyset)=0,
\qquad
\lim_{\varepsilon\to 0}
\sup_{n\geq 0}
\PP(
\Xi_{\nu_{n},\tilde{\varepsilon}}\setminus
\widehat{\Xi}_{\nu_{n},\varepsilon}
\neq\emptyset)=0.
\end{displaymath}
\end{lemma}

\begin{proof}
The Poisson point processes
$\Xi_{\nu,\tilde{\varepsilon}}\setminus
\widehat{\Xi}_{\nu,\varepsilon}$ and
$\Xi_{\nu_{n},\tilde{\varepsilon}}\setminus
\widehat{\Xi}_{\nu_{n},\varepsilon}$
consist precisely of excursions of diameter greater than
$\tilde{\varepsilon}$,
but that do not visit $\D_{\varepsilon}$.
It is easy to see that
\[
\lim_{\varepsilon\to 0}
\sup_{x,y\in \overline{A_{\rm L}}}
\mu^{\D}_{x,y}
(
\{
\gamma :
\diam(\gamma)>\tilde{\varepsilon},
\Range(\gamma)\cap \D_{\varepsilon} = \emptyset
\}
)=0.
\] 
The conclusion follows. 
\end{proof}

Now we are ready to complete the proof of Proposition~\ref{Prop Coupling Exc}. 
\begin{proof}[Proof of Proposition~\ref{Prop Coupling Exc}]
According to Lemma \ref{Lem diam level},
for every $k\geq 1$,
there is $\varepsilon_{k}\in (0, 2^{-k}]$
such that
\begin{displaymath}
\PP\left(
\Xi_{\nu,2^{-k}}\setminus
\widehat{\Xi}_{\nu,\varepsilon_{k}}
\neq\emptyset\right)\leq 2^{-k},
\qquad
\sup_{n\geq 0}
\PP\left(
\Xi_{\nu_{n},2^{-k}}\setminus
\widehat{\Xi}_{\nu_{n},\varepsilon_{k}}
\neq\emptyset\right)\leq 2^{-k}.
\end{displaymath}
We may also take the sequence $(\varepsilon_{k})_{k\geq 1}$
to be non-increasing.

According to Lemma \ref{Lem TV approx},
for every $k\geq 1$,
there is a coupling on the same probability space
of $\Xi^{(k)}_{\nu}$ and $\Xi^{(k)}_{\nu_{n}}$
for $n\geq 0$,
with $\Xi^{(k)}_{\nu}$ distributed as $\Xi_{\nu}$
and $\Xi^{(k)}_{\nu_{n}}$ distributed as $\Xi_{\nu_{n}}$,
such that a.s., 
$\lim_{n\to +\infty}d_{\rm curves}
(\widehat{\Xi}^{(k)}_{\nu_{n},\varepsilon_{k}},
\widehat{\Xi}^{(k)}_{\nu,\varepsilon_{k}})=0$
and
for every $n$ large enough, 
$\widetilde{\Xi}^{(k)}_{\nu_{n},\varepsilon_{k}}
=\widetilde{\Xi}^{(k)}_{\nu,\varepsilon_{k}}$.
Moreover, since
$\Xi^{(k)}_{\nu}\setminus \widehat{\Xi}^{(k)}_{\nu,\varepsilon_{k}}$
is independent from 
$\widehat{\Xi}^{(k)}_{\nu,\varepsilon_{k}}$
and
$\Xi^{(k)}_{\nu_{n}}\setminus
\widehat{\Xi}^{(k)}_{\nu_{n}\varepsilon_{k}}$
is independent from 
$\widehat{\Xi}^{(k)}_{\nu_{n},\varepsilon_{k}}$,
one can further require that for every $k\geq 1$,
\begin{align*}
&\PP\left(
\Xi^{(k)}_{\nu,2^{-k}}\setminus
\widehat{\Xi}^{(k)}_{\nu,\varepsilon_{k}}
\neq\emptyset
\text{ or }
\exists n\geq 0,
\Xi^{(k)}_{\nu_{n},2^{-k}}\setminus
\widehat{\Xi}^{(k)}_{\nu_{n},\varepsilon_{k}}
\neq\emptyset
\right)\\
& \leq 
\PP\left(
\Xi^{(k)}_{\nu,2^{-k}}\setminus
\widehat{\Xi}^{(k)}_{\nu,\varepsilon_{k}}
\neq\emptyset\right)
\vee
\sup_{n\geq 0}
\PP\left(
\Xi^{(k)}_{\nu_{n},2^{-k}}\setminus
\widehat{\Xi}^{(k)}_{\nu_{n},\varepsilon_{k}}
\neq\emptyset\right)
\leq 2^{-k}.
\end{align*}

By considering the conditional law of
the sequence $(\Xi^{(k)}_{\nu_{n}})_{n\geq 0}$
given $\Xi^{(k)}_{\nu}$,
one can further couple all the
$\Xi^{(k)}_{\nu}$ and $\Xi^{(k)}_{\nu_{n}}$
for $n\geq 0$ and $k\geq 1$ on the same probability space such that
the Poisson point processes
$\Xi^{(k)}_{\nu}$ are a.s. all the same for different values of $k$.
Will denote by $\Xi^{\ast}_{\nu}$ their common value.
It is distributed as $\Xi_{\nu}$.

Set $N_{1}:=0$,
and for $k\geq 2$,
\begin{displaymath}
N_{k}:=
\min
\left\{
N>N_{k-1} : 
\PP\left(
\exists n\geq N,
d_{\rm curves}
(\widehat{\Xi}^{(k)}_{\nu_{n},\varepsilon_{k}},
\widehat{\Xi}^{(k)}_{\nu,\varepsilon_{k}})
> 2^{-k}
\text{ or }
\widetilde{\Xi}^{(k)}_{\nu_{n},\varepsilon_{k}}
\neq\widetilde{\Xi}^{\ast}_{\nu,\varepsilon_{k}}
\right)
\leq 2^{-k}
\right\}.
\end{displaymath}
We would like to emphasize that the sequence
$(N_{k})_{k\geq 1}$ is deterministic.
We define the sequence
$(\Xi^{\ast}_{\nu_{n}})_{n\geq 1}$
as follows.
Given $n\geq 0$,
there is a unique $k\geq 1$
such that $N_{k}\leq n < N_{k+1}$, and we set
$\Xi^{\ast}_{\nu_{n}}=\Xi^{(k)}_{\nu_{n}}$.
For every $n\geq 0$,
$\Xi^{\ast}_{\nu_{n}}$ is distributed as $\Xi_{\nu_{n}}$.

For $j\geq 1$ and $k\geq j$, let
$E_{j,k}$ denote the event that there is
$n\in\{ N_{k},\dots , N_{k+1}-1\}$
such that 
$\widetilde{\Xi}^{(k)}_{\nu_{n},\varepsilon_{j}}
\neq\widetilde{\Xi}^{\ast}_{\nu,\varepsilon_{j}}$.
By construction, 
for every $k\geq 1$
$\PP(E_{k,k})\leq 2^{-k}$.
Moreover,
$E_{j,k}\subset E_{k,k}$ for $j\leq k$.
Thus, for every $j\geq 1$,
\begin{displaymath}
\sum_{k\geq j} \PP(E_{j,k}) < +\infty .
\end{displaymath}
By Borel-Cantelli lemma, this means that
a.s., the events $E_{j,k}$ occur for only finitely many values of $k$.
Thus, a.s.
for every $n$ large enough, 
$\widetilde{\Xi}^{\ast}_{\nu_{n},\varepsilon_{j}}
=\widetilde{\Xi}^{\ast}_{\nu,\varepsilon_{j}}$.

Similarly, by using the Borel-Cantelli lemma,
we get that for every $j\geq 1$,
a.s. 
\begin{displaymath}
\lim_{n\to +\infty}d_{\rm curves}
\left(\widehat{\Xi}^{\ast}_{\nu_{n},\varepsilon_{j}},
\widehat{\Xi}^{\ast}_{\nu,\varepsilon_{j}}\right)=0 .
\end{displaymath}
By applying the Borel-Cantelli lemma once more,
we get that a.s.,
for every $j\geq 1$, there is $k\geq j$
such that for every $n\geq N_{k}$,
$\Xi^{\ast}_{\nu_{n},2^{-j}}\subset 
\widehat{\Xi}^{\ast}_{\nu_{n},\varepsilon_{k}}$.
This concludes the proof.
\end{proof}

We end this section with the following lemma which will be useful for the proof of Theorem~\ref{thm::cvg_envelop}. 
\begin{lemma}
\label{Lem Unif Loc aux 1}
Assume
$(\Xi_{\nu_{n}})_{n\geq 0}$ 
and $\Xi_{\nu}$
are coupled on the same probability space as in Proposition~\ref{Prop Coupling Exc}.
Then a.s. the family
\[(\Range(\gamma)\cup\partial\D)_{\gamma\in \Xi_{\nu_{n}}, n\geq 0}\]
is uniformly locally connected;
see Definition~\ref{Def Loc Connec}.
Furthermore, a.s. the family
\begin{displaymath}
(\Range(\gamma)
\cup\Range(\gamma')
\cup\partial\D)_{\gamma,\gamma'\in \Xi_{\nu_{n}}, n\geq 0}
\end{displaymath}
is uniformly locally connected, too.
\end{lemma}

\begin{proof}
We will only prove the first point. The proof of the second point is similar and we omit it.

First note that for any fixed $n\geq 0$, 
the family 
$(\Range(\gamma)\cup\partial\D)_{\gamma\in \Xi_{\nu_{n}}}$
is uniformly locally connected. 
Indeed, each $\Range(\gamma)\cup\partial\D$, is compact, connected, and locally connected due to Lemma \ref{Lem Loc Con Elem}.
Moreover, for every $\varepsilon>0$, 
there are only finitely many $\gamma \in \Xi_{\nu_{n}}$
such that  $\diam \Range (\gamma)\geq\varepsilon$.
So, if
$(\Range(\gamma)\cup\partial\D)_{\gamma\in \Xi_{\nu_{n}}, n\geq 0}$
were not uniformly locally connected,
one of the two cases would occur:
\begin{itemize}
\item Case~1:
there are $\varepsilon>0$,
a subsequence $(n_{j})_{j\geq 0}$,
with $n_{j}\to +\infty$ as $j\to +\infty$,
excursions $\gamma_{n_{j}}\in \Xi_{\nu_{n_{j}}}$,
and points 
$z_{n_{j}},z_{n_{j}}'\in\Range(\gamma_{n_{j}})$,
such that $\vert z_{n_{j}}'-z_{n_{j}}\vert\to 0$
and $z_{n_{j}}$ and $z_{n_{j}}'$ are not
$\varepsilon$-connected in $\Range(\gamma_{n_{j}})$.
\item Case~2:
there are $\varepsilon>0$,
a subsequence $(n_{j})_{j\geq 0}$,
with $n_{j}\to +\infty$ as $j\to +\infty$,
excursions $\gamma_{n_{j}}\in \Xi_{\nu_{n_{j}}}$,
and points 
$z_{n_{j}}\in\Range(\gamma_{n_{j}})$,
$z_{n_{j}}'\in\partial\D$,
such that $\vert z_{n_{j}}'-z_{n_{j}}\vert\to 0$
and $z_{n_{j}}$ and $z_{n_{j}}'$ are not
$\varepsilon$-connected in $\Range(\gamma_{n_{j}})$.
\end{itemize}
In both cases, necessarily
\begin{displaymath}
\inf_{j\geq 0} \diam \Range(\gamma_{n_{j}})
>0.
\end{displaymath} 
So, up to further extracting a subsequence, 
one can assume that
$\gamma_{n_{j}}$ converges to an excursion 
$\gamma_{\infty}\in\Xi_{\nu}$
for $d_{\rm curves}$,
and that $z_{n_{j}}$ and $z'_{n_{j}}$
converge to $z_{\infty}$.

In Case~1, $z_{\infty}\in \Range(\gamma_{\infty})$.
One can further distinguish the following subcases:
\begin{itemize}
\item Case~1a: $z_{\infty}\in \Range(\gamma_{\infty})\cap\D$.
\item Case~1b:
$z_{\infty}\in \Range(\gamma_{\infty})\cap\overline{A_{\rm L}}$.
\end{itemize}

In Case~1a, consider 
$\varepsilon'\in 
(0,\dist(z_{\infty},\overline{A_{\rm L}})\wedge 1)$.
For $j$ large enough, we have that
\begin{displaymath}
(\gamma_{n_{j}}(T_{\gamma_{n_{j}},\varepsilon'}^{\rm f}+t))
_{0\leq t\leq 
T_{\gamma_{n_{j}},\varepsilon'}^{\rm l}
-T_{\gamma_{n_{j}},\varepsilon'}^{\rm f}}
=
(\gamma_{\infty}(T_{\gamma_{\infty},\varepsilon'}^{\rm f}+t))
_{0\leq t\leq 
T_{\gamma_{\infty},\varepsilon'}^{\rm l}
-T_{\gamma_{\infty},\varepsilon'}^{\rm f}}
,
\end{displaymath}
and
\begin{displaymath}
z_{n_{j}}, z'_{n_{j}}\in
\{\gamma_{\infty}(T_{\gamma_{\infty},\varepsilon'}^{\rm f}+t)
: 0\leq t\leq 
T_{\gamma_{\infty},\varepsilon'}^{\rm l}
-T_{\gamma_{\infty},\varepsilon'}^{\rm f}\}
\subset \Range(\gamma_{n_{j}}).
\end{displaymath}
However, since the set
$\{\gamma_{\infty}(T_{\gamma_{\infty},\varepsilon'}^{\rm f}+t)
: 0\leq t\leq 
T_{\gamma_{\infty},\varepsilon'}^{\rm l}
-T_{\gamma_{\infty},\varepsilon'}^{\rm f}\}$
is locally connected,
we get a contradiction.
So Case~1a cannot occur.

In Case~1b, there is a sequence of positive times
$(t_{n_{j}})_{j\geq 0}$ converging to $0$ such that
for every $j\geq 0$,
\begin{displaymath}
z_{n_{j}}, z'_{n_{j}}\in
\{\gamma_{n_{j}}(t) : 0\leq t\leq t_{n_{j}}\}
\cup
\{\gamma_{n_{j}}(t) : 
T_{\gamma_{n_{j}}}-t_{n_{j}}\leq t\leq T_{\gamma_{n_{j}}}\}.
\end{displaymath}
We have that
\begin{displaymath}
\lim_{j\to +\infty}
\diam \{\gamma_{n_{j}}(t) : 0\leq t\leq t_{n_{j}}\}
= \lim_{j\to +\infty} \diam
\{\gamma_{n_{j}}(t) : 
T_{\gamma_{n_{j}}}-t_{n_{j}}\leq t\leq T_{\gamma_{n_{j}}}\} = 0,
\end{displaymath}
and that the set
$\{\gamma_{n_{j}}(t) : 0\leq t\leq t_{n_{j}}\}
\cup
\{\gamma_{n_{j}}(t) : 
T_{\gamma_{n_{j}}}-t_{n_{j}}\leq t\leq T_{\gamma_{n_{j}}}\}
\cup\partial\D$
is closed and connected. 
So we get that the family
\begin{displaymath}
(\{\gamma_{n_{j}}(t) : 0\leq t\leq t_{n_{j}}\}
\cup
\{\gamma_{n_{j}}(t) : 
T_{\gamma_{n_{j}}}-t_{n_{j}}\leq t\leq T_{\gamma_{n_{j}}}\}
\cup\partial\D)_{j\geq 0}
\end{displaymath}
is uniformly locally connected.
So Case~1b cannot occur.

In Case~2, again 
$z_{\infty}\in \Range(\gamma_{\infty})\cap\overline{A_{\rm L}}$.
So Case~2 can be ruled out by an argument very similar to that used for
Case~1b.
\end{proof}

\subsection{Continuous dependence of the curve $\eta_{\kappa,\nu}$ and proof of Theorem~\ref{thm::cvg_envelop}}
\label{SubSec Convergence curve}

In this section we deal with the dependence of the curve
$\eta_{\kappa,\nu}$ on the measure $\nu$.
Recall that
$\LO_{\kappa,\nu}$ is an open simply connected subset of $\D$,
and
$\partial\LO_{\kappa,\nu} = \eta_{\kappa,\nu}\cup A_{\rm R}$.
Recall that $\psi_{\kappa,\nu}$ is the conformal transformation from
$\D$ to $\LO_{\kappa,\nu}$ uniquely defined by the normalization 
$\psi_{\kappa, \nu}(-i)=-i$, $\psi_{\kappa, \nu}(1)=1$, 
$\psi_{\kappa, \nu}(i)=i$
and $\psi_{\kappa, \nu}(A_{\rm R})=A_{\rm R}$.
According to Theorem~\ref{Thm Cara boundary},
$\psi_{\kappa,\nu}$ extends continuously to $\overline{\D}$.
In case the curve $\eta_{\kappa,\nu}$ is simple
(see Proposition \ref{Prop double points}),
$\psi_{\kappa,\nu}$ induces a homeomorphism
from $\overline{\D}$ to $\overline{\LO_{\kappa,\nu}}$.
In general, $\psi_{\kappa,\nu}$ induces a homeomorphism
from $\D\cup\overline{A_{\rm R}}$ to
$\LO_{\kappa,\nu}\cup\overline{A_{\rm R}}$;
see e.g. \cite[Theorem~2.15]{Pommerenke}.
By construction,
$\eta_{\kappa,\nu}=\psi_{\kappa,\nu}(\overline{A_{\rm L}})$.
The goal of this section is to complete the proof of Theorem~\ref{thm::cvg_envelop}. 

We will restrict to the case $\kappa\neq 8/3$,
as the case $\kappa=8/3$ is simpler.
Note that all the probabilistic content of our proof is already
contained in Proposition~\ref{Prop Coupling Exc}.
We will additionally rely on deterministic geometrical arguments and
some a.s. properties of Brownian excursions and CLE. In the rest of this section, we fix the following assumptions: 
\textit{Fix $\kappa\in (8/3,4]$ and $\nu$ a finite non-negative Radon measure on $\overline{A_L}$.
Let $(\nu_{n})_{n\geq 0}$ be a sequence of finite non-negative Radon measures on $\overline{A_{\rm L}}$, 
converging weakly to $\nu$.
Assume
$(\Xi_{\nu_{n}})_{n\geq 0}$ 
and $\Xi_{\nu}$
are coupled on the same probability space as in Proposition~\ref{Prop Coupling Exc}.
Let $\FC_{\kappa}$ be a $\CLE_{\kappa}$ in $\D$
independent from 
$((\Xi_{\nu_{n}})_{n\geq 0},\Xi_{\nu})$.}

\begin{lemma}
\label{Lem Conv CLE}
Denote by $\widetilde{\FC}_{\kappa,\nu}$ the set 
constructed from
$\FC_{\kappa}$ and $\Xi_{\nu}$
as in \eqref{Eq tilde FC}.
Recall that 
$\LO_{\kappa,\nu,\varepsilon}$ denotes the connected component of
$\LO_{\kappa,\nu}\cap\D_{\varepsilon}$
(see \eqref{Eq D eps})
adjacent to $A_{\rm R}\cap\overline{\D_{\varepsilon}}$.
Define $\widetilde{\FC}_{\kappa,\nu_n}$ and $\LO_{\kappa,\nu_{n},\varepsilon}$ for $\FC_{\kappa}$ and $\Xi_{\nu_n}$ similarly. 
Then a.s., for every $\varepsilon\in (0,1)$,
there is $n'_{\varepsilon}\in\N$, 
such that, for every $n\geq n'_{\varepsilon}$, 
\begin{displaymath}
\LO_{\kappa,\nu_{n},\varepsilon}
=\LO_{\kappa,\nu,\varepsilon}.
\end{displaymath}
\end{lemma}

\begin{proof}
It is enough to check this for fixed $\varepsilon$, 
and then consider a sequence of
$\varepsilon$ converging to $0$.
Fix $\varepsilon\in (0,1)$.
The local finiteness of the CLE ensures that a.s. there is 
$\varepsilon'\in (0,\varepsilon)$
such that all the $\CLE_{\kappa}$ loops in $\FC_{\kappa}$
intersecting $\D_{\varepsilon}$ are at distance
greater than $\varepsilon'$ from $A_{\rm L}$.
The condition~\ref{item::couplingexc2} in Proposition~\ref{Prop Coupling Exc}
ensures that one can take 
$n'_{\varepsilon}=n_{\varepsilon'}$.
\end{proof}

\begin{lemma}
\label{Lem Conv Cara}
A.s., for every $w\in\LO_{\kappa,\nu}$,
the point $w$ belongs to $\LO_{\kappa,\nu_{n}}$
for every $n$ large enough and
$(\LO_{\kappa,\nu_{n}},w)$ converges to
$(\LO_{\kappa,\nu},w)$ in the Carathéodory sense as
$n\to +\infty$; see Definition~\ref{Def Cara}.
\end{lemma}

\begin{proof}
The condition~\ref{item::defCara1} in Definition~\ref{Def Cara} is automatic. 
We then check the condition~\ref{item::defCara2} in Definition~\ref{Def Cara}.
Given $z\in\overline{\LO_{\kappa,\nu}}$, 
let $I_{z}$ denote the straight line segment in $\overline{\D}$
with endpoints $1$ and
$\psi_{\kappa,\nu}^{-1}(z)$.
If $z$ is a multiple point on $\partial \LO_{\kappa,\nu}$,
then $\psi_{\kappa,\nu}^{-1}(z)$ will denote an arbitrary choice of a preimage of $z$.
Let $J_{z}$ denote $\psi_{\kappa,\nu}(I_{z})$.
It is a continuous curve in $\overline{\LO_{\kappa,\nu}}$
from $1$ to $z$.
If $z\in \LO_{\kappa,\nu}$, 
then $\dist(I_{z},A_{\rm L})>0$,
and thus
\begin{displaymath}
\dist(J_{z},A_{\rm L})
\geq
\dist(J_{z},\eta_{\kappa,\nu})>0.
\end{displaymath}
So, for $\varepsilon\in(0,\dist(J_{z},A_{\rm L})/2)$,
we have $J_{z}\subset\LO_{\kappa,\nu}\cap\D_{\varepsilon}$.
Then necessarily $J_{z}\subset\LO_{\kappa,\nu,\varepsilon}$.
According to Lemma~\ref{Lem Conv CLE},
for $n\geq n_{\varepsilon}'$,
we have $\LO_{\kappa,\nu,\varepsilon}=
\LO_{\kappa,\nu_{n},\varepsilon}$, 
and thus $\LO_{\kappa,\nu_{n},\varepsilon}$ is a neighborhood of
$z$ in $\LO_{\kappa,\nu_{n}}$.
So we get that for every $w\in\LO_{\kappa,\nu}$,
the point $w$ belongs to $\LO_{\kappa,\nu_{n}}$
for every $n$ large enough. This guarantees the condition~\ref{item::defCara2} in Definition~\ref{Def Cara}. 
It remains to check the condition~\ref{item::defCara3}.

So consider $z\in\partial\LO_{\kappa,\nu}$.
There are two cases, either $z\in\D$ or
$z\in\overline{A_{\rm L}}$.
In the first case, $z\in(\partial\LO_{\kappa,\nu})\cap\D$,
we still have
$\dist(J_{z},A_{\rm L})>0$.
For $\varepsilon\in(0,\dist(J_{z},A_{\rm L})/2)$,
one has that
$z\in(\partial\LO_{\kappa,\nu,\varepsilon})\cap\D_{\varepsilon}$.
It follows that for $n\geq n_{\varepsilon}'$,
$z\in\overline{\LO_{\kappa,\nu_{n},\varepsilon}}$
and also
$z\not\in\LO_{\kappa,\nu_{n}}$,
since otherwise one would have
$z\in\LO_{\kappa,\nu,\varepsilon}$.
Thus, for $n\geq n_{\varepsilon}'$,
$z\in\partial\LO_{\kappa,\nu_{n}}$.

Now consider the second case,
$z\in(\partial\LO_{\kappa,\nu})\cap\overline{A_{\rm L}}$.
For $j\geq 1$, let $\tilde{z}_{j}$ be the point
\begin{displaymath}
\tilde{z}_{j}:=
\psi_{\kappa,\nu}\left(
\left(1-\frac{1}{j}\right)\psi_{\kappa,\nu}^{-1}(z)
+
\dfrac{1}{j}
\right)
\in \LO_{\kappa,\nu}.
\end{displaymath}
We have that $\tilde{z}_{j}\to z$ as $j\to +\infty$.
Let $J_{z,j}$ be the part of the curve $J_{z}$
running from $1$ to $\tilde{z}_{j}$.
Then for every $j\geq 1$,
\begin{displaymath}
\dist(J_{z,j},A_{\rm L})
\geq
\dist(J_{z,j},\eta_{\kappa,\nu})>0.
\end{displaymath}
Thus, 
for $\varepsilon_{j}\in(0,\dist(J_{z,j},A_{\rm L})/2)$,
we have $\tilde{z}_{j}\in \LO_{\kappa,\nu,\varepsilon_{j}}$, 
and moreover, for $n\geq n_{\varepsilon_{j}}'$,
we have $\tilde{z}_{j}\in \LO_{\kappa,\nu_{n},\varepsilon_{j}}$.
Using a diagonal extraction, we get a subsequence
$(\tilde{z}_{j_{n}})_{n\geq n_{0}}$,
with for every $n\geq n_{0}$,
$\tilde{z}_{j_{n}}\in  \LO_{\kappa,\nu_{n}}$,
and
\begin{displaymath}
\lim_{n\to +\infty}\tilde{z}_{j_{n}}=z.
\end{displaymath}
Since $z\in \partial\D$, we have that
$z\not\in \LO_{\kappa,\nu_{n}}$.
Thus, the straight line segment from $\tilde{z}_{j_{n}}$
to $z$ contains a point $z_{n}\in\partial \LO_{\kappa,\nu_{n}}$.
Moreover, by construction,
$\vert z-z_{n}\vert<\vert z-\tilde{z}_{j_{n}}\vert$,
and so
$z_{n}\to z$ as $n\to +\infty$.
So one gets the condition~\ref{item::defCara3} of Definition~\ref{Def Cara}.
\end{proof}

\begin{lemma}
\label{Lem Unif Loc Connect}
A.s. the family 
$(\C\setminus\LO_{\kappa,\nu_{n}})_{n\geq 0}$
is uniformly locally connected;
see Definition \ref{Def Loc Connec}.
\end{lemma}

\begin{proof}
According to Lemma \ref{Lem Loc Con Boundary},
it is enough to check that the family
$(\LS_{\kappa,\nu_{n}}\cup\partial\D)_{n\geq 0}$ 
is uniformly locally connected.
If this is not the case, then at least one of the following happens:
\begin{itemize}
\item Case~1:
there are $\varepsilon>0$,
a subsequence $(n_{j})_{j\geq 0}$,
excursions 
$\gamma_{n_{j}}\in \Xi_{\nu_{n_{j}}}$,
and points 
$z_{n_{j}},z_{n_{j}}'\in\LS_{\kappa}(\gamma_{n_{j}})$
(see \eqref{Eq notations gamma})
such that $\vert z_{n_{j}}'-z_{n_{j}}\vert\to 0$
and $z_{n_{j}}$ and $z_{n_{j}}'$
are not $\varepsilon$-connected in 
$\LS_{\kappa,\nu_{n_{j}}}\cup\partial\D$.
\item Case~2:
there are $\varepsilon>0$,
a subsequence $(n_{j})_{j\geq 0}$,
excursions 
$\gamma_{n_{j}}\in \Xi_{\nu_{n_{j}}}$,
and points 
$z_{n_{j}}\in\LS_{\kappa}(\gamma_{n_{j}})$,
$z_{n_{j}}'\in\partial\D$,
such that $\vert z_{n_{j}}'-z_{n_{j}}\vert\to 0$
and $z_{n_{j}}$ and $z_{n_{j}}'$
are not $\varepsilon$-connected in 
$\LS_{\kappa,\nu_{n_{j}}}\cup\partial\D$.
\item Case~3:
there are $\varepsilon>0$,
a subsequence $(n_{j})_{j\geq 0}$,
excursions 
$\gamma_{n_{j}},\gamma'_{n_{j}}\in \Xi_{\nu_{n_{j}}}$,
$\gamma_{n_{j}}\neq\gamma'_{n_{j}}$,
and points 
$z_{n_{j}}\in\LS_{\kappa}(\gamma_{n_{j}})$,
$z_{n_{j}}'\in\LS_{\kappa}(\gamma'_{n_{j}})$,
such that $\vert z_{n_{j}}'-z_{n_{j}}\vert\to 0$
and $z_{n_{j}}$ and $z_{n_{j}}'$
are not $\varepsilon$-connected in 
$\LS_{\kappa,\nu_{n_{j}}}\cup\partial\D$.
\end{itemize}

\medskip

First consider Case~1.
One can further distinguish between the case
$\inf_{j\geq 0} \diam \Range(\gamma_{n_{j}})=0$
and the case
$\inf_{j\geq 0} \diam \Range(\gamma_{n_{j}})>0$.
By extracting sub-subsequences, one can thus reduce Case~1 to the following two subcases:
\begin{itemize}
\item Case~1a:
Case~1 with moreover $\lim_{j\to +\infty} 
\diam \Range(\gamma_{n_{j}})=0$.
\item Case~1b:
Case~1 with moreover
$\gamma_{n_{j}}$ converging to an excursion
$\gamma_{\infty}\in\Xi_{\nu}$
for $d_{\rm curves}$.
\end{itemize}

Regarding Case~1a, 
$\LS_{\kappa}(\gamma_{n_{j}})$ is a connected compact subset
connecting $z_{n_{j}}$ and $z_{n_{j}}'$.
Moreover, the fact that we use for every
$j\geq 0$ the same $\CLE_{\kappa}$,
together with the local finiteness of the $\CLE_{\kappa}$,
ensures that
\begin{equation}
\label{Eq diam S 0}
\lim_{j\to +\infty}
\diam(\LS_{\kappa}(\gamma_{n_{j}}))=0.
\end{equation}
So Case~1a cannot occur.

Regarding Case~1b, by considering sub-subsequences,
one can further reduce it to the following subcases:
\begin{itemize}
\item Case~1ba:
Case~1b with moreover
$z_{n_{j}},z_{n_{j}}'\in \Range(\gamma_{n_{j}})$.
\item Case~1bb:
Case~1b with moreover
$z_{n_{j}}'\in \Range(\gamma_{n_{j}})$
and $z_{n_{j}}\in \Range(\tilde{\gamma}_{n_{j}})$
for 
$\tilde{\gamma}_{n_{j}}\in
\widetilde{\FC}_{\kappa}(\gamma_{n_{j}})$
(see \eqref{Eq notations gamma}),
with $\diam \Range(\tilde{\gamma}_{n_{j}})\to 0$
as $j\to +\infty$.
\item Case~1bc:
Case~1b with moreover
$z_{n_{j}}'\in \Range(\gamma_{n_{j}})$
and $z_{n_{j}}\in \Range(\tilde{\gamma})$
for 
$\tilde{\gamma}\in \bigcap_{j\geq 0}
\widetilde{\FC}_{\kappa}(\gamma_{n_{j}})$.
\item Case~1bd:
Case~1b with moreover
$z_{n_{j}}\in \Range(\tilde{\gamma}_{n_{j}})$
and $z_{n_{j}}'\in \Range(\tilde{\gamma}'_{n_{j}})$
for 
$\tilde{\gamma}_{n_{j}},\tilde{\gamma}'_{n_{j}}\in
\widetilde{\FC}_{\kappa}(\gamma_{n_{j}})$,
with 
$\diam \Range(\tilde{\gamma}_{n_{j}})\to 0$
and
$\diam \Range(\tilde{\gamma}'_{n_{j}})\to 0$
as $j\to +\infty$.
\item Case~1be:
Case~1b with moreover
$z_{n_{j}}\in \Range(\tilde{\gamma}_{n_{j}})$
and $z_{n_{j}}'\in \Range(\tilde{\gamma}')$
for 
$\tilde{\gamma}_{n_{j}}\in
\widetilde{\FC}_{\kappa}(\gamma_{n_{j}})$
and
$\tilde{\gamma}'\in \bigcap_{j\geq 0}
\widetilde{\FC}_{\kappa}(\gamma_{n_{j}})$
,
with 
$\diam \Range(\tilde{\gamma}_{n_{j}})\to 0$
as $j\to +\infty$.
\item Case~1bf:
Case~1b with moreover
$z_{n_{j}},z_{n_{j}}'\in \Range(\tilde{\gamma})$
for 
$\tilde{\gamma}\in \bigcap_{j\geq 0}
\widetilde{\FC}_{\kappa}(\gamma_{n_{j}})$.
\end{itemize}

Lemma \ref{Lem Unif Loc aux 1} ensures that
Case~1ba cannot occur.

In Case~1bb, take points
$\tilde{z}_{n_{j}}\in
\Range(\tilde{\gamma}_{n_{j}})\cap\Range(\gamma_{n_{j}})$.
Then for $j$ large enough,
$\diam \Range(\tilde{\gamma}_{n_{j}})<\varepsilon/2$,
and thus
$\tilde{z}_{n_{j}}$ and $z_{n_{j}}'$
cannot be $\varepsilon/2$-connected in
$\LS_{\kappa,\nu_{n}}\cup\partial\D$.
So Case~1bb reduces to Case~1ba and cannot occur.

In Case~1bc, one can proceed similarly to the proof of
Lemma \ref{Lem Unif Loc aux 1}.
Indeed, away from $\overline{A_{\rm L}}$,
$\Range(\gamma_{n_{j}})$ coincides with
$\Range(\gamma_{\infty})$ for $j$ large enough,
and by Lemma \ref{Lem Loc Con Elem},
$\Range(\gamma_{\infty})\cup\Range(\tilde{\gamma})$
is locally connected.
This rules out Case~1bc.

Case~1bd reduces to Case~1bb by considering points
$\tilde{z}_{n_{j}}\in
\Range(\tilde{\gamma}_{n_{j}})\cap\Range(\gamma_{n_{j}})$.

Similarly, Case~1be reduces to Case~1bc.

Case~1bf cannot occur because $\Range(\tilde{\gamma})$
is locally connected.

\medskip

In Case~2 one can see that
$\dist(z_{n_{j}}',\Range(\gamma_{n_{j}})\cap\overline{A_{\rm L}})
\to 0$ as $j\to +\infty$. Thus, Case~2 reduces to Case~1.

\medskip

Case~3 can be reduced, by considering sub-subsequences,
to the following two subcases:
\begin{itemize}
\item Case~3a:
Case~3 with moreover $\lim_{j\to +\infty} 
\diam \Range(\gamma_{n_{j}})=0$.
\item Case~3b:
Case~3 with moreover
$\gamma_{n_{j}}$, respectively $\gamma_{n_{j}}'$,
converging to excursions $\gamma_{\infty}$,
respectively $\gamma_{\infty}'$ in $\Xi_{\nu}$
for $d_{\rm curves}$.
\end{itemize}

In Case~3a, since~\eqref{Eq diam S 0} holds, and in this way Case~3a reduces to Case~2.

Case~3b can be ruled out by arguments similar to those used for Case~1b.
\end{proof}

\begin{proof}[Proof of Theorem~\ref{thm::cvg_envelop}]
As mentioned, we deal only with $\kappa\in (8/3,4]$.
Assume
$(\Xi_{\nu_{n}})_{n\geq 0}$ 
and $\Xi_{\nu}$
are coupled on the same probability space as in Proposition~\ref{Prop Coupling Exc},
and that $\FC_{\kappa}$ is
sampled independent from 
$((\Xi_{\nu_{n}})_{n\geq 0},\Xi_{\nu})$.
We will deduce the a.s. convergence of
$((\psi_{\kappa,\nu_{n}}(x))_{x\in \overline{A_{\rm L}}})_{n \geq 0}$
to
$(\psi_{\kappa,\nu}(x))_{x\in \overline{A_{\rm L}}}$
in this coupling.

Take $w\in \LO_{\kappa,\nu}$.
According to Lemma \ref{Lem Conv Cara},
there is $n_{w}\geq 0$ such that
for $n\geq n_{w}$, we have $w\in \LO_{\kappa,\nu_{n}}$.
Denote by $\psi_{w}$
the conformal map
from $\D$ to $\LO_{\kappa,\nu}$ uniquely determined by the normalization:
$\psi_{w}(0)=w$ and $\psi_{w}'(0)>0$.
According to Theorem~\ref{Thm Cara boundary}, 
$\psi_{w}$
extends continuously
from $\overline{\D}$ to $\overline{\LO_{\kappa,\nu}}$. 
Define $\psi_{w,n}$ for $\LO_{\kappa,\nu_{n}}$ similarly. 
Since
$(\LO_{\kappa,\nu_{n}},w)$ converges to
$(\LO_{\kappa,\nu},w)$ in the Carathéodory sense
(Lemma~\ref{Lem Conv Cara}),
it follows that
$\psi_{w,n}$ converges to $\psi_{w}$ uniformly on compact
subsets of $\D$; 
see \cite[Theorem~1.8]{Pommerenke}.
Since the family
$(\C\setminus\LO_{\kappa,\nu_{n}})_{n\geq n_{w}}$
is uniformly locally connected
(Lemma~\ref{Lem Unif Loc Connect}),
$\psi_{w,n}$ converges to $\psi_{w}$ uniformly on
$\overline{\D}$;
see \cite[Corollary~2.4]{Pommerenke}.

Further, we write 
\begin{displaymath}
\psi_{\kappa,\nu}=
\psi_{w}\circ
\widetilde{\psi}_{w},
%\qquad
%\psi_{\kappa,\nu_{n}}=
%\psi_{w,n}\circ
%\widetilde{\psi}_{w,n},
\end{displaymath}
where $\tilde{\psi}_{w}$ is the Möbius transformation from
$\D$ to $\D$ uniquely determined by the normalization 
\begin{displaymath}
\widetilde{\psi}_{w}(-i)=\psi_{w}^{-1}(-i),
\qquad
\widetilde{\psi}_{w}(i)=\psi_{w}^{-1}(i),
\qquad
\widetilde{\psi}_{w}(1)=\psi_{w}^{-1}(1). 
\end{displaymath}
%\begin{displaymath}
%\widetilde{\psi}_{w,n}(-i)=\psi_{w,n}^{-1}(-i),
%\qquad
%\widetilde{\psi}_{w,n}(i)=\psi_{w,n}^{-1}(i),
%\qquad
%\widetilde{\psi}_{w,n}(1)=\psi_{w,n}^{-1}(1).
%\end{displaymath}
In case $-i$ or $i$
are not simple points
of $\eta_{\kappa,\nu}$,
%respectively $\eta_{\kappa,\nu_{n}}$,
the notions $\psi_{w}^{-1}(-i)$ and $\psi_{w}^{-1}(i)$
%respectively 
%$\psi_{w,n}^{-1}(-i)$ and $\psi_{w,n}^{-1}(i)$
are to be understood as
\begin{displaymath}
\psi_{w}^{-1}(-i)
=\lim_{\substack{\theta\to -\frac{\pi}{2}\\ \theta > -\frac{\pi}{2}}}
\psi_{w}^{-1}(e^{i\theta}),
\qquad
\psi_{w}^{-1}(i)
=\lim_{\substack{\theta\to \frac{\pi}{2}\\ \theta < \frac{\pi}{2}}}
\psi_{w}^{-1}(e^{i\theta}). 
\end{displaymath}
%\begin{displaymath}
%\psi_{w,n}^{-1}(-i)
%=\lim_{\substack{\theta\to -\frac{\pi}{2}\\ \theta > -\frac{\pi}{2}}}
%\psi_{w,n}^{-1}(e^{i\theta}),
%\qquad
%\psi_{w,n}^{-1}(i)
%=\lim_{\substack{\theta\to \frac{\pi}{2}\\ \theta < \frac{\pi}{2}}}
%\psi_{w,n}^{-1}(e^{i\theta}).
%\end{displaymath}
Define $\widetilde{\psi}_{w, n}$ for $\psi_{\kappa, \nu_n}$ and $\psi_{w, n}$ as $\psi_{\kappa, \nu_n}=\psi_{w, n}\circ\widetilde{\psi}_{w, n}$ similarly. 
Since $\psi_{w,n}^{-1}(-i)$, $\psi_{w,n}^{-1}(i)$,
respectively $\psi_{w,n}^{-1}(1)$,
converges to
$\psi_{w}^{-1}(-i)$, $\psi_{w}^{-1}(i)$,
respectively $\psi_{w}^{-1}(1)$,
we get that $\widetilde{\psi}_{w,n}$ converges to
$\widetilde{\psi}_{w}$ uniformly on $\overline{\D}$,
and $\psi_{\kappa,\nu_{n}}$ converges to
$\psi_{\kappa,\nu}$ uniformly on
$\overline{\D}$.
\end{proof}

\subsection{Continuous dependence of the driving functions and proof of Proposition~\ref{prop::cvg_envelop_drivingfunction}}
\label{subsec::cvg_drivingfunction}
Suppose $\eta_n$ is a continuous curve with continuous driving function $\xi^{(n)}$. In the literature, one is always interested in the following question: whether the convergence of driving function 
$\xi^{(n)}$ implies the convergence of curves $\eta_n$. See~\cite{SheffieldSunStrongPathCvg} and~\cite{KemppainenSmirnovRandomCurves}. In this section, we are interested in the question in reverse direction and the goal is to show Proposition~\ref{prop::cvg_envelop_drivingfunction}. We first give a general conclusion on convergence of driving functions out of convergence of curves: Proposition~\ref{prop::cvg_fromcurvetodrivingfunction}. Then Proposition~\ref{prop::cvg_envelop_drivingfunction} follows. Recall that $\psi_0$ is the conformal map from $\D$ to $\HH$ defined in~\eqref{eqn::psi0}. 

\begin{proposition}\label{prop::cvg_fromcurvetodrivingfunction}
Suppose $(\eta_n(t))_{0\le t\le 1}$ and $(\eta(t))_{0\le t\le 1}$ are parameterized continuous curves in $\overline{\D}$ from $-i$ to $+i$.  Assume the following hold.
\begin{enumerate}[label=(\alph*)]
\item \label{item::cvg_curvetodriving_a} The curves $((\eta_n(t))_{0\le t\le 1})_{n}$ converges to $(\eta(t))_{0\le t\le 1}$ for the uniform topology: 
\begin{equation}\label{eqn::uniform_certainparameter}
\|\eta_n-\eta\|_{\infty}:=\sup\{|\eta_n(t)-\eta(t)|: 0\le t\le 1\}\to 0,\quad\text{as }n\to\infty.
\end{equation}
\item The curves $\psi_0(\eta_n)$ and $\psi_0(\eta)$ satisfy assumptions in Proposition~\ref{prop::Loewnerchain_qualification}. 
\end{enumerate}
For each $t\in (0,1)$, let $g_t$ be the conformal map from the unbounded connected component of $\HH\setminus\psi_0(\eta[0,t])$ onto $\HH$ with normalization $\lim_{z\to\infty}|g_t(z)-z|=0$. 
Denote by $\xi_t=g_t(\psi_0(\eta(t)))$. Define $g^{(n)}_t$ and 
$\xi^{(n)}_{t}$ for $\psi_0(\eta_n)$ similarly. 
Then we have the following conclusions. 
\begin{enumerate}[label=(\arabic*)]
\item The half-plane capacity converges: for any $t\in (0,1)$, 
\begin{equation}\label{eqn::cvg_halfplanecapacity}
\sup\{|\hcap(\psi_0(\eta_n[0,s]))-\hcap(\psi_0(\eta[0,s]))|: 0\le s\le t\}\to 0, \quad \text{as }n\to\infty. 
\end{equation}
\item  When parameterized by the half-plane capacity, we have $\psi_0(\eta_n)\to \psi_0(\eta)$ and $\xi^{(n)}\to \xi$ for the local uniform topology. 
\end{enumerate}
\end{proposition}
\begin{proof}
For each $t\in (0,1)$, define $f_t=\psi_0^{-1}\circ g_t\circ\psi_0$. From Schwarz reflection principle, $f_t$ can be extended analytically in a neighborhood of $i$. Denote by $D_t$ the connected component of $\D\setminus\eta[0,t]$ with $i$ on the boundary. From elementary calculation, the function $f_t$ is the conformal map from $D_t$ onto $\D$ with the normalization 
$f_t(i)= i$, $f_t'(i)=1$, and $f_t''(i)= 0$.
Moreover, we have
\[\hcap(\psi_0(\eta[0,t]))=\frac{2}{3}f_t'''(i). \]
We define $f^{(n)}_t=\psi_0^{-1}\circ g^{(n)}_t\circ\psi_0$ for $\eta_n$ similarly. 

For each $t\in (0,1)$, we have $\eta_n[0,t]\to \eta[0,t]$ in Hausdorff metric. Consequently, $f^{(n)}_t\to f_t$ uniformly when bounded away from $\eta[0,t]$. Therefore, 
\[\hcap(\psi_0(\eta_n[0,t]))=\frac{2}{3}(f_t^{(n)})'''(i)\to \frac{2}{3}f_t'''(i)=\hcap(\psi_0(\eta[0,t])),\quad n\to\infty.\]
This gives the pointwise convergence of half-plane capacity. We will explain the uniform convergence below. 
Fix $t\in (0,1)$. From~\cite[Lemma~4.1]{LawlerConformallyInvariantProcesses}, one can show that, for any $0\le s_1<s_2\le t$, 
\begin{equation}\label{eqn::equicontinuity_Lawler}
\|f_{s_1}-f_{s_2}\|_{\infty}:=\sup\{|f_{s_1}(z)-f_{s_2}(z)|: z\in D_t\}\le C(\eta[0,t]) \sqrt{\osc(\eta, s_2-s_1, t)}, 
\end{equation}
where $\osc(\eta, \delta, t):=\sup\{|\eta(u)-\eta(v)|: 0\le u, v\le t, |u-v|\le\delta \}$ and $C(\eta[0,t])$ is a constant depending on the diameter of $\psi_0(\eta[0,t])$. Similarly, we have 
\[\|f^{(n)}_{s_1}-f^{(n)}_{s_2}\|_{\infty}\le C(\eta_n[0,t]) \sqrt{\osc(\eta_n, s_2-s_1, t)}. \]
From~\eqref{eqn::uniform_certainparameter}, we may choose $C_t^{\eta}$ large so that $C(\eta_n[0,t]), C(\eta[0,t])\le C_t^{\eta}$. 
Note that 
\[\osc(\eta_n, \delta, t)\le 2\|\eta_n-\eta\|_{\infty}+\osc(\eta, \delta, t). \]
Therefore, there exists a constant $C_t^{\eta}\in (0,\infty)$ depending on $\eta[0,t]$ such that 
\begin{equation}\label{eqn::equicontinuity}
\|f_{s_1}-f_{s_2}\|_{\infty}\le C_t^{\eta}\sqrt{\osc(\eta, s_2-s_1, t)},\quad \|f^{(n)}_{s_1}-f^{(n)}_{s_2}\|_{\infty}\le C_t^{\eta}\sqrt{2\|\eta_n-\eta\|_{\infty}+\osc(\eta, s_2-s_1, t)}. 
\end{equation}
Combining with~\eqref{eqn::uniform_certainparameter} and the pointwise convergence, we obtain  the uniform convergence of half-plane capacity~\eqref{eqn::cvg_halfplanecapacity}. 
As a consequence, we have $\psi_0(\eta_n)\to \psi_0(\eta)$ locally uniformly when parameterized by the half-plane capacity. 
It remains to show the convergence of $\xi^{(n)}$. 

Pick $x>0$ large enough so that it has positive distance to $\psi_0(\eta[0,t])$. For $t\in (0,1)$, define
\begin{equation}\label{eqn::def_renormalizedharm1}
h_t(x)=g_t(x)-\xi_t. 
\end{equation}
Define $h_t^{(n)}$ for $\eta_n$ similarly. 
From Lemma~\ref{lem::cvg_renormalizedharm}, 
we have $\sup\{|h_s^{(n)}(x)-h_s(x)|: 0\le s\le t\}\to 0$ as $n\to\infty$. Therefore,
\begin{equation}\label{eqn::cvg_renormalizedharm}
\sup\{|g_s^{(n)}(x)-\xi^{(n)}_{s}-g_s(x)+\xi_s|: 0\le s\le t\}\to 0,\quad \text{as }n\to\infty. 
\end{equation}
From~\eqref{eqn::equicontinuity} and pointwise convergence, we have $\sup\{|g_s^{(n)}(x)-g_s(x)|: 0\le s\le t\}\to 0$ as $n\to\infty$. 
Combining with~\eqref{eqn::cvg_renormalizedharm}, we have 
$\sup\{|\xi^{(n)}_{s}-\xi_s|: 0\le s\le t\}\to 0$ as $n\to\infty$. As the half-plane capacity also converge as in~\eqref{eqn::cvg_halfplanecapacity}, we have that $\xi^{(n)}\to \xi$ locally uniformly when parameterized by the half-plane capacity. 
\end{proof}

\begin{lemma}\label{lem::cvg_renormalizedharm}
Assume the same setup as in Proposition~\ref{prop::cvg_fromcurvetodrivingfunction}. Define $h_t(x)$ for $\eta$ and define $h_t^{(n)}(x)$ for $\eta_n$ as in~\eqref{eqn::def_renormalizedharm1}. Then we have 
\[\sup\{|h_s^{(n)}(x)-h_s(x)|: 0\le s\le t\}\to 0, \quad\text{as }n\to\infty. \]
\end{lemma}
\begin{proof}
The proof relies on a useful interpretation of the quantity $h_t(x)$. 
For $\eta[0,t]$, denote by $\LR(\psi_0(\eta[0,t]))$ the right-side of $\psi_0(\eta[0,t])$ and by $\LL(\psi_0(\eta[0,t]))$ the left-side of $\psi_0(\eta[0,t])$. 
Denote by $\LB=(\LB_t)_{t\ge 0}$ the Brownian motion in $\HH$ starting from $yi$ with $y>0$ large. Define $\tau$ to be the first time that it exits $\HH\setminus\psi_0(\eta[0,t])$.
From conformal invariance of Brownian motion, we have 
\begin{equation*}\label{eqn::def_renormalizedharm2}
h_t(x)=\lim_{y\to\infty}\pi y\PP^{yi}\left(\LB_{\tau}\in \LR(\psi_0(\eta[0,t]))\cup(0,x)\right). 
\end{equation*}
Similarly, define $\tau^{(n)}$ to be the first time that $\LB$ exits $\HH\setminus\psi_0(\eta_n[0,t])$. Then we have
\[h^{(n)}_t(x)=\lim_{y\to\infty}\pi y\PP^{yi}\left(\LB_{\tau^{(n)}}\in \LR(\psi_0(\eta_n[0,t]))\cup(0,x)\right). \]
For $\eps>0$, denote by $\LV_{\eps}(\psi_0(\eta[0,t]))$ the $\eps$-neighborhood of $\psi_0(\eta[0,t])$ and define $T^{\eps}$ to be the first time that $\LB$ hits $\LV_{\eps}(\psi_0(\eta[0,t]))$. Choose $n$ large enough so that 
\begin{equation}\label{eqn::rhm_aux}
\sup\{|\psi_0(\eta_n(s))-\psi_0(\eta(s))|: 0\le s\le t\}\le \eps/2.
\end{equation}
Then $\psi_0(\eta_n[0,t])$ is contained in $\LV_{\eps}(\psi_0(\eta[0,t]))$. Denote by 
\[V_t^R=\sup\{\psi_0(\eta[0,t])\cap \R\},\quad V_t^L=\inf\{\psi_0(\eta[0,t])\cap \R\}.\] 
If $\tau<T^{\eps}$, we have $\LB_{\tau^{(n)}}=\LB_{\tau}\in (V_t^R+\eps, x)$. 
Thus 
\begin{align}\label{eqn::rhm_aux1}
h_t(x)-h^{(n)}_t(x)=\lim_{y\to\infty}\pi y\PP^{yi}\left(R_t^{(n)}\right)-\lim_{y\to\infty}\pi y\PP^{yi}\left(S_t^{(n)}\right), 
\end{align}
where 
\begin{align*}
R_t^{(n)}=\left\{T^{\eps}\le \tau, \LB_{\tau}\in \LR(\psi_0(\eta[0,t]))\cup(V_t^R,V_t^R+\eps),  \LB_{\tau^{(n)}}\not\in \LR(\psi_0(\eta_n[0,t]))\cup(0,x)\right\}, \\
S_t^{(n)}=\left\{T^{\eps}\le \tau, \LB_{\tau_n}\in \LR(\psi_0(\eta_n[0,t]))\cup(V_t^R,V_t^R+\eps),  \LB_{\tau}\not\in \LR(\psi_0(\eta[0,t]))\cup(0,x)\right\}. 
\end{align*}

Let us first estimate the probability of $R_t^{(n)}$. For $\delta>0$, consider a path in $\HH\setminus\psi_0(\eta[0,t])$ starting from a point in 
\[\LR(\psi_0(\eta[0,t]))\setminus B(\psi_0(\eta(t)), \delta)\]
and terminating at a point in 
\[\LL(\psi_0(\eta[0,t]))\cup(-\infty, V_t^L)\cup (x,\infty).\]
Denote by $r_t(\delta)$ the infimum of the length of such paths. Recall that $\eta$ satisfies Proposition~\ref{prop::Loewnerchain_qualification}~\ref{item::Loewnerchain_qualifya}. We have $r_t(\delta)>0$ and $r_t(\delta)\to 0$ as $\delta\to 0$. 
From~\eqref{eqn::rhm_aux}, we see that any path in $\HH\setminus\psi_0(\eta[0,t])$ starting from a point in 
\[\LR(\psi_0(\eta[0,t]))\setminus B(\psi_0(\eta(t)), \delta)\] 
and terminating at a point in 
\[\LL(\psi_0(\eta_n[0,t]))\cup(-\infty, V_t^L)\cup (x,\infty)\]
has length at least $r_t(\delta)-\eps$. From Beurling estimate, there exists a universal constant $c\in (0,\infty)$ such that 
\begin{align*}
\PP^{yi}\left(R_t^{(n)}\right)\le &\PP^{yi}\left(R_t^{(n)}\cap \{\LB_{\tau}\not\in B(\psi_0(\eta(t)), \delta)\}\right)+\PP^{yi}\left(T^{\eps}\le \tau, \LB_{\tau}\in B(\psi_0(\eta(t)), \delta)\right)\\
\le & c\sqrt{\frac{\eps}{r_t(\delta)-\eps}}\PP^{yi}(T^{\eps}\le \tau)+\PP^{yi}\left(T^{\eps}\le \tau, \LB_{\tau}\in B(\psi_0(\eta(t)), \delta)\right). 
\end{align*}

Next, we estimate $\PP(S_t^{(n)})$ in a similar way. For $\delta>0$, consider a path in $\HH\setminus\psi_0(\eta[0,t])$ such that it starts from a point in 
\[\LL(\psi_0(\eta[0,t]))\cup(-\infty, V_t^L)\cup(x,\infty)\setminus B(\psi_0(\eta(t)), \delta)\]
and it terminates at a point in 
\[\LR(\psi_0(\eta[0,t])).\]
Denote by $s_t(\delta)$ the infimum of the length of such paths. Similarly, we have $s_t(\delta)>0$ and $s_t(\delta)\to 0$ as $\delta\to 0$, and 
\begin{align*}
\PP^{yi}\left(S_t^{(n)}\right)\le &\PP^{yi}\left(S_t^{(n)}\cap \{\LB_{\tau}\not\in B(\psi_0(\eta(t)), \delta)\}\right)+\PP^{yi}\left(T^{\eps}\le \tau, \LB_{\tau}\in B(\psi_0(\eta(t)), \delta)\right)\\
\le & c\sqrt{\frac{\eps}{s_t(\delta)-\eps}}\PP^{yi}(T^{\eps}\le \tau)+\PP^{yi}\left(T^{\eps}\le\tau, \LB_{\tau}\in B(\psi_0(\eta(t)), \delta)\right). 
\end{align*}

Plugging these into~\eqref{eqn::rhm_aux1}, we have 
\begin{align*}
|h_t(x)-h_t^{(n)}(x)|\le C_t^{\eta}(\eps, \delta)+F_t^{\eta}(\delta), 
\end{align*}
where $C_t^{\eta}(\eps, \delta)\to 0$ as $\eps\to 0$ and $F_t^{\eta}(\delta)\to 0$ as $\delta\to 0$. The same analysis applies for all $s\in [0,t]$. Thus, there exist $\tilde{C}_t^{\eta}(\eps, \delta)$ and $\tilde{F}_t^{\eta}(\delta)$ such that $\tilde{C}_t^{\eta}(\eps, \delta)\to 0$ as $\eps\to 0$ and $\tilde{F}_t^{\eta}(\delta)\to 0$ as $\delta\to 0$ and that 
\[\sup\{|h_s(x)-h_s^{(n)}(x)|: s\in [0,t]\}\le \tilde{C}_t^{\eta}(\eps, \delta)+\tilde{F}_t^{\eta}(\delta), \quad \text{as long as~\eqref{eqn::rhm_aux} holds.}\]
This gives the conclusion. 
\end{proof}

\begin{proof}[Proof of Proposition~\ref{prop::cvg_envelop_drivingfunction}]
The curves $\eta_{\kappa, \nu_n}$ and $\eta_{\kappa, \nu}$ satisfy the conditions in Proposition~\ref{prop::cvg_fromcurvetodrivingfunction}. 
\begin{itemize}
\item From Theorem~\ref{thm::cvg_envelop}, $(\eta_{\kappa, \nu_n})_{n\ge 0}$ and $\eta_{\kappa, \nu}$ satisfy Proposition~\ref{prop::cvg_fromcurvetodrivingfunction}~\ref{item::cvg_curvetodriving_a}. 
\item From the proof of Proposition~\ref{prop::envelop_drivingfunction}, $(\eta_{\kappa,\nu_n})_{n\ge 0}$ and $\eta_{\kappa,\nu}$ satisfy all the conditions in Proposition~\ref{prop::Loewnerchain_qualification}. 
\end{itemize}
Thus the conclusion follows. 
\end{proof}

\section{Identification with level lines of the GFF for $\kappa=4$}
\label{Sec Level lines}
\subsection{Proof of Theorem~\ref{thm::gfflevelline_coupling}}

In this section, we will construct a level line of GFF and complete the proof of Theorem~\ref{thm::gfflevelline_coupling}. Recall that $\Phi$ is a zero-boundary GFF in $\D$ and $\nu$ is a finite non-negative Radon measure on $\overline{A_{\rm L}}$ such that the support of $\nu$ equals $\overline{A_{\rm L}}$ and that $\pm i$ are not atoms of $\nu$. The goal of this section is to construct a level line of $\Phi+\nu$ as in Definition~\ref{def::GFFlevelline}. The strategy is to approximate such level line by the level line of GFF with piecewise constant boundary data which is already well understood. 

Let us introduce approximations of the measure $\nu$. For $n\ge 1$, we first decompose the arc $\overline{A_{\rm L}}$ into subarcs of length $\frac{\pi}{n+1}$: 
\[\theta^{(n)}_k=\frac{\pi}{2}+\frac{(n+1-k)\pi}{n+1},\quad k\in \{0, 1, \ldots, n+1\}. \]
Note that $\theta^{(n)}_{n+1}=\frac{\pi}{2}<\theta^{(n)}_{n}<\cdots<\theta^{(n)}_{1}<\theta^{(n)}_{0}=\frac{3\pi}{2}$. Denote by 
$A^{(n)}_k$ the subarc of $\overline{A_{\rm L}}$: 
\[A^{(n)}_k=\left\{e^{i\theta}: \theta^{(n)}_{k+1}<\theta\le \theta^{(n)}_{k}\right\}, \quad k\in \{0, 1, \ldots, n\}.\]
Define the measure $\nu_{n}$ on $\overline{A_{\rm L}}$:
\begin{equation}
\label{Eq nu n}
\nu_{n}=2\lambda\one_{A^{(n)}_0}\sigma_{\partial\D}+\sum_{k=1}^n \frac{\nu\left(A^{(n)}_k\right)}{\sigma_{\partial\D}\left(A^{(n)}_k\right)}\one_{A^{(n)}_k}\sigma_{\partial\D}.
\end{equation} 
Denote by $\eta^{(n)}=\eta_{4,\nu_{n}}$ the curve constructed from $\CLE_4$ and the Poisson point process of intensity 
$\mu^{\D}_{\nu_{n}}$. 
Then we have the following.
\begin{itemize}
\item From~\cite{AruLupuSepulvedaFPSGFFCVGISO}, $\eta^{(n)}$ can be coupled with $\Phi$ as the level line of $\Phi+\nu_{n}$.
\item From conclusions recalled in Section~\ref{subsec::levellines}, in the above coupling, $\eta^{(n)}$ is a.s. determined by $\Phi$. 
\end{itemize}
 In particular, the law of $\eta^{(n)}$ is the same as $\SLE_4(\rho^{(n)}_n, \ldots, \rho^{(n)}_1)$ in $\overline{\D}$ from $-i$ to $i$ with force points 
$(e^{i\theta^{(n)}_n}, \ldots, e^{i\theta^{(n)}_1})$.
Comparing with~\eqref{eqn::piecewisebc}, the parameters $(\rho^{(n)}_n, \ldots, \rho^{(n)}_1)$ are determined as follows: 
\begin{align*}
\bar{\rho}^{(n)}_k=\frac{1}{\lambda}\frac{\nu\left(A^{(n)}_k\right)}{\sigma_{\partial\D}\left(A^{(n)}_k\right)}-2, \quad k\in\{1, \ldots, n\}; \qquad \rho^{(n)}_k=\bar{\rho}^{(n)}_k-\bar{\rho}^{(n)}_{k-1}, \quad k\in\{1, \ldots, n\};
\end{align*}
with the convention that $\rho^{(n)}_0=0$. From Proposition~\ref{prop::cvg_envelop_drivingfunction}, the sequence $\eta^{(n)}$ converges to $\eta_{4,\nu}$. 

\begin{lemma}\label{lem::levelline_approxiamtion}
Suppose $\nu$ is a finite non-negative Radon measure on 
$\overline{A_{\rm L}}$ such that the support of $\nu$ equals $\overline{A_{\rm L}}$ and $\Atom_{\rm conv}^{\ast}(\nu)=\emptyset$.
Then the sequence $(\nu_{n})_n$ converges weakly to $\nu$. Consequently, when parameterized by the half-plane capacity, $\tilde{\eta}^{(n)}:=\psi_0(\eta^{(n)})$ converges in law to $\tilde{\eta}_{4,\nu}:=\psi_0(\eta_{4,\nu})$ and the driving function of $\tilde{\eta}^{(n)}$ converges in law to the driving function of $\tilde{\eta}_{4,\nu}$ for the local uniform topology. 
\end{lemma}

We already know that $\eta^{(n)}$ can be coupled with $\Phi$ as a level line of $\Phi+\nu_{n}$. 
The goal is to argue that $\eta$ can be coupled with $\Phi$ as a level line of $\Phi+\nu$. To this end, the following lemma plays an essential role. 

\begin{lemma}\label{lem::criteria_coupling}
Suppose $\eta$ is a continuous curve in $\overline{\D}$ from $-i$ to $i$ such that $\psi_0(\eta)$ has continuous driving function $\xi$.  Suppose $\Phi$ is zero-boundary GFF in $\D$. 
Then the pair $(\Phi, \eta)$ can be coupled such that $\eta$ is a level line of $\Phi+\nu$ as in Definition~\ref{def::GFFlevelline} if and only if for every choice of $z\in\D$, 
the process $(\nu_{t}(z))_{ t\ge 0}$ is a Brownian motion with respect to the filtration generated by $\xi$ when parameterized by minus log of the conformal radius: 
\[\log\CR(z, \D)-\log\CR(z, \D\setminus\eta[0,t]). \]
\end{lemma}

\begin{proof}
The proof follows~\cite[Section~2.2]{SchrammSheffieldContinuumGFF}, see also~\cite[Lemma~2.16]{PowellWuLevellinesGFF}. 
\end{proof}

\begin{proof}[Proof of Theorem~\ref{thm::gfflevelline_coupling}]
Suppose $\eta^{(n)}$ is parameterized by the half-plane capacity of $\psi_0(\eta^{(n)})$ and define $\nu_{n,t}$ as in Definition~\ref{def::GFFlevelline}. Applying Lemma~\ref{lem::criteria_coupling} to $\eta^{(n)}$, for any $z\in\D$, the process $(\nu_{n,t}(z))_{t\ge 0}$ is a Brownian motion when parameterized by 
\[C^{(n)}_t(z):=\log\CR(z, \D)-\log\CR(z, \D\setminus\eta^{(n)}[0,t]). \]

Suppose $\eta_{4,\nu}$ is parameterized by the half-plane capacity of $\psi_0(\eta_{4,\nu})$ and define $\nu_t$ as in Definition~\ref{def::GFFlevelline}. 
We wish to argue that $\eta_{4,\nu}$ can be coupled with $\Phi$ as a level line of $\Phi+\nu$. 
Fix an arbitrary $z\in\D$.
From Lemma~\ref{lem::criteria_coupling}, 
we need to argue that 
$(\nu_t(z))_{t\ge 0}$ is a Brownian motion when parameterized by 
\[C_t(z)=\log\CR(z, \D)-\log\CR(z, \D\setminus\eta[0,t]).\]
It suffices to show that $(\nu_{n,t}(z))_{t\ge 0}$ converges to 
$(\nu_t(z))_{t\ge 0}$ and $(C^{(n)}_t(z))_{t\ge 0}$ converges to 
$(C_t(z))_{t\ge 0}$ for the local uniform topology on processes parametrized by the time $t$. 
To this end, we use similar analysis as in Section~\ref{subsec::cvg_drivingfunction}. 

We parameterize $\psi_0(\eta_{4,\nu})$ by the half-plane capacity and denote by $g_t$ be the conformal map from the unbounded connected component of $\HH\setminus\psi_0(\eta_{4,\nu}[0,t])$ with normalization $\lim_{z\to\infty}|g_t(z)-z|=0$. Set $f_t=\psi_0^{-1}\circ g_t\circ\psi_0$. Define $g^{(n)}_t$ and $f^{(n)}_t$ for $\eta^{(n)}$ similarly.  

From the conformal covariance of the conformal radius, we have
\begin{align*}
C_t(z)=&\log|f_t'(z)|-\log\CR(f_t(z), \D)+\log\CR(z, \D);\\
C^{(n)}_t(z)=&\log|(f_t^{(n)})'(z)|-\log\CR(f_t^{(n)}(z), \D)+\log\CR(z, \D).
\end{align*}
From the proof of Proposition~\ref{prop::cvg_fromcurvetodrivingfunction}, $(f^{(n)}_t(z))_{t\ge 0}$ converges to $(f_t(z))_{t\ge 0}$ and $((f^{(n)}_t)'(z))_{t\ge 0}$ converges to $(f_t'(z))_{t\ge 0}$ for the local uniform topology, this implies that $(C^{(n)}_t(z))_{t\ge 0}$ converges to $(C_t(z))_{t\ge 0}$ for the local uniform topology.

We further claim that $(\nu_{n,t}(z))_{t\ge 0}$ converges to 
$(\nu_t(z))_{t\ge 0}$ for the local uniform topology.
The convergence of
$2\lambda H_{\D\setminus\eta^{(n)}[0,t]}(z, \LL(\eta^{(n)}[0,t]))$
towards
$2\lambda H_{\D\setminus\eta[0,t]}(z, \LL(\eta[0,t]))$,
locally uniformly in $t$, is similar to 
Lemma~\ref{lem::cvg_renormalizedharm}.
It remains to check the convergence of
\begin{equation}
\label{Eq I1}
\int_{\partial \D}H_{\D\setminus\eta^{(n)}[0,t]}(z,x)d\nu_{n}(x)
\end{equation}
towards
\begin{equation}
\label{Eq I2}
\int_{\partial \D}H_{\D\setminus\eta[0,t]}(z,x)d\nu(x).
\end{equation}
Without loss of generality, we assume that
$\eta$ and the $\eta^{(n)}$ are defined on the same probability space
such that the convergence of the $\eta^{(n)}$
towards $\eta$ is a.s. for the local uniform topology.
Denote by $A_{t}(z)$
the maximal open arc of $\partial\D$ that can be accessed
from $z$ without hitting $\eta[0,t]$.
Then the following holds a.s.:
\begin{itemize}
\item Both $H_{\D\setminus \eta^{(n)}[0,t]}(z,x)$
and $H_{\D\setminus \eta[0,t]}(z,x)$
are bounded from above by $H_{\D}(z,x)$.
\item $H_{\D\setminus \eta^{(n)}[0,t]}(z,x)$
converges to $H_{\D\setminus \eta[0,t]}(z,x)$
uniformly for $x$ belonging to compact
subsets of $A_{t}(z)$,
and locally uniformly in $t$.
\item $H_{\D\setminus \eta^{(n)}[0,t]}(z,x)$
converges to $0$
uniformly for $x$ belonging to compact
subsets of $\partial\D\setminus \overline{A_{t}(z)}$,
and locally uniformly in $t$.
\end{itemize}
Note that we do not claim that
$H_{\D\setminus \eta^{(n)}[0,t]}(z,x)$ converges to
$H_{\D\setminus \eta[0,t]}(z,x)$
for $x\in\partial A_{t}(z)$,
and the convergence of $\eta^{(n)}$ to $\eta$
is not sufficient to ensure that.
However, this is not needed, since by Lemma \ref{Lem hit},
a.s. for every $t$,
$\nu(\partial A_{t}(z))=0$.
So the three points listed above are sufficient to ensure the
convergence of \eqref{Eq I1} towards \eqref{Eq I2}.
\end{proof}

\subsection{Proof of Theorem~\ref{thm::determination}}

We may repeat the same argument in~\cite[Section~4]{PowellWuLevellinesGFF} and arrive at the following conclusion.
\begin{proposition}\label{prop::determination_PW}
Suppose $\Phi$ is zero-boundary GFF in $\D$.  Fix $\eps>0$. 
Suppose $\nu$ is a finite non-negative Radon measure on 
$\overline{A_{\rm L}}$ such that 
\begin{equation}\label{eqn::determination_assumptionPW}
\nu\ge \eps \one_{\overline{A_{\rm L}}}\sigma_{\partial\D}.
\end{equation}
Suppose that $\eta_{\nu}$ is a continuous simple curve in 
$\overline{\D}$ from $-i$ to $i$ with continuous driving function. 
Assume that $\eta_{\nu}$ is coupled with $\Phi$ as a level line of 
$\Phi+\nu$. Then the level line coupling is unique and $\eta_{\nu}$ is almost surely determined by $\Phi$. 
\end{proposition}

The proof of Proposition~\ref{prop::determination_PW} follows the same argument in~\cite[Section~4]{PowellWuLevellinesGFF}. To be self-contained, we briefly summarize the proof below. The proof relies on the following three lemmas~\ref{lem::PW_avoidanarc}, \ref{lem::PW_avoidapoint} and~\ref{lem::PW_mono}. 

\begin{figure}[ht!]
\begin{center}
\includegraphics[width=0.3\textwidth]{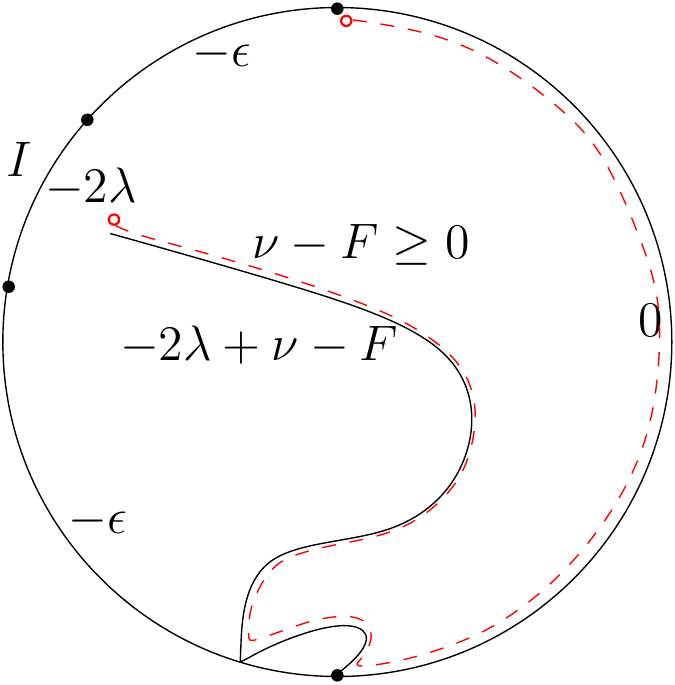}
\end{center}
\caption{\label{fig::PW_avoidanarc} Given $\eta_{\nu}[0,T_{\delta}]$, let $\tilde{\Phi}$ be $\Phi$ restricted to the connected component of $\D\setminus\eta_{\nu}[0,T_{\delta}]$ with $i$ on the boundary. Then the boundary data for $-\tilde{\Phi}-F$ is as follows: it equals $F$ on $\partial\D$, it equals $-2\lambda+\nu-F$ to the left side of $\eta_{\nu}[0,T_{\delta}]$, and it equals $\nu-F$ to the right side of $\eta_{\nu}[0,T_{\delta}]$. Note that $\nu-F\ge 0$. Thus $\eta'_F$ does not hit the union the union of the right side of $\eta_{\nu}[0,T_{\delta}]$ and $A_R$ (the dashed line marked in red). }
\end{figure}

\begin{lemma}\label{lem::PW_avoidanarc}
Assume the same assumptions as in Proposition~\ref{prop::determination_PW}. Suppose there is an open arc $I$ of $A_{\rm L}$ such that 
$\nu\ge 2\lambda\one_I\sigma_{\partial\D}$.
Then $\eta_{\nu}\cap I=\emptyset$ almost surely. 
\end{lemma}

\begin{proof}
We prove by contradiction. 
Suppose $\eta_{\nu}$ does hit $I$ with positive probability.
Then there exists an open arc $J\subset\overline{J}\subset I$ so that $\eta_{\nu}$ hits $J$ with positive probability. 
On this event, for $\delta>0$, define $T_{\delta}$ to the first time that $\eta_{\nu}$ gets within $\delta$ of $J$. Let $F$ be the bounded harmonic function in $\D$ with the following boundary data: it equals $2\lambda$ on $I$, it equals $\eps$ on $A_{\rm L}\setminus I$, and it equals zero on $A_{\rm R}$.  Then it is clear that $\nu(z)\ge F(z)$ for all $z\in\D$. Let $\eta'_F$ be the level line of $-\Phi-F$ in $\overline{\D}$ from $i$ to $-i$ and assume that the triple $(\Phi, \eta_{\nu}, \eta'_F)$ is coupled so that $\eta_{\nu}$ and $\eta'_F$ are conditionally independent given $\Phi$. Note that $F$ is piecewise constant on $\partial\D$ and that $\eta'_F$ does not hit $I$ almost surely.

For any $\delta>0$, given $\eta_{\nu}[0,T_{\delta}]$, let $\tilde{\Phi}$ be $\Phi$ restricted to the connected component of $\D\setminus\eta_{\nu}[0,T_{\delta}]$ with $i$ on the boundary. Then, given $\eta_{\nu}[0,T_{\delta}]$, the curve $\eta'_F$ is coupled with $-\tilde{\Phi}$ as the level line of $-\tilde{\Phi}-F$ whose boundary data is shown in Figure~\ref{fig::PW_avoidanarc} up to the first hitting time of $\eta_{\nu}[0,T_{\delta}]$.  Note that such boundary data is regulated. From~\cite[Lemma~4.1]{PowellWuLevellinesGFF}, the curve $\eta'_F$ does not hit the union of the right side of $\eta_{\nu}[0,T_{\delta}]$ and $A_R$. This implies that $\eta'_F$ has to get within $\delta$ of $J$. This holds for any $\delta>0$. Let $\delta\to 0$, it implies that $\eta'_F$ hits $\overline{J}\subset I$ with positive probability, contradiction. 
\end{proof}

\begin{lemma}\label{lem::PW_avoidapoint}
Assume the same assumptions as in Proposition~\ref{prop::determination_PW}. For any fixed $x_0\in A_{\rm L}$, the curve $\eta_{\nu}$ does not hit $\{x_0\}$ almost surely. 
\end{lemma}

\begin{proof}
Let $F$ be the bounded harmonic function in $\D$ with the following boundary data: it equals $\eps$ on $A_{\rm L}$ and it equals zero on $A_{\rm R}$. Then $\nu(z)\ge F(z)$ for all $z\in \D$. Let $\eta'_F$ be the level line of $-\Phi-F$ in $\overline{\D}$ from $i$ to $-i$ and assume that the triple $(\Phi, \eta_{\nu}, \eta'_F)$ is coupled so that $\eta_{\nu}$ and $\eta'_F$ are conditionally independent given $\Phi$. Note that $F$ is piecewise constant on $\partial\D$ and that $\eta'_F$ does not hit $\{x_0\}$ almost surely. If $\eta_{\nu}$ hits $\{x_0\}$ with positive probability, by the same argument as in the proof of Lemma~\ref{lem::PW_avoidanarc}, it would imply that $\eta'_F$ hits $\{x_0\}$ with positive probability, contradiction. 
\end{proof}

\begin{figure}[ht!]
\begin{center}
\includegraphics[width=0.3\textwidth]{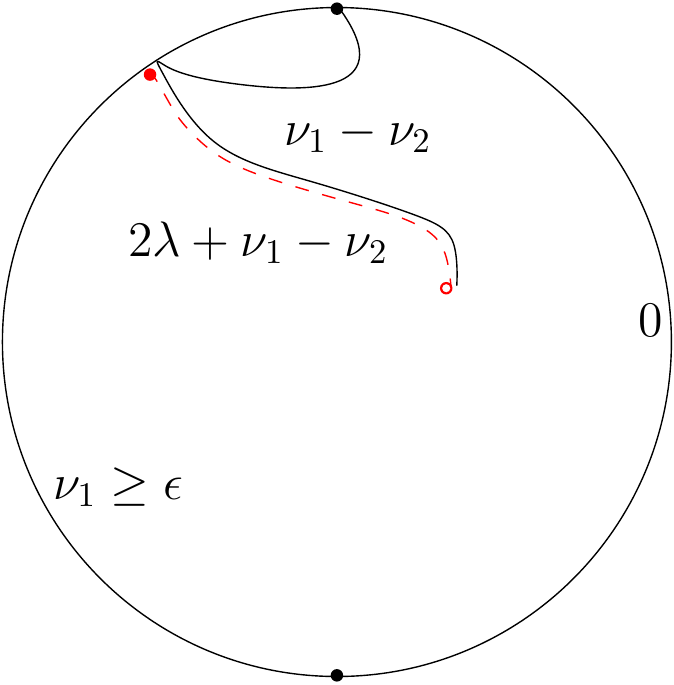}
\end{center}
\caption{\label{fig::PW_mono}  Given $\eta'_{\nu_2}[0,\tau']$, let $\tilde{\Phi}$ be $\Phi$ restricted to the connected component of $\D\setminus\eta'_{\nu_2}[0,\tau']$ with $-i$ on the boundary. Then the boundary data for $\tilde{\Phi}+\nu_1$ is as follows: it equals $\nu_1$ on $\partial\D$, it equals $2\lambda+\nu_1-\nu_2$ to the right side of $\eta'_{\nu_2}[0,\tau']$, and it equals $\nu_1-\nu_2$ to the left side of $\eta'_{\nu_2}[0,\tau']$. As $2\lambda+\nu_1-\nu_2\ge 2\lambda$, Lemma~\ref{lem::PW_avoidanarc} and absolute continuity imply that $\eta_{\nu_1}$ does not hit the right side of $\eta'_{\nu_2}[0,\tau']$. Let $x'$ be the last point of $\partial\D\cap\eta'_{\nu_2}[0,\tau']$ (the solid point marked in red). As the boundary data is greater than $\eps$ in neighborhood of $x'$, Lemma~\ref{lem::PW_avoidapoint} and absolute continuity imply that $\eta_{\nu_1}$ does not hit the point $x'$. In summary, $\eta_{\nu_1}$ does not hit the union of the right side of $\eta'_{\nu_2}[0,\tau']$ and the point $\{x'\}$ (the dash line marked in red). This implies that the point $\eta'_{\nu_2}(\tau')$ stays to the left of $\eta_{\nu_1}$.}
\end{figure}

\begin{lemma}\label{lem::PW_mono}
Fix $\eps>0$. 
Suppose $\nu_1, \nu_2$ are finite non-negative Radon measures on 
$\overline{A_{\rm L}}$ such that 
\begin{equation*}
\nu_1\ge \nu_2, \quad\text{and}\quad \nu_1\ge \eps \one_{\overline{A_{\rm L}}}\sigma_{\partial\D}.
\end{equation*}
Suppose that $\eta_{\nu_1}$ is a  continuous simple curve in $\overline{\D}$ from $-i$ to $i$ with continuous driving function and suppose $\eta'_{\nu_2}$ is a continuous simple curve in $\overline{\D}$ from $i$ to $-i$ with continuous driving function. Assume that $\eta_{\nu_1}$ is coupled with $\Phi$ as a level line of $\Phi+\nu_1$ from $-i$ to $i$ and $\eta'_{\nu_2}$ is coupled with $\Phi$ as a level line of $-\Phi-\nu_2$ from $i$ to $-i$ and that the triple $(\Phi, \eta_{\nu_1}, \eta'_{\nu_2})$ is coupled so that $\eta_{\nu_1}$ and $\eta'_{\nu_2}$ are conditionally independent given $\Phi$. Then $\eta_{\nu_1}$ stays to the right of $\eta'_{\nu_2}$ almost surely. 
\end{lemma}

\begin{proof}
As both $\eta_{\nu_1}$ and $\eta'_{\nu_2}$ are simple, it suffices to show that, for any $\eta'_{\nu_2}$-stopping time $\tau'$, the point $\eta'_{\nu_2}(\tau')$ is to the left of $\eta_{\nu_1}$. Given $\eta'_{\nu_2}[0,\tau']$, let $\tilde{\Phi}$ be $\Phi$ restricted to the connected component of $\D\setminus\eta'_{\nu_2}[0,\tau']$ with $-i$ on the boundary. Then, given $\eta'_{\nu_2}[0,\tau']$, the curve $\eta_{\nu_1}$ is coupled with $\tilde{\Phi}$ as a level line of $\tilde{\Phi}+\nu$ whose boundary data is shown in Figure~\ref{fig::PW_mono} up to the first hitting time of $\eta'_{\nu_2}[0,\tau']$. From Lemmas~\ref{lem::PW_avoidanarc} and~\ref{lem::PW_avoidapoint}, the curve $\eta_{\nu_1}$ does not hit the right side of $\eta'_{\nu_2}[0,\tau']$, see detail in Figure~\ref{fig::PW_mono}. This implies the point $\eta'_{\nu_2}(\tau')$ is to the left of $\eta_{\nu_1}$ as desired. 
\end{proof}

Now, we are ready to complete the proof of Proposition~\ref{prop::determination_PW}. 

\begin{proof}[Proof of Proposition~\ref{prop::determination_PW}]
Suppose $\eta_{\nu}'$ is a continuous simple curve in $\overline{\D}$ from $i$ to $-i$ with continuous driving function. Suppose that $\eta'_{\nu}$ is coupled with $\Phi$ as a level line of $-\Phi-\nu$ from $i$ to $-i$ and that  the triple $(\Phi, \eta_{\nu}, \eta'_{\nu})$ is coupled so that $\eta_{\nu}$ and $\eta'_{\nu}$ are conditionally independent given $\Phi$. From Lemma~\ref{lem::PW_mono}, we know that $\eta_{\nu}$ stays to the right of $\eta'_{\nu}$ and that $\eta_{\nu}$ stays to the left of $\eta'_{\nu}$ almost surely. As both of them are simple, we have $\eta_{\nu}=\eta'_{\nu}$ (viewed as sets) almost surely. Since $\eta_{\nu}$ and $\eta'_{\nu}$ are coupled with $\Phi$ so that they are conditionally independent given $\Phi$, the fact $\eta_{\nu}=\eta'_{\nu}$ implies that $\eta_{\nu}$ is almost surely determined by $\Phi$. 
\end{proof}

\begin{corollary}\label{cor::mono_PW}
Suppose $\Phi$ is zero-boundary GFF in $\D$. Fix $\eps>0$. 
Suppose $\nu_1, \nu_2$ are finite non-negative Radon measures on 
$\overline{A_{\rm L}}$ such that 
\begin{equation}\label{eqn::mono_assumptionPW}
\nu_1\ge \nu_2, \quad\text{and}\quad \nu_1\ge \eps \one_{\overline{A_{\rm L}}}\sigma_{\partial\D}.
\end{equation}
Suppose that $\eta_{\nu_1}, \eta_{\nu_2}$ are continuous simple curves in $\overline{\D}$ from $-i$ to $i$ with continuous driving functions. 
Assume that $\eta_{\nu_1}$ (resp. $\eta_{\nu_2}$) is coupled with $\Phi$ as a level line of $\Phi+\nu_1$ (resp. of $\Phi+\nu_2$), then $\eta_{\nu_1}$ stays to the right of $\eta_{\nu_2}$ almost surely. 
\end{corollary}

\begin{proof}
Suppose $\eta_{\nu_1}'$ is a continuous simple curve in $\overline{\D}$ from $i$ to $-i$ with continuous driving function. Suppose that $\eta'_{\nu_1}$ is coupled with $\Phi$ as a level line of $-\Phi-\nu_1$ from $i$ to $-i$ and that $(\Phi, \eta_{\nu_1}, \eta_{\nu_2}, \eta'_{\nu_1})$ is coupled so that $\eta_{\nu_1}, \eta_{\nu_2}$ and $\eta'_{\nu_1}$ are conditionally independent given $\Phi$. From Lemma~\ref{lem::PW_mono}, the curve $\eta'_{\nu_1}$ stays to the right of $\eta_{\nu_2}$ almost surely. From the proof of Proposition~\ref{prop::determination_PW}, we have $\eta_{\nu_1}=\eta'_{\nu_1}$ (viewed as sets) almost surely. These imply that $\eta_{\nu_1}$ stays to the right of $\eta_{\nu_2}$ almost surely as desired. 
\end{proof}

We emphasize that in the above proof of Proposition~\ref{prop::determination_PW}, we follow the method in~\cite{PowellWuLevellinesGFF} and the assumption~\eqref{eqn::determination_assumptionPW} plays an essential role. In the following, we will remove the assumption and complete the proof of Theorem~\ref{thm::determination}. 

\begin{proof}[Proof of Theorem \ref{thm::determination}]
Consider $(\Phi, \eta_{4,\nu})$ to be coupled as in Theorem~\ref{thm::gfflevelline_coupling}.
For $n\geq 1$, denote
$\nu_{n} := \nu +2^{-n}$.
Let $\eta_{\nu_{n}}$ be the level line of
$\Phi + \nu_{n}$ from $-i$ to $i$.
The existence of $\eta_{\nu_{n}}$ is ensured by
Theorem~\ref{thm::gfflevelline_coupling}.
Its uniqueness and measurability with respect to $\Phi$
is given by Proposition~\ref{prop::determination_PW}.
Moreover, for every $n\geq 1$,
$\eta_{\nu_{n}}$ is distributed as 
$\eta_{4,\nu_{n}}$ given by~\eqref{Eq eta kappa nu}.
However, the sequence $(\eta_{\nu_{n}})_{n\geq 1}$
is \textit{a priori} not coupled in the same way as the
sequence $(\eta_{4,\nu_{n}})_{n\geq 1}$
in Section~\ref{SubSec Convergence curve}
in the proof of Theorem~\ref{thm::cvg_envelop}.

According to Corollary~\ref{cor::mono_PW},
$\eta_{\nu_{n}}$ stays a.s. to the right of $\eta_{4,\nu}$
and to the right of $\eta_{\nu_{n + 1}}$,
for every $n\geq 1$.
Let $\LO_{4,\nu}$ denote the connected component of
$\D\setminus \eta_{4,\nu}$ to the right of $\eta_{4,\nu}$, 
and $\LO_{n}$ the connected component of
$\D\setminus \eta_{\nu_{n}}$ to the right of
$\eta_{\nu_{n}}$. 
A.s., $\LO_{n}\subset \LO_{n+1}$.
Denote
\begin{displaymath}
\LO_{\infty} :=\bigcup_{n\geq 1} \LO_{n}.
\end{displaymath}
We have that $\LO_{\infty}\subset \LO_{4,\nu}$ a.s.
Moreover, Theorem \ref{thm::cvg_envelop}
ensures that $\LO_{\infty}$ has the same distribution as
$\LO_{4,\nu}$.
This implies that $\LO_{\infty}=\LO_{4,\nu}$ a.s.
Since the sequence $(\eta_{\nu_{n}})_{n\geq 1}$
is measurable with respect to $\Phi$,
we get that $\LO_{4,\nu}$, 
and thus $\eta_{4,\nu}$,
are measurable with respect to $\Phi$.
\end{proof}

\subsection{An equation for the driving function}
\label{Subsec equation kappa 4}

Let $\nu$ be a finite non-negative Radon measure on 
$\overline{A_{\rm L}}$.
We assume that $\nu$ has full support 
on $\overline{A_{\rm L}}$.
We also assume that a.s., the curve
$\eta_{4,\nu}$ does not hit
$\Atom(\nu)$.
A sufficient condition for that is given by
Lemma \ref{Lem hit}.
Denote $\zeta_{\nu}$ the following Radon measure on
$\R$:
\begin{displaymath}
d \zeta_{\nu} (x) := 
\dfrac{1}{2} (1+x^{2})
d((\psi_{0})_{\ast}\nu)(x),
\end{displaymath}
where $\psi_{0}$ is the Möbius transformation from $\D$ 
to $\HH$ given by \eqref{eqn::psi0}.
$\zeta_{\nu}$ is a
non-negative Radon measure on $(-\infty, 0]$
satisfying
\begin{displaymath}
\int_{(-\infty, 0]}
\dfrac{1}{1+x^{2}}
d \zeta_{\nu}(x)
< +\infty .
\end{displaymath}
On $(0,+\infty)$, $\zeta_{\nu}$ equals $0$.
We see $\zeta_{\nu}$ as an analogue of the boundary 
condition \eqref{Eq bc}.
In particular, if $\nu$ is of the form
$\nu = a\one_{A_{\rm L}}\sigma_{\partial\D}$ with
$a>0$ a constant,
then $\zeta_{\nu}$ is a piecewise constant function,
equal to $a$ on
$(-\infty,0)$ and $0$ on $(0,+\infty)$.
Similarly to \eqref{Eq bc}, we define $\rho_{\nu}$ by
\begin{displaymath}
\rho_{\nu} :=
-\dfrac{1}{\lambda} \dfrac{d}{dx} \zeta_{\nu} - 2\delta_{0},
\end{displaymath}
where $\delta_{0}$ is the Dirac measure at $0$ and
where the derivative $\frac{d}{dx}$ is to be taken in the sense of
generalized function.
In general, $\rho_{\nu}$ is an order $1$
generalized function on 
$\R$ which is $0$ on $(0,+\infty)$.
Given $f\in \LC^{1}(\R)$ with compact support,
by integration by parts, we have that
\begin{equation}
\label{Eq IPP rho}
\int_{\R} f \rho_{\nu} dx
= -2 f(0)+
\dfrac{1}{\lambda}\int_{\R} f'(x) d \zeta_{\nu}(x).
\end{equation}

Consider now the curve $\eta_{4,\nu}$, 
parametrized by the half-plane capacity.
Denote by $g_t$ be the conformal map from the unbounded connected component of $\HH\setminus\psi_0(\eta_{4,\nu}[0,t])$ with normalization 
$\lim_{z\to\infty}|g_t(z)-z|=0$,
and $\xi_t$ the driving function in the corresponding Loewner chain.
Following \eqref{Eq SDE}, we are interested in giving a
meaning to
\begin{equation}
\label{Eq int rho nu}
\int_{(-\infty,0]} \dfrac{\rho_{\nu} dx}{\xi_t-g_{t}(x)} . 
\end{equation}
Denote
\begin{displaymath}
x_{{\rm L}}(t)
:=
\min\{x\in (-\infty,0] : x\in \psi_0(\eta_{4,\nu}[0,t])\} .
\end{displaymath}
Given \eqref{Eq IPP rho}, we set
\begin{equation}
\label{Eq drift}
\SZ_{t} :=
-\dfrac{2}{\xi_t -g_{t}(0^{-})}
+
\dfrac{1}{\lambda}
\int_{(-\infty,x_{{\rm L}}(t)]}
\dfrac{g'_{t}(x) d \zeta_{\nu}(x)}{(\xi_t-g_{t}(x))^{2}}
.
\end{equation}
If $\nu$ is actually a piecewise constant function on
$A_{\rm L}$, then $\SZ_{t}$ coincides with 
\eqref{Eq int rho nu}.
Denote $f_{t}=\psi_{0}^{-1}\circ g_{t}\circ \psi_{0}$,
and let $\SU_{t}\in \big(-\frac{3}{2}\pi,\frac{1}{2}\pi\big)$
such that 
$\psi_{0}(e^{i \SU_{t}}) = \xi_t$.
Then $\SZ_{t}$ can be expressed as
\begin{equation}
\label{Eq A t 2}
\SZ_{t} = 
\dfrac{(e^{i \SU_{t}}-i)(f_{t}((-i)^{-})-i)}
{e^{i \SU_{t}}-f_{t}((-i)^{-})}
-
\dfrac{i (e^{i\SU_{t}}-i)^{2}}{2\lambda}
\int_{\overline{A_{\rm L}(t)}} 
\dfrac{x f'_{t}(x)}{(e^{i\SU_{t}}-f_{t}(x))^{2}} 
d\nu (x)
,
\end{equation}
where
\begin{displaymath}
f_{t}((-i)^{-}) = 
\lim_{\substack{\theta\to -\frac{\pi}{2}\\ \theta < -\frac{\pi}{2}}}
f_{t}(e^{i\theta}),
\end{displaymath}
and $A_{\rm L}(t)$ is the connected component of
$A_{\rm L}\setminus \eta_{4,\nu}[0,t]$
adjacent to $i$.

%{\color{blue}
%\[
%\psi_0(z)=-i\frac{z+i}{z-i},\quad \psi_0'(z)=\frac{2}{(z-i)^2}, \quad \psi_0^{-1}(w)=\frac{w-i}{1-wi},\quad (\psi_0^{-1})'(w)=\frac{2}{(1-wi)^2}.\]
%\[
%\frac{1}{\xi_t-g_t(\psi_0(x))}=\frac{(e^{i\SU_t}-i)(f_t(x)-i)}{2(f_t(x)-e^{i\SU_t})}, \quad g_t'(\psi_0(x))=\frac{(x-i)^2}{(f_t(x)-i)^2}f_t'(x). \]
%Therefore,
%\[\frac{g_t'(\psi_0(x))}{(\xi_t-g_t(\psi_0(x)))^2}=\frac{(e^{i\SU_t}-i)^2(x-i)^2}{4(f_t(x)-e^{i\SU_t})^2}f_t'(x).\]
%Thus, 
%\[\frac{g_t'(\psi_0(x))d\zeta_{\nu}(\psi_0(x))}{(\xi_t-g_t(\psi_0(x)))^2}=\frac{(e^{i\SU_t}-i)^2(x-i)^2}{4(f_t(x)-e^{i\SU_t})^2}f_t'(x)\frac{-2xi}{(x-i)^2}d\nu(x)=\frac{(e^{i\SU_t}-i)^2(-2xi)}{4(f_t(x)-e^{i\SU_t})^2}f_t'(x)d\nu(x).\]
%}

\begin{proposition}
\label{Prop drift}
Let $\nu$ be a finite non-negative Radon measure on 
$\overline{A_{\rm L}}$
with full support and assume
$\eta_{4,\nu}$ is
parametrized by the half-plane capacity.
Also assume that a.s. the curve
$\eta_{4,\nu}$ does not hit
$\Atom(\nu)$.
\begin{enumerate}[label=(\arabic*)]
\item Let $\SZ_{t}$ be given by \eqref{Eq drift}.
Then $\SZ_{t}$ is well defined and continuous on the
subset of times
\begin{equation}
\label{Eq set away boundary}
I_{\nu}:=\{t\in [0,T_{\rm max}) : \eta_{4,\nu}(t)\not\in \overline{A_{\rm L}}\}.
\end{equation}
\item Let $(\nu_{n})_{n\geq 0}$ be a sequence of finite non-negative Radon measures on $\overline{A_{\rm L}}$
with full support, 
converging weakly to $\nu$. 
Assume that for every $n\geq 0$,
a.s. $\eta_{4,\nu_{n}}$ does not hit
$\Atom(\nu_{n})$.
Assume that each $\eta_{4,\nu_{n}}$
is parametrized by the half-plane capacity and that
$\eta_{4,\nu}$ and all the $\eta_{4,\nu_{n}}$
are coupled on the same probability space such that 
the sequence $(\eta_{4,\nu_{n}})_{n\geq 0}$
converges a.s. locally uniformly to $\eta_{4,\nu}$.
Let $\SZ^{(n)}_{t}$ be defined as $\SZ_{t}$,
but with $\nu_{n}$ and $\eta_{4,\nu_{n}}$
instead of $\nu$ and $\eta_{4,\nu}$.
Then, as $n\to +\infty$, $\SZ^{(n)}_{t}$ converges a.s.
to $\SZ_{t}$ uniformly on compact subsets of
$I_{\nu}$.
\end{enumerate}
\end{proposition}

\begin{proof}
We will use the expression \eqref{Eq A t 2}.
Observe that $A_{\rm L}(t)$
is constant on connected components of
$I_{\nu}$.
For the first point we use the following:
\begin{itemize}
\item $f_{t}$ is continuous
on $\overline{A_{\rm L}(t)}$,
locally uniformly in $t$ for $t\in I_{\nu}$;
see \eqref{eqn::equicontinuity_Lawler}.
\item $f'_{t}$ is bounded on $A_{\rm L}(t)$,
locally uniformly in $t$ for $t\in I_{\nu}$.
Indeed, $t\in [0,T_{\rm max})$
and $x\in A_{\rm L}(t)$,
\begin{displaymath}
\vert f'_{t}(x)\vert =
2\pi H_{\D\setminus\eta_{4,\nu_{n}}[0,t]}
((f_{t})^{-1}(0),x)
\leq
2\pi H_{\D}
((f_{t})^{-1}(0),x).
\end{displaymath}
\item $f'_{t}$ is continuous on compact subsets of $A_{\rm L}(t)$,
locally uniformly in $t$ for $t\in I_{\nu}$.
Indeed, one can use the Schwarz reflection principle
so as to analytically extend
$f_{t}$ across $A_{\rm L}(t)$, and then
Cauchy's integral formula to express
$f'_{t}$ through $f_{t}$.
\item For every $t\in I_{\nu}$,
$e^{i \SU_{t}}\not\in f_{t}(\overline{A_{\rm L}(t)})$.
\end{itemize}

Now let us check the second point.
Let $T^{(n)}_{\rm max}\in(0,+\infty]$ denote
the maximal parameter in the parametrization of
$\eta_{4,\nu_{n}}$ by half-plane capacity. Denote
\begin{displaymath}
I_{\nu_{n}}:=\{t\in [0,T_{\rm max}^{(n)}) : 
\eta_{4,\nu_{n}}(t)\not\in \overline{A_{\rm L}}\}.
\end{displaymath}
Denote $A_{\rm L}^{(n)}(t)$
the connected component of
$A_{\rm L}\setminus \eta_{4,\nu_{n}}[0,t]$
adjacent to $i$.
We will also use the notations
$\SU^{(n)}_t$ and $f^{(n)}_{t}$ in the case of
$\eta_{4,\nu_{n}}$, with straightforward meaning.
Every compact subset of $I_{\nu}$ is contained in
$I_{\nu_{n}}$ for $n$ large enough.
Moreover, for every $t\in I_{\nu}$,
\begin{displaymath}
A_{\rm L}(t)
\subset
\liminf_{n\to +\infty} A_{\rm L}^{(n)}(t).
\end{displaymath}
The equality does not hold in general.
The following holds.
\begin{itemize}
\item $\SU^{(n)}_t$ converges to
$\SU_t$, locally uniformly in $t$;
see Proposition \ref{prop::cvg_envelop_drivingfunction}.
\item For every $t\in I_{\nu}$
and $K$ compact subset of
$A_{\rm L}^{(n)}(t)\cup\{ i\}$,
$f^{(n)}_{t}$, respectively
$(f^{(n)}_{t})'$ converges
to $f_{t}$, respectively $f'_{t}$,
uniformly on $K$ and locally uniformly in $t$.
\item For every $n\geq 0$,
$t\in [0,T^{(n)}_{\rm max})$
and $x\in A_{\rm L}^{(n)}(t)$,
\begin{displaymath}
\vert (f^{(n)}_{t})'(x)\vert =
2\pi H_{\D\setminus\eta_{4,\nu_{n}}[0,t]}
((f^{(n)}_{t})^{-1}(0),x)
\leq
2\pi H_{\D}
((f^{(n)}_{t})^{-1}(0),x).
\end{displaymath}
In particular,
for every $t_{0}\in [0,T_{\rm max})$,
\begin{displaymath}
\limsup_{n\to +\infty}
\sup_{t\in[0,t_{0}]}
\sup_{x\in A_{\rm L}^{(n)}(t)}
\vert (f^{(n)}_{t})'(x)\vert
< +\infty.
\end{displaymath}
\item For every $t_{0}\in [0,T_{\rm max})$,
\begin{displaymath}
\lim_{n\to +\infty}
\sup_{t\in[0,t_{0}]}
\sup_{x\in A_{\rm L}^{(n)}(t)\setminus A_{\rm L}(t)}
\vert (f^{(n)}_{t})'(x)\vert
= 0.
\end{displaymath}
\end{itemize}
This implies the convergence.
\end{proof}

The following proposition tells that the driving function
$\xi_t$ satisfies the SDE
\begin{displaymath}
d \xi_t = 2 dB_{t} + \SZ_{t} dt
\end{displaymath}
on the set of times \eqref{Eq set away boundary},
where $(B_{t})_{t\geq 0}$ is a standard Brownian motion.

\begin{proposition}
\label{Prop equation}
Let $\nu$ be a finite non-negative Radon measure on 
$\overline{A_{\rm L}}$
with full support and assume
$\eta_{4,\nu}$ is
parametrized by the half-plane capacity.
Also assume that a.s. the curve
$\eta_{4,\nu}$ does not hit
$\Atom(\nu)$.
Let $(\LF_{t})_{t\geq 0}$
be the natural filtration of
$\eta_{4,\nu}[0,t\wedge T_{\max}]$.
Fix $t_{0}> 0$.
Let $E_{t_{0}}$ be the event defined by
$t_{0}<T_{\max}$ and by
$\eta_{4,\nu}(t_{0})\not\in\overline{A_{\rm L}}$.
Let be the stopping time
\begin{displaymath}
\tau(t_{0}):=
\sup \{ t\in [t_{0}, T_{\max}) :
\eta_{4,\nu}[t_{0},t] \cap \overline{A_{\rm L}}
=\emptyset
\}.
\end{displaymath}
Then, conditionally on the event $E_{t_{0}}$,
the stochastic process
\begin{displaymath}
\Big(\xi_{t\wedge \tau(t_{0})}
-\int_{t_{0}}^{t\wedge \tau(t_{0})}
\SZ_{s} ds
\Big)_{t\geq t_{0}}
\end{displaymath}
is a continuous martingale for the filtration
$(\LF_{t})_{t\geq t_{0}}$,
with quadratic variation given by
\begin{displaymath}
4(t\wedge \tau(t_{0}) - t_{0}).
\end{displaymath}
\end{proposition}

\begin{proof}
The result is true in $\nu$ is a piecewise constant function.
For general $\nu$, one takes an approximation of 
$\nu$ for the weak topology on measures by piecewise constant functions.
For instance, one can take \eqref{Eq nu n}.
Then the result follows by convergence,
by applying Proposition \ref{Prop drift}.
\end{proof}

\section{Some open questions}
\label{Sec questions}

Here we present some open questions related to this work:
\begin{enumerate}
\item In Proposition \ref{Prop double points}
we present a necessary and a sufficient condition for the
presence of double points in
$\eta_{\kappa,\nu}$, but
the two do not match. 
What is the optimal criterion for the presence of double points?
\item Similarly,
in Lemma \ref{Lem hit}
we give a necessary condition for $\eta_{\kappa,\nu}$
hitting an atom of $\nu$ with positive probability.
But what is the optimal criterion for this?
\item If $\nu$ is a Dirac measure at $-i$,
then $\eta_{\kappa,\nu}$ draws a bubble from
$-i$ to $-i$ in $\D$.
What is the distribution of this bubble?
We believe that it is singular to
the usual $\SLE_{\kappa}$ bubble measure
\cite[Section~4]{SheffieldWernerCLE}
because of the behaviour near $-i$.
\item If $\nu$ is a Dirac measure at $-i$ and $\kappa=4$,
what is the harmonic extension of
$\nu$ inside the bubble created by 
$\eta_{4,\nu}$ ?
Does an uniformizing map for this bubble
actually admit a derivative at $-i$ ?
\item In Proposition \ref{Prop equation}
we give an equation for the driving function
of $\eta_{4,\nu}$ when the curve is away from the boundary.
But what happens when the curve hits the boundary?
Is there an additional term
accounting for the interaction with the boundary?
This might be the case in some situations,
given that $\eta_{4,\nu}$
can actually intersect the boundary with a positive Lebesgue measure (Proposition \ref{Prop pos meas}).
\item What would be an equation for the driving
function of $\eta_{\kappa,\nu}$
for $\kappa\neq 4$?
This is not known even for $\nu$ being a piecewise constant function,
but it is known that the curve does not belong in general to the
$\SLE_{\kappa}(\rho)$ family.
\end{enumerate}

\appendix
\section{Appendix: Non-negative harmonic functions}
\label{appendix}
Here we recall some classical properties of non-negative harmonic functions.

\begin{propA}
Let $f$ be a non-negative harmonic function on the unit disk $\D$.
Then there is a finite non-negative Radon measure $\nu$ on
$\partial\D$, such that for every $z\in\D$,
\begin{equation}
\label{Eq pos harm}
f(z)=\int_{\partial\D} H_{\D}(z,x) d\nu(x),
\end{equation}
where $H_{\D}(z,x)$ is the Poisson kernel on $\D$
\eqref{Eq Poisson kernel}.
Moreover, the measure $\nu$ is unique.
\end{propA}

\begin{proof}
Let us first prove the existence of $\nu$.
For $\varepsilon\in (0,1)$, 
denote the following absolutely continuous measure on $\partial\D$:
\begin{displaymath}
d\nu_{\varepsilon}(x) :=
f((1-\varepsilon)x) 
\sigma_{\partial\D}(dx).
\end{displaymath}
For every $z\in\D$ with $\vert z\vert< 1-\varepsilon$,
we have that
\begin{displaymath}
f(z)=\int_{\partial \D}
H_{\D}((1-\varepsilon)^{-1}z,x) 
d\nu_{\varepsilon}(x),
\end{displaymath}
In particular, the total mass of $\nu_{\varepsilon}$ is
always $2\pi f(0)$.
Thus, the family
$(\nu_{\varepsilon})_{\varepsilon\in (0,1)}$
is relatively compact for the weak topology of measures,
and admits subsequential limits as $\varepsilon \to 0$.
Any such subsequential limit $\nu$ satisfies
\eqref{Eq pos harm}.

Now let us show the uniqueness.
Let $\nu$ be such that \eqref{Eq pos harm} is satisfied.
Let $u$ be a continuous function on $\partial\D$.
We have that
\begin{displaymath}
\int_{\partial\D} u(x) f((1-\varepsilon)x)
\sigma_{\partial\D}(dx)
=
\int_{\partial\D} d\nu(y)
\Big(
\int_{\partial \D}
u(x)H_{\D}((1-\varepsilon)x,y)\sigma_{\partial\D}(dx)
\Big) .
\end{displaymath}
The function
\begin{displaymath}
y\mapsto 
\int_{\partial \D}
u(x)H_{\D}((1-\varepsilon)x,y)\sigma_{\partial\D}(dx)
\end{displaymath}
converges uniformly to $u$ as $\varepsilon\to 0$.
Thus,
\begin{displaymath}
\lim_{\varepsilon\to 0}
\int_{\partial\D} u(x) f((1-\varepsilon)x)
\sigma_{\partial\D}(dx)
= \int_{\partial\D} u(y) d\nu(y).
\end{displaymath}
This characterizes $\nu$.
\end{proof}

\begin{corA}
\label{Cor pos harm arc}
A function 
$f: \D\rightarrow\R$ is of form
\begin{displaymath}
f(z)=\int_{\overline{A_{\rm L}}} H_{\D}(z,x) d\nu(x),
\qquad z\in \D ,
\end{displaymath}
where $\nu$ is
a finite non-negative Radon measure on $\overline{A_{\rm L}}$
if and only if
$f$ is non-negative harmonic on $\D$,
and for every $x\in A_{\rm R}$,
\begin{displaymath}
\lim_{\substack{z\to x\\ z\in\D}} f(z) = 0.
\end{displaymath}
\end{corA}

\end{document}